\documentclass[reqno]{amsart}

\usepackage{ifdraft}
\usepackage{latexsym,amsfonts,amssymb,exscale,enumerate,comment}
\usepackage{amsmath,amsthm,amscd}
\usepackage{mathrsfs}
\usepackage{comment}
\usepackage{xcolor}

\definecolor{answercolor}{RGB}{0, 112, 48}
\excludecomment{answer}      %<----- Comment this out to show details.
\newcommand{\wt}{\tilde{w}}
\newcommand{\com}{\color{green}  }

\usepackage{url}
\usepackage[bookmarks=true,%
    colorlinks=true,%
    linkcolor=blue,%
    citecolor=blue,%
    filecolor=blue,%
    menucolor=blue,%
    urlcolor=blue,%
    breaklinks=true]{hyperref}

\input xy
\usepackage[all]{xy}
\xyoption{line}
\xyoption{arrow}
\xyoption{color}
\SelectTips{cm}{}
\usepackage{tikz}
\usetikzlibrary{shapes,decorations}
\usetikzlibrary{decorations.markings}
\usetikzlibrary{decorations.pathreplacing}
\usetikzlibrary{cd}
\tikzstyle directed=[postaction={decorate,decoration={markings, mark=at position #1 with {\arrow[scale=1]{>}}}}]
\tikzstyle rdirected=[postaction={decorate,decoration={markings, mark=at position #1 with {\arrow[scale=1]{<}}}}]
\newcommand{\hackcenter}[1]{
 \xy (0,0)*{#1}; \endxy}

\tikzset{->-/.style={decoration={
  markings,
  mark=at position #1 with {\arrow{>}}},postaction={decorate}}}

\tikzset{middlearrow/.style={
        decoration={markings,
            mark= at position 0.5 with {\arrow{#1}} ,
        },
        postaction={decorate}
    }
}
\newcommand{\bbullet}{
\begin{tikzpicture}
  \draw[fill=black] circle (0.55ex);
\end{tikzpicture}
}
\usepackage{graphicx}
\usepackage{color}
\usetikzlibrary{shapes.geometric}

\def\bbi{\text{\boldmath$i$}}
\def\bbj{\text{\boldmath$j$}}
\def\bbk{\text{\boldmath$k$}}
\def\bbm{\text{\boldmath$m$}}
\def\bba{\text{\boldmath$a$}}
\def\bbep{\text{\boldmath$\varepsilon$}}
\def\bbde{\text{\boldmath$\delta$}}
\def\bbb{\text{\boldmath$b$}}
\def\bbc{\text{\boldmath$c$}}
\def\bbd{\text{\boldmath$d$}}

\newcommand{\tsigma}{\sigma}
\newcommand{\tomega}{\omega}
\newcommand{\tpsi}{\psi}

\newcommand{\Ucas}{\cal{C}}
\newcommand{\UupD}{\cal{E}}
\newcommand{\UdownD}{\cal{F}}
\newcommand{\ogam}{\xi^+}
\newcommand{\odelt}{\hat{\xi}^+}
\newcommand{\gam}{\xi^{-}}
\newcommand{\delt}{\hat{\xi}^{-}}
\newcommand{\gams}{\xi^{-}_{\tsigma}}
\newcommand{\delts}{\hat{\xi}^{-}_{\tsigma}}
\newcommand{\ogams}{\xi^+_{\tsigma}}
\newcommand{\odelts}{\hat{\xi}^+_{\tsigma}}
\newcommand{\gamp}{\xi'^{-}}

\newcommand{\SL}{\mathrm{SL}}

\def\W{W_{1+\infty}}
\def\tW{\widetilde{W}_{1+\infty}}
\def\id{\mathrm{id}}
\def\adj{\mathrm{adj}}

\usepackage{todonotes}
\usepackage{hyperref}
\newcommand{\ob}[1]{\mathsf{#1}}
\usepackage{enumitem}
\usepackage[capitalize]{cleveref}
\crefname{theorem}{Theorem}{Theorems}
\crefname{fact}{Fact}{Facts}
\crefname{note}{Note}{Notes}
\crefname{lemma}{Lemma}{Lemmas}
\crefname{alg}{Algorithm}{Algorithms}
\crefname{remark}{Remark}{Remarks}
\crefname{example}{Example}{Examples}
\crefname{prop}{Proposition}{Propositions}
\crefname{conj}{Conjecture}{Conjectures}
\crefname{cor}{Corollary}{Corollaries}
\crefname{definition}{Definition}{Definitions}
\crefname{Relation}{Relation}{Relations}
\crefname{equation}{\!\!}{\!\!} %Remove spacing around phantom equation name
\newcommand{\creflastconjunction}{, and\nobreakspace}

\newcommand{\END}{{\rm END}}
\newcommand{\Gr}{\cat{Flag}_{N}}
\newcommand{\Grn}[1]{\cat{Flag}_{#1}}
\theoremstyle{plain}
\newtheorem{theorem}{Theorem}[subsection]

\newtheorem{corollary}[theorem]{Corollary}
\newtheorem{proposition}[theorem]{Proposition}
\newtheorem{lemma}[theorem]{Lemma}

\theoremstyle{plain}
\newtheorem{example}[theorem]{Example}
\newtheorem{definition}[theorem]{Definition}
\newtheorem{conjecture}[theorem]{Conjecture}
\theoremstyle{definition}
\newtheorem{remark}[theorem]{Remark}

\newcommand{\sym}{{\rm Sym}}
\newcommand{\maps}{\colon}
\newcommand{\und}[1]{\underline{#1}}

\newcommand{\xsum}[2]{
  \xy
  (0,.4)*{\sum};
  (0,3.7)*{\scs #2};
  (0,-2.9)*{\scs #1};
  \endxy
}

\newcommand{\refequal}[1]{\xy {\ar@{=}^{#1}
(-1,0)*{};(1,0)*{}};
\endxy}
\newcommand{\cat}[1]{\ensuremath{\mbox{\bfseries {\upshape {#1}}}}}
\newcommand{\numroman}{\renewcommand{\labelenumi}{\roman{enumi})}}
\newcommand{\numarabic}{\renewcommand{\labelenumi}{\arabic{enumi})}}
\newcommand{\numAlph}{\renewcommand{\labelenumi}{\Alph{enumi}.}}

\newcommand{\To}{\Rightarrow}
\newcommand{\TO}{\Rrightarrow}
\newcommand{\Hom}{{\rm Hom}}
\newcommand{\HOM}{{\rm HOM}}
\renewcommand{\to}{\rightarrow}

\def\Res{{\mathrm{Res}}}
\def\Ind{{\mathrm{Ind}}}
\def\lra{{\longrightarrow}}
\def\dmod{{\mathrm{-mod}}}   %% finitely-generated modules
\def\fmod{{\mathrm{-fmod}}}   %% finite-dimensional modules
\def\pmod{{\mathrm{-pmod}}}  %% fin-gen projective modules
\def\rk{{\mathrm{rk}}}
\def\mc{\mathcal}
\def\mf{\mathfrak}

\def\Web{{\mathbf{Web}}}
\def\pWeb{\mathfrak{p}\textup{-}\mathbf{Web}}
\def\aWeb{\mathfrak{gl}\textup{-}\mathbf{Web}}

\numberwithin{equation}{section}
\def\TL#1{\textcolor[rgb]{1.00,1.00,1.00}{[TL: #1]}}%
\def\AL#1{\textcolor[rgb]{1.00,0.00,0.00}{[AL: #1]}}%
\def\PS#1{\textcolor[rgb]{0.00,0.49,0.25}{[PS: #1]}}%
\def\JS#1{\textcolor[rgb]{0.40,0.00,0.90}{[JS: #1]}}%
\def\SC#1{\textcolor[rgb]{1.00,0.50,0.00}{[SC: #1]}}%
\def\b{$\blacktriangleright$}
\def\e{$\blacktriangleleft$}
\def\new#1{\b #1\e}%

\let\hat=\widehat
\let\tilde=\widetilde

\let\epsilon=\varepsilon

\usepackage{bbm}
\def\C{{\mathbb{C}}}
\def\N{{\mathbbm N}}
\def\R{{\mathbbm R}}
\def\Z{{\mathbbm Z}}
\def\H{{\mathcal{H}}}
\def\G{{\mathcal{G}}}

\def\cal#1{\mathcal{#1}}%
\def\1{\mathbbm{1}}%
\def\ev{\mathrm{ev}}%
\def\coev{\mathrm{coev}}%
\def\tr{\mathrm{tr}}%
\def\gr{\mathrm{gr}}%
\def\st{\mathrm{st}}%
\def\nn{\notag}
\newcommand{\ontop}[2]{\genfrac{}{}{0pt}{2}{\scriptstyle #1}{\scriptstyle #2}}

\def\la{\langle}
\def\ra{\rangle}

\usepackage{bbm}

\def\cal#1{\mathcal{#1}}

\def \k {\mathbbm{k}}
\def \Z {\mathbbm{Z}}
\def \D {\mathbbm{D}}
\def \N {\mathbbm{N}}
\def \E {\mathcal{E}}
\def \F {\mathcal{F}}
\def \U {\mathcal{U}}
\def \T {\mathcal{T}}
\def \B {\mathcal{B}}
\def \Tr{\operatorname{Tr}}
\def \TTr{\operatorname{\mathbf{Tr}}}
\def \Span{\operatorname{Span}}
\def \Ob{\operatorname{Ob}}
\def \Set{\mathbf{Set}}
\def \Cat{\mathbf{Cat}}

\def \HH{\operatorname{HH}}
\def \min {{\rm min}}
\def \i {{\rm i}}

\renewcommand{\labelitemiii}{$\circ$}
\newcommand{\xto}[1]{{\overset{#1}{\longrightarrow}}}

\def\bH{\mathbb{H}}
\def\bP{\mathbb{P}}
\def\bQ{\mathbb{Q}}

\newcommand\nc{\newcommand}
\nc\rnc{\renewcommand}
\nc\Kar{\operatorname{Kar}}
\nc\modQ {{\mathbb Q}}
\nc\modZ {{\mathbb Z}}
\nc\simeqto{\overset{\simeq}{\longrightarrow }}
\nc\K{\mathcal {K}}
\nc\CC{\mathbf{C}}
\nc\qh{\mathcal{H}}
\nc\hbm{\mathcal{B}}
\nc\bu{\mathbf{u}}
\nc\bk{\mathbf{\kappa}}
\nc\bZ{\mathbf{Z}}
\nc\theirs{\mathrm{theirs}}
\nc\ours{\mathrm{ours}}
\nc\hE{\mathcal{\hat E}}
\nc\bK{\mathbf{K}}
\nc\bw{\mathbf{w}}
\nc\bh{\mathbf{h}}
\nc\tE{\tilde{E} }
\nc\ba{\mathbf{a}}
\nc\bb{\mathbf{b}}
\nc\tA{\tilde{A}}
\nc\tB{\tilde{B}}
\nc\tC{\tilde{C}}
\nc\ta{\tilde{a}}
\nc\tb{\tilde{b}}
\nc\tc{\tilde{c}}
\nc\te{\tilde{e}}
\nc\tf{\tilde{f}}
\def\th{\tilde{h}}
\nc\tH{\tilde{H}}
\nc\tF{\tilde{F}}

\nc{\urcap}{\textrm{uRCap}}
\nc{\urcup}{\textrm{uRCup}}
\nc{\ulcap}{\textrm{uLCap}}
\nc{\ulcup}{\textrm{uLCup}}
\nc{\drcap}{\textrm{dRCap}}
\nc{\drcup}{\textrm{dRCup}}
\nc{\dlcap}{\textrm{dRCap}}
\nc{\dlcup}{\textrm{dRCup}}
\newcommand{\dah}{{\rm DH}}
\newcommand{\ah}{{\rm AH}}
\nc{\cl}{\mathrm{cl}}
\nc{\cala}{\mathcal{A}}
\nc{\calb}{\mathcal{B}}
\nc{\calc}{\mathcal{C}}
\nc{\bx}{\mathbf{x}}
\nc{\by}{\mathbf{y}}
\nc{\bE}{\mathbbm{E}}
\nc{\bElong}{{\mathbbm E^{\bullet, \geq}}}
\nc{\Sk}{\mathrm{Sk}}
\nc{\Hilb}{\mathrm{Hilb}}

\newcommand{\scs}{\scriptstyle}
\newcommand{\eqr}[1]{{\overset{#1}{=}}}

\nc\col{\colon\thinspace}

\allowdisplaybreaks

\newcommand{\clr}{rgb:black,1;blue,4;red,1}

\newcommand{\p}[1]{|#1|}
\newcommand{\0}{\bar{0}}
\renewcommand{\1}{\bar{1}}

\newcommand{\ZZ}{\ensuremath{\mathbb{Z}}}
\newcommand{\fp}{\ensuremath{\mathfrak{p}}}
\newcommand{\fgl}{\mathfrak{gl}}
\newcommand{\fg}{\mathfrak{g}}
\newcommand{\fq}{\mathfrak{q}}
\newcommand{\fh}{\mathfrak{h}}
\newcommand{\fn}{\mathfrak{n}}
\newcommand{\fb}{\mathfrak{b}}

\newcommand{\pmodS}{\ensuremath{\fp(n)\textup{-mod}_{\mathcal{S}}}}
\newcommand{\pmodSS}{\ensuremath{\fp(n)\textup{-mod}_{\mathcal{S, S^{*}}}}}
\newcommand{\mods}{\mathfrak{p}(n)\text{-}\text{mod}_{\mathcal{S}}}
\newcommand{\modsupdown}{\mathfrak{p}(n)\text{-}\text{mod}_{\mathcal{S},\mathcal{S}^*}}

\newcommand{\Id}{\operatorname{Id}}
\newcommand{\mer}{\ensuremath{\operatorname{mer}}}
\newcommand{\spl}{\ensuremath{\operatorname{spl}}}

\newcommand{\Q}{\mathbb{Q}}
\newcommand{\End}{\operatorname{End}}
\newcommand{\pnWeb}{\ensuremath{\fp (n)\textup{-}\mathbf{Web}}}
\newcommand{\gl}{\mathfrak{gl}}
\newcommand{\tfp}{\ensuremath{\tilde{\mathfrak{p}}}}

\newcommand{\Udot}{\dot{U}}
\newcommand{\Udotm}{\ensuremath{\Udot(\tfp(m))}}
\newcommand{\bUdot}{\dot{\mathbf{U}}}
\newcommand{\bUdotm}{\ensuremath{\bUdot(\tfp(m))}}

\newcommand{\con}{\ensuremath{con}}
\renewcommand{\exp}{\ensuremath{exp}}

\title{Webs of Type P
}

\begin{document}
\setcounter{tocdepth}{2}

\author{Nicholas Davidson}
\email{davidsonnj@cofc.edu}
\address{Department of Mathematics\\ College of Charleston \\ Charleston, SC}

\author{Jonathan R. Kujawa}
\email{kujawa@ou.edu}
\address{Department of Mathematics\\ University of Oklahoma \\ Norman, OK}
\thanks{The second author was supported in part by Simons Collaboration Grant for Mathematicians No.\ 525043.}

\author{Robert Muth}
\email{muthr@duq.edu}
\address{ Department of Mathematics and Computer Science \\ Duquesne University \\ Pittsburgh, PA}

\date{\today}

\begin{abstract}  This paper introduces type P web supercategories. They are defined as diagrammatic monoidal $\k$-linear supercategories via generators and relations.  We study the structure of these categories and provide diagrammatic bases for their morphism spaces.  We also prove these supercategories provide combinatorial models for the monoidal supercategory generated by the symmetric powers of the natural module and their duals for the Lie superalgebra of type P.
\end{abstract}

\maketitle

\section{Introduction}

\subsection{Background}\label{SS:background}   This paper introduces certain diagrammatic supercategories via generators and relations.  These supercategories provide a combinatorial model of certain monoidal supercategories of representations for the type P Lie superalgebra $\fp (n)$.  The prefix ``super'' means there is a $\Z/2\Z$-grading and definitions include signs according to the grading.  For example, a supercategory is a category enriched in the category of $\Z_2$-graded vector spaces, while a monoidal supercategory is additionally equipped with a monoidal structure satisfying a graded analogue of the interchange law.  
%, and $\fp (n)$ is equipped with a $\Z_{2}$-graded and the bracket satisfies a graded version of the Jacobi identity.  
Recall that the type P Lie superalgebra is one of the so-called strange families which appears in Kac's classification of the simple complex Lie superalgebras \cite{Kac}.  It has no direct analogue in the classical world and its representation theory is still relatively mysterious.  One reason for this is that many classical techniques used to study Lie algebras cannot be directly adapted to the study of $\fp(n)$, e.g.  its enveloping superalgebra has trivial center, so the tools of central characters cannot be used.

In \cite{Moon} Moon gave a generators and relations presentation for the endomorphism algebras of the tensor powers of the natural supermodule for $\fp (n)$. Due to their similarity to Brauer's algebras for the orthogonal and symplectic Lie algebras they are variously called $X$ Brauer algebras where $X \in \left\{\text{marked}, \text{odd}, \text{periplectic} \right\}$.  Building on Moon's work, the second author and Tharp introduced a diagrammatic supercategory that describes the full sub-supercategory of $\fp (n)$-supermodules which are tensor products of the natural supermodule \cite{KT} (see also \cite{Serganova}).   The objects of this supercategory are nonnegative integers and the morphisms are $\k$-linear combinations of Brauer diagrams that are subject to signed versions of Brauer's original relations.  The endomorphism algebras in this supercategory give a diagrammatic realization of Moon's algebra.  Since then there has been substantial work applying diagrammatic, combinatorial, and categorical techniques to the study of $\fp(n)$ and Moon's algebra;   see \cite{WINART1,WINART2,BK,CCP,CP,Coulembier1,CEIII,CEII} and references therein.  The present paper further develops this approach to the representation theory of $\fp (n)$.  

Because we allow all (not just grading-preserving) homomorphisms, the category of finite-dimensional $\fp (n)$-supermodules is a supercategory.  As this supercategory is closed under taking tensor products and duals, this structure makes the supercategory of finite-dimensional $\fp (n)$-supermodules into a rigid monoidal supercategory in the sense of \cite{BE}.   In this paper we introduce and study diagrammatic supercategories which completely describe certain natural monoidal sub-supercategories of $\fp (n)$-supermodules. 

\subsection{Main Results on Webs}\label{SS:mainresults}  
For the discussion in this subsection we assume that $\k$ is an integral domain where two is invertible.  In \cref{S:pWebdef} we use generators and relations to define a $\k$-linear monoidal supercategory called $\pWeb$.  The objects of this supercategory are finite tuples of nonnegative integers.  Morphisms in this supercategory are $\k$-linear combinations of \emph{webs}, which are diagrams obtained by vertically and horizontally concatenating the generating diagrams (explained below in \cref{D:pWebdf}).   In this paper we use the convention that diagrams are read from bottom to top.  For example, given any integer $a > 1$, the following sum of webs is a morphism in $\pWeb$ from $(1,a,1)$ to $(a)$:
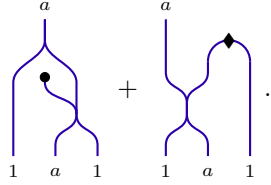
\[
\hackcenter{
\begin{tikzpicture}[scale=0.7]
  \draw[thick, color=\clr] (0.8, -1.5)--(0.8,-1.3)   .. controls ++(0,0.35) and ++(0,-0.35) ..    
  (0.4,-0.7)--(0.4,-0.1) .. controls ++(0,0.35) and ++(0,-0.35) .. (-0.2,0.7)--(-0.2,1.1);
  \draw[thick, color=\clr] (0, -1.5)--(0,-1.3)   .. controls ++(0,0.35) and ++(0,-0.35) ..    
  (0.4,-0.7)--(0.4,-0.6);
    \draw[thick, color=\clr] (-0.8,-1.5)--(-0.8,-0.1) .. controls ++(0,0.35) and ++(0,-0.35) ..
  (-0.2,0.7)--(-0.2,1.1); 
  \draw[thick, color=\clr] (0.4,-0.7) .. controls ++(0,0.35) and ++(0,-0.35) ..
  (-0.2,-0.1)--(-0.2,0); 
       %\node[left] at (0,0.1) {$\scriptstyle 3$};
       %\node at (0.4,0.3) {$\scriptstyle 2$};
       %\node[left] at (-0.4,-0.7) {$\scriptstyle 2$};
       \node[shape=coordinate](DOT) at (-0.2,0) {};
     \filldraw  (DOT) circle (2.5pt);
         \node[below] at (0.8,-1.5) {$\scriptstyle 1$};
      \node[below=1pt] at (0,-1.5) {$\scriptstyle a$};
       \node[below] at (-0.8,-1.5) {$\scriptstyle 1$};
        \node[above] at (-0.2,1.1) {$\scriptstyle a$};
\end{tikzpicture}}
+
\;
 \hackcenter{
\begin{tikzpicture}[scale=0.7]
   \draw[thick, color=\clr] (0,0)--(0,0.2) .. controls ++(0,0.35) and ++(0,-0.35)  .. (-0.4,0.8)--(-0.4,1)
    .. controls ++(0,0.35) and ++(0,-0.35)  .. (0,1.6)--(0,1.8);
   \draw[thick, color=\clr] (-0.8,0)--(-0.8,0.2) .. controls ++(0,0.35) and ++(0,-0.35)  .. (-0.4,0.8)--(-0.4,1)
   .. controls ++(0,0.35) and ++(0,-0.35)  .. (-0.8,1.6)--(-0.8,2.6);
   \draw[thick, color=\clr] (0.8,0)--(0.8,1.8);
   \draw (0.8,1.8) arc (0:180:0.4cm) [thick, color=\clr];
     \node[below] at (-0.8,0) {$\scriptstyle 1$};
      \node[below=1pt] at (0,0) {$\scriptstyle a$};
       \node[below] at (0.8,0) {$\scriptstyle 1$};
        \node[above] at (-0.8,2.6) {$\scriptstyle a$};
          \node at (0.4,2.2) {$\scriptstyle\blacklozenge$};
 \end{tikzpicture}}.
\]
Compared to webs which have previous appeared in the literature, experienced readers will notice that our webs contain (unoriented) cups and caps on strings of thickness one which are decorated by beads.  These are odd morphisms in the category, and correspond to the fact that the object 1 is self-dual.  The bead is used to distinguish these unoriented morphisms from the oriented cups and caps drawn in the oriented version of the web category (see below).  The one-valent vertex, called an \textit{antenna}, is a shorthand used to represent the composition of a beaded cap with the split.  See \cref{E:antenna}.

In \cref{S:Pwebupdowndefinition} we introduce an oriented version of $\pWeb$ which we call $\pWeb_{\uparrow \downarrow}$.   Again, this is a monoidal supercategory defined by generators and relations.  The objects of $\pWeb_{\uparrow \downarrow}$ are finite words in the symbols 
\[
\left\{\uparrow_{a}, \downarrow_{a} \mid a \in \Z_{\geq 0} \right\}.
\]  As before, morphisms are $\k$-linear combinations of diagrams obtained by vertical and horizontal concatenation of generating diagrams.  Any such diagram is called an \emph{oriented web}.  For example, given any integer $a \geq 1$, the following sum of oriented webs is a morphism in $\pWeb_{\uparrow \downarrow}$ from \(\uparrow_1 \, \uparrow_a \, \downarrow_1 \) to \(\uparrow_1 \, \uparrow_{a-1}\):
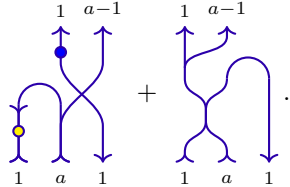
\[
\hackcenter{
\begin{tikzpicture}[scale=0.7]
    \draw[thick, color=\clr,<->] (0.8,-1.5)--(0.8,-0.8) .. controls ++(0,0.45) and ++(0,-0.45) ..
  (0,0.4)--(0,1.1); 
   \draw[thick, color=\clr,>-] (0, -1.5)--(0, -0.4);
    \draw[thick, color=\clr,>-<] (-0.8, -1.5)--(-0.8, -0.4);
      \draw[thick, color=\clr,>-] (-0.8, -1.5)--(-0.8, -0.4);
    \draw[thick, color=\clr,>->] (0,-1.5)--(0,-0.8) .. controls ++(0,0.45) and ++(0,-0.45) ..
  (0.8,0.4)--(0.8,1.1); 
     \draw (0,-0.4) arc (0:180:0.4cm) [thick, color=\clr];
         \node[below] at (-0.8,-1.5) {$\scriptstyle 1$};
      \node[below=1pt] at (0,-1.5) {$\scriptstyle a$};
       \node[below] at (0.8,-1.5) {$\scriptstyle 1$};
        \node[above] at (0.8,1.1) {$\scriptstyle a-1$};
 \node[shape=coordinate](DOT) at (-0.8,-0.9) {};
  \draw[thick, color=\clr, fill=yellow]  (DOT) circle (1mm);
       \node[above] at (0,1.1) {$\scriptstyle 1$};
 \node[shape=coordinate](DOT) at (0,0.6) {};
  \draw[thick, color=\clr, fill=blue]  (DOT) circle (1mm);
\end{tikzpicture}}
+
\;
 \hackcenter{
\begin{tikzpicture}[scale=0.7]
   \draw[thick, color=\clr,>-] (0,0)--(0,0.2) .. controls ++(0,0.35) and ++(0,-0.35)  .. (-0.4,0.8)--(-0.4,1)
    .. controls ++(0,0.35) and ++(0,-0.35)  .. (0,1.6)--(0,1.6);
   \draw[thick, color=\clr,>->] (-0.8,0)--(-0.8,0.2) .. controls ++(0,0.35) and ++(0,-0.35)  .. (-0.4,0.8)--(-0.4,1)
   .. controls ++(0,0.35) and ++(0,-0.35)  .. (-0.8,1.6)--(-0.8,2.6);
    \draw[thick, color=\clr,->] (-0.8,1.8) .. controls ++(0,0.35) and ++(0,-0.35)  .. (0,2.4)--(0,2.6);
   \draw[thick, color=\clr,<-] (0.8,0)--(0.8,1.6);
   \draw (0.8,1.6) arc (0:180:0.4cm) [thick, color=\clr];
     \node[below] at (-0.8,0) {$\scriptstyle 1$};
      \node[below=1pt] at (0,0) {$\scriptstyle a$};
       \node[below] at (0.8,0) {$\scriptstyle 1$};
        \node[above] at (-0.8,2.6) {$\scriptstyle 1$};
           \node[above] at (0,2.6) {$\scriptstyle a-1$};
 \end{tikzpicture}}.
\]
Experienced webslingers will also note that these webs are decorated with yellow and blue dots, which reverse the orientation of the strand.  These represent odd morphisms in this category.  The yellow dot encodes an isomorphism $\uparrow_1 \to \downarrow_1$, while the blue dot encodes its inverse.

Our first set of results concern the structure of these categories. In \cref{C:pwebbraiding,C:pwebupdownbraiding} we prove that both $\pWeb$ and $\pWeb_{\uparrow\downarrow}$ are symmetric braided categories and that $\pWeb_{\uparrow\downarrow}$ is rigid.  This is perhaps not surprising since these categories are constructed to provide diagrammatic models of categories of $\fp(n)$-supermodules which have these properties.   In \cref{BasisThm} we give a $\k$-linear `stable basis' for the morphism spaces in $\pWeb$ in terms of web diagrams (see \cref{SS:BasisForPWebs,SS:PuttingTogetherPWeb} for details).  By applying standard techniques (\cref{SS:IsomorphicHoms}) we extend our arguments to prove a basis theorem for the morphisms in $\pWeb_{\uparrow\downarrow}$.  

We also describe relationships among these categories that could be predicted by readers familiar with webs in other settings.  Let $\pWeb_{1}$ denote the full subcategory of $\pWeb$ consisting of all objects which are sequences of ones, and let $\pWeb_{\uparrow}$ denote the full subcategory of $\pWeb_{\uparrow\downarrow}$ consisting of all objects which are finite sequences of upward oriented arrows labeled by nonnegative integers.  In \cref{Bwebfun} we demonstrate $\pWeb_{1}$ is isomorphic to the marked Brauer supercategory introduced in \cite{KT}.  In \cref{T:pWebtopWebup} we prove that $\pWeb$ and $\pWeb_{\uparrow}$ are isomorphic monoidal supercategories.

\subsection{Main Results on Representations of \texorpdfstring{$\fp (n)$}{p(n)}}\label{SS:MainResults2}

Our second set of results require that the ground ring $\k$ is an algebraically closed field of characteristic zero.   The results explain how the categories $\pWeb$ and $\pWeb_{\uparrow\downarrow}$ are combinatorial models for certain natural subcategories of $\fp (n)$-supermodules.  

Let $V_{n}$ denote the natural $\fp (n)$-supermodule coming from its usual matrix representation, and for $a \geq 0$ let $S^a(V_n)$ and $\bigwedge^a(V_n)$ denote its $a$-th symmetric and exterior powers.   Write $\fp (n)\text{-mod}_{V}$ for the full subcategory of $\fp (n)$-modules consisting of all finite tensor powers of $V_n$, and let $\mods$ and $\modsupdown$ denote the full subcategory consisting of tensor products of  $S^{a}(V_{n})$ for various $a \geq 0$, and tensor products of $S^{a}(V_{n})$ and its dual $S^{a}(V_{n})^{*}$ for various $a \geq 0$, respectively.  We remark that since $S^{a}(V_{n})^{*} \cong \bigwedge^{a}(V_{n})$ for all $a \geq 1$, the category $\modsupdown$ also, up to isomorphism, includes exterior powers.

We can now describe the main results of the paper.  For each $n \geq 1$ we demonstrate that certain categories of $\fp(n)$-modules are equivalent to a quotients of the aforementioned web categories obtained by imposing one additional relation (which depends on $n$).  That is, the \emph{single} diagrammatic category $\pWeb$ can be used to describe the category $\mods$ for \textit{all} $n \geq 1$.  Similarly, $\pWeb_{1}$ describes $\fp (n)\text{-mod}_{V}$ and $\pWeb_{\uparrow\downarrow}$ describes $\modsupdown.$

To be precise, for every $n \geq 1$ we show in \cref{T:KTFunctor,Gthm,OrGthm} that there are essentially surjective functors of $\k$-linear, monoidal supercategories
\begin{align*}
F &: \pWeb_{1} \to \fp (n)\text{-mod}_{V}, \\
G &: \pWeb \to \mods, \\
G_{\uparrow\downarrow} &: \pWeb_{\uparrow\downarrow} \to \modsupdown.
\end{align*}  In \cref{T:KTFunctor,HomIsomPlus,OrHomIsomPlus} we show these functors are full.  It is worth noting that fullness can fail in positive characteristic.   See \cref{FailFull} for an example.

Next, using results of \cite{Coulembier1}, we define a certain morphism
\[
f_{n+1} \in \End_{\pWeb_{1}}\left([\ell], [\ell] \right) \cong \End_{\pWeb}\left(1^{\ell}, 1^{\ell} \right) \cong  \End_{\pWeb_{\uparrow \downarrow}}\left(\uparrow_{1}^{\ell}, \uparrow_{1}^{\ell}  \right),
\] where $\ell = (n+1)(n+2)/2$, and $1^{\ell}$ and $\uparrow_{1}^{\ell}$ denote an $\ell$-tuple of ones and $\uparrow_{1}$'s, respectively.   The definition of this morphism is subtle, and it does not seem to admit a nice diagrammatic description.  We define $\pnWeb$ to be the monoidal supercategory given by the same generators and relations as  $\pWeb$, along with the single additional relation 
\begin{equation}\label{E:fnzero}
f_{n+1}=0.  
\end{equation}
The monoidal supercategories $\pnWeb_{1}$ and $\pnWeb_{\uparrow \downarrow}$ are defined similarly.

  In \cref{T:CategoryEquivalences} it is shown that the functors $F$, $G$, and $G_{\uparrow \downarrow}$ induce equivalences of monoidal supercategories:
\begin{align*}
F &: \pnWeb_{1} \xrightarrow{\cong} \fp (n)\text{-mod}_{V}, \\
G &: \pnWeb \xrightarrow{\cong} \mods, \\
G_{\uparrow\downarrow} &: \pnWeb_{\uparrow\downarrow} \xrightarrow{\cong} \modsupdown.
\end{align*}
It is worth noting that these categories of $\fp (n)$-supermodules are not semisimple, unlike some of the more familiar contexts where web categories are used.

\subsection{Future Work}\label{SS:FutureWork}

Cautis--Kamnitzer--Morrison \cite{CKM} illustrated that Howe-type dualities give rise to web-like categories, but it is also sometimes the case that a Howe duality can be deduced from the existence of web-like categories (see \cite{QS,ST}).  In a sequel to this paper we use the results herein to construct a Howe duality between $\fp (m)$ and $\fp (n)$ \cite{DKM}.

In \cite{WINART2}, the authors introduce the  affine VW-supercategory. This can be regarded as an extension of $\pWeb_{1}$ given by including an additional even morphism $1 \to 1$ which defines a subalgebra of $\End_{\pWeb_1}(1^d)$ isomorphic to a polynomial ring in $d$ variables.   This diagrammatic supercategory admits a functor to the category of endofunctors of $\fp(n)\text{-mod}$ of the form $V^{\otimes d} \otimes -$ where the additional generator acts via a Casimir-like element. It would be interesting to define and study affine versions of $\pWeb$ and $\pWeb_{\uparrow \downarrow}$, as well.

In \cite{AGG} Ahmed--Grantcharov--Guay introduced a quantum superalgebra of type P via the FRT formalism.  As an outcome of the construction they obtain a Hopf superalgebra with a quantum analogue of the natural representation and an action of the braid group on its tensor powers.  We expect one can also define quantum analogues of Moon's algebra and the supercategories $\pWeb_{1}$, $\pWeb$, and $\pWeb_{\uparrow\downarrow}$.

\subsection{Conventions}\label{SS:conventions}

Throughout the paper, we will write $\k$ for our ground ring.  Our requirements for $\k$ will vary so we will endeavor to make clear what is assumed in each section.  At a minimum $\k$ will always be a commutative ring with identity.   \

We assume the reader is familiar with  monoidal categories, defining them by generators and relations, and in using diagrammatics to represent morphisms in such categories.  See \ref{SS:aWeb} for a brief discussion and \cite{Kas,Turaev} for further background.  To set our conventions, we read diagrams bottom to top with vertical concatenation corresponding to composition of morphisms.  Horizontal concatenation corresponds to the monoidal (or, tensor) product of morphisms.

This paper investigates mathematical objects in the ``super'' (i.e., $\Z_{2}=\Z/2\Z$-graded) setting.  To establish nomenclature, we say an element has \emph{parity} $r$ if it is homogenous and of degree $r \in \Z_{2}$.  We write $\p{w}$ for the parity of a homogeneous element, and we say that $w$ is \textit{even} (resp. \textit{odd}) if $\p{w} = \0$ (resp. $\p{w} = \1$).   We view $\k$ as a superalgebra concentrated in parity $\0$.

In particular, the context of this work is $\k$-linear monoidal \emph{super}categories.   As with $\k $-linear monoidal categories, they can be studied using a graphical calculus.  One difference is there is now a graded version of the interchange law.  Diagrammatically, this so-called super-interchange law introduces a sign whenever two odd morphisms are isotopied past each other in the vertical direction:
\begin{equation}\label{super-interchange}
%    \begin{tikzpicture}[baseline = 19pt,scale=0.5,color=\clr,inner sep=0pt, minimum width=11pt]
%        \draw[-,thick] (0,0) to (0,3);
%        \draw[-,thick] (2,0) to (2,3);
%        \draw (0,1.5) node[circle,draw,thick,fill=white]{$f$};
%        \draw (2,1.5) node[circle,draw,thick,fill=white]{$g$};
%    \end{tikzpicture} \quad := \quad  
    \begin{tikzpicture}[baseline = 19pt,scale=0.5,color=\clr,inner sep=0pt, minimum width=11pt]
        \draw[-,thick] (0,0) to (0,3);
        \draw[-,thick] (2,0) to (2,3);
        \draw (0,2) node[circle,draw,thick,fill=white]{$f$};
        \draw (2,1) node[circle,draw,thick,fill=white]{$g$};
    \end{tikzpicture}
    \quad ~= \quad (-1)^{\p{f}\p{g}}~
    \begin{tikzpicture}[baseline = 19pt,scale=0.5,color=\clr,inner sep=0pt, minimum width=11pt]
        \draw[-,thick] (0,0) to (0,3);
        \draw[-,thick] (2,0) to (2,3);
        \draw (0,1) node[circle,draw,thick,fill=white]{$f$};
        \draw (2,2) node[circle,draw,thick,fill=white]{$g$};
    \end{tikzpicture}
    ~.
\end{equation}
 Because of this, whenever two diagrams are horizontally concatenated the left diagram should be understood to be drawn above the right diagram:
 \begin{equation}\label{super-interchange}
    \begin{tikzpicture}[baseline = 19pt,scale=0.5,color=\clr,inner sep=0pt, minimum width=11pt]
        \draw[-,thick] (0,0) to (0,3);
        \draw[-,thick] (2,0) to (2,3);
        \draw (0,1.5) node[circle,draw,thick,fill=white]{$f$};
        \draw (2,1.5) node[circle,draw,thick,fill=white]{$g$};
    \end{tikzpicture} \quad := \quad  
    \begin{tikzpicture}[baseline = 19pt,scale=0.5,color=\clr,inner sep=0pt, minimum width=11pt]
        \draw[-,thick] (0,0) to (0,3);
        \draw[-,thick] (2,0) to (2,3);
        \draw (0,2) node[circle,draw,thick,fill=white]{$f$};
        \draw (2,1) node[circle,draw,thick,fill=white]{$g$};
    \end{tikzpicture}
%    \quad ~= \quad (-1)^{\p{f}\p{g}}~
%    \begin{tikzpicture}[baseline = 19pt,scale=0.5,color=\clr,inner sep=0pt, minimum width=11pt]
%        \draw[-,thick] (0,0) to (0,3);
%        \draw[-,thick] (2,0) to (2,3);
%        \draw (0,1) node[circle,draw,thick,fill=white]{$f$};
%        \draw (2,2) node[circle,draw,thick,fill=white]{$g$};
%    \end{tikzpicture}
    ~.
    \end{equation}
See \cite[Section 1]{BE} for details.   

In what follows, we assume all modules, categories, and functors are $\k$-linear.  We also assume that everything is $\Z_{2}$-graded and so will sometimes omit the prefix ``super''.

\subsection{ArXiv Version}\label{SS:ArxivVersion}  We chose to relegate a number of the more lengthy but straightforward calculations to the {\tt arXiv} version of the paper.  The reader interested in seeing these additional details can download the \LaTeX~source file from the \texttt{arXiv} and find a toggle near the beginning of the file which allows one to compile the paper with these calculations included.

\subsection{Acknowledgments}\label{SS:ArxivVersion} The authors thank the anonymous referee for their detailed and helpful comments.

\section{The  \texorpdfstring{$\aWeb$}{a-Web} Category}

\subsection{Definition of  \texorpdfstring{$\aWeb$}{a-Web}}\label{SS:aWeb}
Let \(\k\) be a commutative ring with identity. 
The definitions and results in this section are generally well-known and we record them for convenience. 

Here and below we will define combinatorial $\k$-linear strict monoidal (super)categories by generators and relations.  This method of construction is well-known and we only briefly describe how this works in our setting.  See, for example, \cite[Section XII.1]{Kas} or \cite[Section I.3]{Turaev} for a careful treatment in the classical case. The objects will be words from some set (e.g., $\Z_{\geq 1}$ or $\left\{\uparrow_{a}, \downarrow_{a} \mid a \in \Z_{\geq 1} \right\}$) with the monodial product given by concatenation of words.  This set of objects will evidently generate the set of all objects under the monoidal product and the empty word will be the monoidal unit object.

Morphisms will be given by providing a set of generating morphisms.  A general morphism will be constructed from these generators (and identity morphisms) by a finite sequence of compositions, monoidal products, and $\k$-linear combinations.  Since composition is given by vertical concatenation and the monoidal product is given by horizontal concatenation, a general morphism will be a $\k$-linear combination of diagrams with the same objects along the top and bottom, and where each diagrams was obtained by a finite sequence of vertical and horizontal concatenations of generating morphisms and identities.  To define the category we impose relations on the morphisms.  These relations are local in the sense that if two morphisms are identical other than in some small region where they differ by an imposed relation, then the morphisms are equal in the category.  Finally, in the cases when we have a supercategory the generating morphisms and defining relations will be homogenous in the $\Z_{2}$-grading and this will provide the grading on the morphism spaces.

%
%  For any $n \in \Z$ and $k \in \Z_{\geq 0}$, we use the generalized binomial coefficient
%\begin{align*}
%\binom{n}{k}:= \frac{n(n-1) \cdots (n-k+1)}{k!} \in \Z,
%\end{align*} viewed as an element of $\k$.
%
%
%

\begin{definition}
 Let \(\aWeb\) denote the strict monoidal \(\k\)-linear category given by generators and relations as follows. The objects are sequences of non-negative integers.  The morphisms are generated by the diagrams:
 
 \begin{align*}
\hackcenter{
{}
}
\hackcenter{
\begin{tikzpicture}[scale=0.8]
  \draw[thick, color=\clr] (0,0)--(0,0.2) .. controls ++(0,0.35) and ++(0,-0.35) .. (-0.4,0.9)--(-0.4,1);
  \draw[thick, color=\clr] (0,0)--(0,0.2) .. controls ++(0,0.35) and ++(0,-0.35) .. (0.4,0.9)--(0.4,1);
      \node[above] at (-0.4,1) {$\scriptstyle a$};
      \node[above] at (0.4,1) {$\scriptstyle b$};
      \node[above] at (0,-.5) {$\scriptstyle a+b$};
\end{tikzpicture}}  : a + b \to (a,b),
\qquad
\qquad
\hackcenter{
\begin{tikzpicture}[scale=0.8]
  \draw[thick, color=\clr] (-0.4,0)--(-0.4,0.1) .. controls ++(0,0.35) and ++(0,-0.35) .. (0,0.8)--(0,1);
\draw[thick, color=\clr] (0.4,0)--(0.4,0.1) .. controls ++(0,0.35) and ++(0,-0.35) .. (0,0.8)--(0,1);
      \node[above] at (-0.4,-0.47) {$\scriptstyle a$};
      \node[above] at (0.4,-0.47) {$\scriptstyle b$};
      \node[above] at (0,1) {$\scriptstyle a+b$};
\end{tikzpicture}}: (a,b) \to a + b,
\end{align*}
where $a,b \in \Z_{\geq 0}$.  We call these morphisms {\em split} and {\em merge}, respectively.  The identity morphism of the object $(a_{1}, \dotsc , a_{k})$ will be depicted by $k$ vertical strands labelled in order by $a_{1}, \dotsc , a_{k}$.

%Morphisms are generated by merges, splits, and identities by a finite sequence of compositions, monoidal products, and $\k$-linear combinations.  Thus a general morphism will be a $\k$-linear combination of diagrams with the same sequences along the top and bottom, and where each of the diagrams was obtained by a finite sequence of vertical and horizontal concatenations of splits, merges, and identities. ***yoyo

On the morphisms in $\aWeb$ we impose the following relations for all $a,b,c \in \Z_{\geq 0}$:\\

{ Web-associativity:}
\begin{align}\label{AssocRel}
\hackcenter{
{}
}
\hackcenter{
\begin{tikzpicture}[scale=0.8]
  \draw[thick, color=\clr] (0,0)--(0,0.2) .. controls ++(0,0.35) and ++(0,-0.35) .. (-0.4,0.9)--(-0.4,1) 
  .. controls ++(0,0.35) and ++(0,-0.35) .. (0,1.7)--(0,1.8); 
    \draw[thick, color=\clr] (0,0)--(0,0.2) .. controls ++(0,0.35) and ++(0,-0.35) .. (-0.4,0.9)--(-0.4,1) 
  .. controls ++(0,0.35) and ++(0,-0.35) .. (-0.8,1.7)--(-0.8,1.8); 
  \draw[thick, color=\clr] (0,0)--(0,0.2) .. controls ++(0,0.5) and ++(0,-0.5) .. (0.8,1.5)--(0.8,1.8);
      \node[above] at (-0.8,1.8) {$\scriptstyle a$};
      \node[above] at (0,1.8) {$\scriptstyle b$};
      \node[above] at (0.8,1.8) {$\scriptstyle c$};
      \node[above] at (0,-.5) {$\scriptstyle a+b+c$};
      \node[left] at (-0.4,0.8) {$\scriptstyle a+b$};
\end{tikzpicture}}
=
\hackcenter{
\begin{tikzpicture}[scale=0.8]
  \draw[thick, color=\clr] (0,0)--(0,0.2) .. controls ++(0,0.35) and ++(0,-0.35) .. (0.4,0.9)--(0.4,1) 
  .. controls ++(0,0.35) and ++(0,-0.35) .. (0,1.7)--(0,1.8); 
    \draw[thick, color=\clr] (0,0)--(0,0.2) .. controls ++(0,0.35) and ++(0,-0.35) .. (0.4,0.9)--(0.4,1) 
  .. controls ++(0,0.35) and ++(0,-0.35) .. (0.8,1.7)--(0.8,1.8); 
  \draw[thick, color=\clr] (0,0)--(0,0.2) .. controls ++(0,0.5) and ++(0,-0.5) .. (-0.8,1.5)--(-0.8,1.8);
      \node[above] at (-0.8,1.8) {$\scriptstyle a$};
      \node[above] at (0,1.8) {$\scriptstyle b$};
      \node[above] at (0.8,1.8) {$\scriptstyle c$};
      \node[above] at (0,-.5) {$\scriptstyle a+b+c$};
      \node[right] at (0.4,0.8) {$\scriptstyle b+c$};
\end{tikzpicture}},
\qquad
\qquad
\hackcenter{
\begin{tikzpicture}[scale=0.8]
  \draw[thick, color=\clr] (0,0)--(0,-0.2) .. controls ++(0,-0.35) and ++(0,0.35) .. (-0.4,-0.9)--(-0.4,-1) 
  .. controls ++(0,-0.35) and ++(0,0.35) .. (0,-1.7)--(0,-1.8); 
    \draw[thick, color=\clr] (0,0)--(0,-0.2) .. controls ++(0,-0.35) and ++(0,0.35) .. (-0.4,-0.9)--(-0.4,-1) 
  .. controls ++(0,-0.35) and ++(0,0.35) .. (-0.8,-1.7)--(-0.8,-1.8); 
  \draw[thick, color=\clr] (0,0)--(0,-0.2) .. controls ++(0,-0.5) and ++(0,0.5) .. (0.8,-1.5)--(0.8,-1.8);
      \node[above] at (-0.8,-2.25) {$\scriptstyle a$};
      \node[above] at (0,-2.25) {$\scriptstyle b$};
      \node[above] at (0.8,-2.25) {$\scriptstyle c$};
      \node[above] at (0,0) {$\scriptstyle a+b+c$};
      \node[left] at (-0.4,-0.8) {$\scriptstyle a+b$};
\end{tikzpicture}}
=
\hackcenter{
\begin{tikzpicture}[scale=0.8]
  \draw[thick, color=\clr] (0,0)--(0,-0.2) .. controls ++(0,-0.35) and ++(0,0.35) .. (0.4,-0.9)--(0.4,-1) 
  .. controls ++(0,-0.35) and ++(0,0.35) .. (0,-1.7)--(0,-1.8); 
    \draw[thick, color=\clr] (0,0)--(0,-0.2) .. controls ++(0,-0.35) and ++(0,0.35) .. (0.4,-0.9)--(0.4,-1) 
  .. controls ++(0,-0.35) and ++(0,0.35) .. (0.8,-1.7)--(0.8,-1.8); 
  \draw[thick, color=\clr] (0,0)--(0,-0.2) .. controls ++(0,-0.5) and ++(0,0.5) .. (-0.8,-1.5)--(-0.8,-1.8);
      \node[above] at (-0.8,-2.25) {$\scriptstyle a$};
      \node[above] at (0,-2.25) {$\scriptstyle b$};
      \node[above] at (0.8,-2.25) {$\scriptstyle c$};
      \node[above] at (0,0) {$\scriptstyle a+b+c$};
      \node[right] at (0.4,-0.8) {$\scriptstyle b+c$};
\end{tikzpicture}};
\end{align}

{Rung swap:}
\begin{align}\label{DiagSwitchRel}
\hackcenter{}
\hackcenter{
\begin{tikzpicture}[scale=0.8]
\draw[thick, color=\clr] (0,0)--(0,2);
\draw[thick, color=\clr] (1.4,0)--(1.4,2);
  \draw[thick, color=\clr] (0,0)--(0,0.2) .. controls ++(0,0.35) and ++(0,-0.35)  .. (1.4,0.8)--(1.4,2); 
    \draw[thick, color=\clr] (1.4,0)--(1.4,1.2) .. controls ++(0,0.35) and ++(0,-0.35)  .. (0,1.8)--(0,2); 
     \node[above] at (0,-0.4) {$\scriptstyle a$};
     \node[above] at (1.4,-0.4) {$\scriptstyle b$};
     \node[above] at (0,2) {$\scriptstyle a-s+r$};
      \node[above] at (1.4,2) {$\scriptstyle b+s-r$};
      \node[left] at (0,1) {$\scriptstyle a-s$};
      \node[right] at (1.4,1) {$\scriptstyle b+s$};
      \node[above] at (0.7,1.5) {$\scriptstyle r$};
      \node[below] at (0.7,0.5) {$\scriptstyle s$};
\end{tikzpicture}}
\;
=
\;
\sum_{t \in \Z_{\geq 0}}
\binom{a-b+r-s}{t}
\hackcenter{
\begin{tikzpicture}[scale=0.8]
\draw[thick, color=\clr] (0,0)--(0,2);
\draw[thick, color=\clr] (1.4,0)--(1.4,2);
  \draw[thick, color=\clr] (1.4,0)--(1.4,0.2) .. controls ++(0,0.35) and ++(0,-0.35)  .. (0,0.8)--(0,2); 
    \draw[thick, color=\clr] (0,0)--(0,1.2) .. controls ++(0,0.35) and ++(0,-0.35)  .. (1.4,1.8)--(1.4,2); 
     \node[above] at (0,-0.4) {$\scriptstyle a$};
     \node[above] at (1.4,-0.4) {$\scriptstyle b$};
     \node[above] at (0,2) {$\scriptstyle a-s+r$};
      \node[above] at (1.4,2) {$\scriptstyle b+s-r$};
      \node[left] at (0,1) {$\scriptstyle a+r-t$};
      \node[right] at (1.4,1) {$\scriptstyle b-r+t$};
      \node[above] at (0.7,1.5) {$\scriptstyle s-t$};
      \node[below] at (0.7,0.5) {$\scriptstyle r-t$};
\end{tikzpicture}}.
\end{align}
\end{definition}

%
%If $\bba$ and $\bbb$ are two sequences of nonnegative integers, then a  general morphism in $\aWeb$ from  $\bba$ to $\bbb$  is a $\k$-linear combination of diagrams obtained by vertically and horizontally concatenating splits, merges, and vertical strands labeled by nonnegative integers and, moreover, when the labels along the bottom (resp.\  top) of each diagram are read left to right one obtains the sequence $\bba$ (resp.\  $\bbb$).  On diagrams composition is given by vertical concatenation and the monoidal structure is given by horizontal concatenation.

Going forward, and in the relations defined above, we use the following conventions:
\begin{itemize}
\item Strands labeled by `0' are to be deleted;
\item Diagrams containing a negatively-labeled strand are to be read as zero;
\item We will sometimes choose to omit labels on strands when the label is clear from context.
\end{itemize}
For brevity, we also adopt the convention in calculations that when an equality follows from a previous result, this fact is indicated by placing the relevant equation number over the equals sign in question.  We also adopt the convention that we sometimes write $0$ for the monoidal unit object.

% It is useful to note that the objects in $\aWeb$ are monoidally generated by the objects \(a \in \Z_{\geq 1}\) and that similar remarks will apply to the diagrammatic categories defined in the later sections.

\begin{remark}
When $\k$ is a field \(\aWeb\) can be viewed as a non-quantum version of categories which appear in \cite{QS, ST, TVW}.  It can also be seen to be isomorphic to the Schur category defined in \cite{BEPO}, which appeared as this paper was being prepared.  As explained therein, the Schur category is related to the category introduced in \cite{CKM}.   
\end{remark}

\subsection{Implied Relations for  \texorpdfstring{$\aWeb$}{a-Web}}
We first record a few relations which are implied by the defining relations of $\aWeb$. 
Many of the relations established in the remainder of Section 2 can be inferred from \cite{BEPO, CKM} by using \cite[Remark 4.8, Theorem 4.10]{BEPO}, or may be viewed as analogues of those shown in \cite[\S2]{ST} in the case \(q=1\). Complete proofs of \cref{KnotholeRel,CoxeterWeb,CrossAbsorb,RungsToCross} and \cref{BraidThm} are also available in \cite{MuthNotes}.

\begin{lemma}\label{KnotholeRel}
For all \(a,b \in \ZZ_{\geq 0}\) we have
\begin{align*}
{}
\hackcenter{
}.
\end{align*}
\end{lemma}

\begin{lemma}\label{CoxeterWeb}
The following equalities hold in \(\aWeb\):
\begin{align*}
\hackcenter{}
\hackcenter{
%
},
\end{align*}
for all admissible \(a,b,c,r,r',r'',s,s',s'' \in \Z_{\geq 0}\).
\end{lemma}

\subsection{Braiding for  \texorpdfstring{$\aWeb$}{a-Web}}\label{SS:BraidingInAWeb}  We next establish the category $\aWeb$ admits a symmetric braiding.

For any \(a,b \in \Z_{\geq 0}\), we define the crossing morphism:
\begin{align}\label{CrossDef}
\hackcenter{}
\hackcenter{
%
}.
\end{align*}
\end{lemma}

Using the crossing we record an identity in $\aWeb$ which will be useful in later calculations. This may be viewed as a special case of the Schur product rule, see \cite[(2.3b)]{GreenBook}, \cite[Theorem 4.10]{BEPO}.
\begin{lemma}\label{RungsToCross}
For all $a,b,c,d \in \Z_{\geq 0}$ such that $a+b = c+d$, we have
\begin{align*}
\hackcenter{}
\hackcenter{
%
}.
\end{align*}
\end{lemma}

The following theorem describes the basic relations involving the crossing morphism.

\begin{theorem}\label{BraidThm}
For all \(a,b,c \in \Z_{\geq 0}\), we have:
\begin{align}\label{B1}
\hackcenter{}
\hackcenter{
%
}.
\end{align}
\end{theorem}

We define a crossing morphism  $\bba \otimes \bbb \to \bbb \otimes \bba$ for objects $\bba$ and $\bbb$ by the following diagram:
\begin{equation}\label{E:generalcrossing}
\hackcenter{}
\hackcenter{
%
}.
\end{equation}
The following follows from \cref{BraidThm}.
\begin{corollary}\label{C:aWeb}  The crossing morphisms defined in \cref{E:generalcrossing} define a symmetric braiding on $\aWeb$.
\end{corollary}

%%%%%%%%%%%%%%%%%%%%%%%%%%%%%%%%%%%%%%%%%%%
%%%%%%%%%%%%%%%%%%%%%%%%%%%%%%%%%%%%%%%%%%%
%%%%%%%%%%%%%%%%%%%%%%%%%%%%%%%%%%%%%%%%%%%

\section{The  \texorpdfstring{$\pWeb$}{p-Web} Category} \label{S:pWebdef}

\subsection{Definition of  \texorpdfstring{$\pWeb$}{p-Web}}
From now on we assume $\k$ is an integral domain where $2$ is invertible.  For example, $\k$ could be $\Z[\frac{1}{2}]$.   We again define a diagrammatic $\k$-linear monoidal supercategories by generators and relations as discussed in \cref{SS:aWeb}.

%As in $\aWeb$, the following definition implies that the objects in $\pWeb$ are monoidally generated by the objects \(a \in \Z_{\geq 1}\), and we write $0$ for the identity object.  

\begin{definition} \label{D:pWebdf}  Let \(\pWeb\) be the strict $\k$-linear monoidal supercategory given by generators and relations as follows.  The objects given by sequences of non-negative integers.

The generating morphisms:
\begin{align*}
\hackcenter{
{}
}
\hackcenter{
\begin{tikzpicture}[scale=0.8]
  \draw[thick, color=\clr] (0,0)--(0,0.2) .. controls ++(0,0.35) and ++(0,-0.35) .. (-0.4,0.9)--(-0.4,1);
  \draw[thick, color=\clr] (0,0)--(0,0.2) .. controls ++(0,0.35) and ++(0,-0.35) .. (0.4,0.9)--(0.4,1);
      \node[above] at (-0.4,1) {$\scriptstyle a$};
      \node[above] at (0.4,1) {$\scriptstyle b$};
      \node[above] at (0,-.5) {$\scriptstyle a+b$};
\end{tikzpicture}},
\qquad
\qquad
\hackcenter{
\begin{tikzpicture}[scale=0.8]
  \draw[thick, color=\clr] (-0.4,0)--(-0.4,0.1) .. controls ++(0,0.35) and ++(0,-0.35) .. (0,0.8)--(0,1);
\draw[thick, color=\clr] (0.4,0)--(0.4,0.1) .. controls ++(0,0.35) and ++(0,-0.35) .. (0,0.8)--(0,1);
      \node[above] at (-0.4,-0.47) {$\scriptstyle a$};
      \node[above] at (0.4,-0.47) {$\scriptstyle b$};
      \node[above] at (0,1) {$\scriptstyle a+b$};
\end{tikzpicture}},
\qquad
\qquad
\hackcenter{
\begin{tikzpicture}[scale=0.8]
 \draw(0,0) arc (0:180:0.4cm) [thick, color=\clr];
    \node[below] at (0,0) {$\scriptstyle 1$};
      \node[below] at (-0.8,0) {$\scriptstyle 1$};
      \node at (-0.4,0.4) {$\scriptstyle\blacklozenge$};
\end{tikzpicture}},
\qquad
\qquad
\hackcenter{
\begin{tikzpicture}[scale=0.8]
 \draw[decoration={
    markings,
    mark=at position 0.5 with {\arrow{>}}}] (0,0) arc (0:-180:0.4cm) [thick, color=\clr];
    \node[above] at (0,0) {$\scriptstyle 1$};
      \node[above] at (-0.8,0) {$\scriptstyle 1$};
       \node at (-0.4,-0.4) {$\scriptstyle\blacklozenge$};
\end{tikzpicture}},
\end{align*}
for \(a,b \in \Z_{\geq 0}\). We call these morphisms {\em split, merge, cap, and cup}, respectively. The \(\Z_2\)-grading is given by declaring splits and merges to have parity \(\bar 0\), and caps and cups to have parity \(\bar 1\).  The identity morphism of the object $(a_{1}, \dotsc , a_{k})$ will be depicted by $k$ vertical strands labelled in order by $a_{1}, \dotsc , a_{k}$.

To describe the imposed relations it will be convenient to first define an additional odd morphism,
\begin{equation} \label{E:antenna}
\hackcenter{}
\hackcenter{
\begin{tikzpicture}[scale=0.8]
 \draw[thick, color=\clr] (0,-0.1)--(0,0.3);
 \draw[thick, color=\clr] (0,-0.1)--(0,0.8);
    \node[below] at (0,-0.1) {$\scriptstyle 2$};
    \node[shape=coordinate](DOT) at (0,0.8) {};
     \filldraw  (DOT) circle (2.5pt);
      \node[white] at (0,1.4) {$\scriptstyle\blacklozenge$};
\end{tikzpicture}}
:=
\left(\frac{1}{2} \right)
\,
\hackcenter{
\begin{tikzpicture}[scale=0.8]
\draw[thick, color=\clr] (0,-0.1)--(0,0.2);
  \draw[thick, color=\clr] (0,0)--(0,0.2) .. controls ++(0,0.35) and ++(0,-0.35) .. (-0.4,0.9)--(-0.4,1);
  \draw[thick, color=\clr] (0,0)--(0,0.2) .. controls ++(0,0.35) and ++(0,-0.35) .. (0.4,0.9)--(0.4,1);
      \node[below] at (0,-0.1) {$\scriptstyle 2$};
       \draw(0.4,1) arc (0:180:0.4cm) [thick, color=\clr];
         \node at (0,1.4) {$\scriptstyle\blacklozenge$};
\end{tikzpicture}},\,
\end{equation}
which we call the {\em antenna}.    Here and below, when we scale a diagram by an element of $\k$, we often write the scalar in parentheses to make clear it is not an edge label.

%%% These seem to be things which we have already been doing:
%We also maintain the following conventions:
%\begin{itemize}
%\item Strands labeled by `0' are to be deleted;
%\item Diagrams containing a negatively-labeled strand are to be read as zero;
%\item We will sometimes choose to omit labels on strands when the label is clear from context;
%\item As we did in defining the antenna, when a diagram is scaled by an element of $\k$ (or an integer viewed as an element of $\k$) we will sometimes write the scalar in parentheses to make clear it is not an edge label.
%\end{itemize}

The defining relations of $\pWeb$ are \cref{AssocRel,DiagSwitchRel} along with the following relations for all $a,b \in \Z_{\geq 0}$:\\

{\em Straightening:}
\begin{align}\label{StraightRel}
\hackcenter{}
\hackcenter{
\begin{tikzpicture}[scale=0.8]
  \draw[thick, color=\clr] (0,0)--(0,-0.8); 
 \draw (0,0) arc (0:180:0.4cm) [thick, color=\clr];
   \draw[thick, color=\clr] (-0.8,0)--(-0.8,-0.2); 
  \draw (-0.8,-0.2) arc (0:-180:0.4cm) [thick, color=\clr];
   \draw[thick, color=\clr] (-1.6,-0.2)--(-1.6,0.6); 
    \node[above] at (-1.6,0.6) {$\scriptstyle 1$};
     \node[below] at (0,-0.8) {$\scriptstyle 1$};
      \node at (-1.2,-0.6) {$\scriptstyle\blacklozenge$};
       \node at (-0.4,0.4) {$\scriptstyle\blacklozenge$};
\end{tikzpicture}}
\;
=
\;
\hackcenter{
\begin{tikzpicture}[scale=0.8]
   \draw[thick, color=\clr] (0,-0.8)--(0,0.6); 
      \node[above,white] at (0,0.6) {$\scriptstyle 1$};
     \node[below] at (0,-0.8) {$\scriptstyle 1$};
\end{tikzpicture}}
\;
=
\;
-
\;
\hackcenter{
\begin{tikzpicture}[scale=0.8]
   \draw[thick, color=\clr] (-0.8,0)--(-0.8,-0.8); 
 \draw (0,0) arc (0:180:0.4cm) [thick, color=\clr];
   \draw[thick, color=\clr] (0,0)--(0,-0.2); 
  \draw (0.8,-0.2) arc (0:-180:0.4cm) [thick, color=\clr];
   \draw[thick, color=\clr] (0.8,-0.2)--(0.8,0.6); 
        \node[above] at (0.8,0.6) {$\scriptstyle 1$};
     \node[below] at (-0.8,-0.8) {$\scriptstyle 1$};
         \node at (-0.4,0.4) {$\scriptstyle\blacklozenge$};
       \node at (0.4,-0.6) {$\scriptstyle\blacklozenge$};
\end{tikzpicture}};
\end{align}

{\em Antenna retraction:}
\begin{align}\label{AntRel}
\hackcenter{}
\hackcenter{
\begin{tikzpicture}[scale=0.8]
  \draw[thick, color=\clr] (-0.4,0)--(-0.4,0.1) .. controls ++(0,0.35) and ++(0,-0.35) .. (0,0.8)--(0,1.2);
\draw[thick, color=\clr] (0.4,0)--(0.4,0.1) .. controls ++(0,0.35) and ++(0,-0.35) .. (0,0.8)--(0,1.2);
       \node[shape=coordinate](DOT) at (0,1.2) {};
     \filldraw  (DOT) circle (2.5pt);
     \node[left] at (0,0.9) {$\scriptstyle 2$};
         \node[below] at (0.4,0) {$\scriptstyle 1$};
     \node[below] at (-0.4,0) {$\scriptstyle 1$};
\end{tikzpicture}}
\;
=
\;
\hackcenter{
\begin{tikzpicture}[scale=0.8]
\draw[thick, color=\clr] (-0.8,0)--(-0.8,0.1);
\draw[thick, color=\clr] (0,0)--(0,0.1);
 \draw (0,0.1) arc (0:180:0.4cm) [thick, color=\clr];
            \node[white] at (0,1.13) {\(\)};
               \node[below] at (0,0) {$\scriptstyle 1$};
     \node[below] at (-0.8,0) {$\scriptstyle 1$};
     \node at (-0.4,0.5) {$\scriptstyle\blacklozenge$};
     \end{tikzpicture}};
\end{align}

{\em Cap/rung swap:}
\begin{align}\label{CapDiagSwitchRel1}
\hackcenter{}
\hackcenter{
\begin{tikzpicture}[scale=0.8]
\draw[thick, color=\clr] (0,0)--(0,2);
\draw[thick, color=\clr] (1.4,0)--(1.4,2);
  \draw[thick, color=\clr] (0,0)--(0,0.2) .. controls ++(0,0.35) and ++(0,-0.35)  .. (1.4,0.8)--(1.4,2); 
   \draw (1.4,1.2) arc (0:180:0.7cm) [thick, color=\clr];
     \node[above] at (0,-0.4) { $\scriptstyle a$};
     \node[above] at (1.4,-0.4) {$\scriptstyle b$};
     \node[above] at (0,2) {$\scriptstyle a-r-1$};
      \node[above] at (1.4,2) {$\scriptstyle b+r-1$};
      \node[left] at (0,1) {$\scriptstyle a-r$};
      \node[right] at (1.4,1) {$\scriptstyle b+r$};
        \node[below] at (0.7,0.5) {$\scriptstyle r$};
        \node at (0.7,1.9) {$\scriptstyle\blacklozenge$};
\end{tikzpicture}}
\;
&=
\;
\hackcenter{
\begin{tikzpicture}[scale=0.8]
\draw[thick, color=\clr] (0,0)--(0,2);
\draw[thick, color=\clr] (1.4,0)--(1.4,2);
  \draw[thick, color=\clr] (0,0)--(0,1.2) .. controls ++(0,0.35) and ++(0,-0.35)  .. (1.4,1.8)--(1.4,2); 
   \draw (1.4,0.2) arc (0:180:0.7cm) [thick, color=\clr];
     \node[above] at (0,-0.4) { $\scriptstyle a$};
     \node[above] at (1.4,-0.4) {$\scriptstyle b$};
     \node[above] at (0,2) {$\scriptstyle a-r-1$};
      \node[above] at (1.4,2) {$\scriptstyle b+r-1$};
      \node[left] at (0,1) {$\scriptstyle a-1$};
      \node[right] at (1.4,1) {$\scriptstyle b-1$};
         \node[above] at (0.7,1.5) {$\scriptstyle r$};
          \node at (0.7,0.9) {$\scriptstyle\blacklozenge$};
\end{tikzpicture}}
\;
+
\;
(2)
\;
\hspace{-1mm}
\hackcenter{
\begin{tikzpicture}[scale=0.8]
\draw[thick, color=\clr] (0,0)--(0,2);
\draw[thick, color=\clr] (1.4,0)--(1.4,2);
\draw[thick, color=\clr] (0,0)--(0,1.2) .. controls ++(0,0.35) and ++(0,-0.35)  .. (1.4,1.8)--(1.4,2); 
  \draw[thick, color=\clr] (0,0)--(0,0.2) .. controls ++(0,0.35) and ++(0,-0.35)  .. (0.7,0.8)--(0.7,1); 
     \node[above] at (0,-0.4) { $\scriptstyle a$};
     \node[above] at (1.4,-0.4) {$\scriptstyle b$};
     \node[above] at (0,2) {$\scriptstyle a-r-1$};
      \node[above] at (1.4,2) {$\scriptstyle b+r-1$};
          \node[above] at (0.7,0.1) {$\scriptstyle 2$};
           \node[left] at (0,0.8) {$\scriptstyle a-2$};
     \node[shape=coordinate](DOT) at (0.7,1) {};
     \filldraw  (DOT) circle (2.5pt);
       \node[above] at (0.7,1.5) {$\scriptstyle r-1$};
\end{tikzpicture}};
\end{align}
\begin{align}\label{CapDiagSwitchRel2}
\hackcenter{}
\hackcenter{
\begin{tikzpicture}[scale=0.8]
\draw[thick, color=\clr] (0,0)--(0,2);
\draw[thick, color=\clr] (1.4,0)--(1.4,2);
  \draw[thick, color=\clr] (1.4,0)--(1.4,0.2) .. controls ++(0,0.35) and ++(0,-0.35)  .. (0,0.8)--(0,2); 
   \draw (1.4,1.2) arc (0:180:0.7cm) [thick, color=\clr];
     \node[above] at (0,-0.4) { $\scriptstyle a$};
     \node[above] at (1.4,-0.4) {$\scriptstyle b$};
     \node[above] at (0,2) {$\scriptstyle a+r-1$};
      \node[above] at (1.4,2) {$\scriptstyle b-r-1$};
      \node[left] at (0,1) {$\scriptstyle a+r$};
      \node[right] at (1.4,1) {$\scriptstyle b-r$};
        \node[below] at (0.7,0.5) {$\scriptstyle r$};
          \node at (0.7,1.9) {$\scriptstyle\blacklozenge$};
\end{tikzpicture}}
\;
&=
\;
\hackcenter{
\begin{tikzpicture}[scale=0.8]
\draw[thick, color=\clr] (0,0)--(0,2);
\draw[thick, color=\clr] (1.4,0)--(1.4,2);
  \draw[thick, color=\clr] (1.4,0)--(1.4,1.2) .. controls ++(0,0.35) and ++(0,-0.35)  .. (0,1.8)--(0,2); 
   \draw (1.4,0.2) arc (0:180:0.7cm) [thick, color=\clr];
     \node[above] at (0,-0.4) { $\scriptstyle a$};
     \node[above] at (1.4,-0.4) {$\scriptstyle b$};
     \node[above] at (0,2) {$\scriptstyle a+r-1$};
      \node[above] at (1.4,2) {$\scriptstyle b-r-1$};
      \node[left] at (0,1) {$\scriptstyle a-1$};
      \node[right] at (1.4,1) {$\scriptstyle b-1$};
         \node[above] at (0.7,1.5) {$\scriptstyle r$};
           \node at (0.7,0.9) {$\scriptstyle\blacklozenge$};
\end{tikzpicture}}
\;
+
\;
(2)
\;
\hspace{-1mm}
\hackcenter{
\begin{tikzpicture}[scale=0.8]
\draw[thick, color=\clr] (0,0)--(0,2);
\draw[thick, color=\clr] (1.4,0)--(1.4,2);
\draw[thick, color=\clr] (1.4,0)--(1.4,1.2) .. controls ++(0,0.35) and ++(0,-0.35)  .. (0,1.8)--(0,2); 
  \draw[thick, color=\clr] (1.4,0)--(1.4,0.2) .. controls ++(0,0.35) and ++(0,-0.35)  .. (0.7,0.8)--(0.7,1); 
     \node[above] at (0,-0.4) { $\scriptstyle a$};
     \node[above] at (1.4,-0.4) {$\scriptstyle b$};
     \node[above] at (0,2) {$\scriptstyle a+r-1$};
      \node[above] at (1.4,2) {$\scriptstyle b-r-1$};
          \node[above] at (0.7,0.1) {$\scriptstyle 2$};
           \node[right] at (1.4,0.8) {$\scriptstyle b-2$};
     \node[shape=coordinate](DOT) at (0.7,1) {};
     \filldraw  (DOT) circle (2.5pt);
       \node[above] at (0.7,1.5) {$\scriptstyle r-1$};
\end{tikzpicture}};
\end{align}

{\em Cup/rung swap:}
\begin{align}\label{CupDiagSwitchRel1}
\hackcenter{}
\hackcenter{
\begin{tikzpicture}[scale=0.8]
\draw[thick, color=\clr] (0,0)--(0,2);
\draw[thick, color=\clr] (1.4,0)--(1.4,2);
  \draw[thick, color=\clr] (0,0)--(0,1.2) .. controls ++(0,0.35) and ++(0,-0.35)  .. (1.4,1.8)--(1.4,2); 
   \draw (1.4,0.8) arc (0:-180:0.7cm) [thick, color=\clr];
     \node[above] at (0,-0.4) { $\scriptstyle a$};
     \node[above] at (1.4,-0.4) {$\scriptstyle b$};
     \node[above] at (0,2) {$\scriptstyle a-r+1$};
      \node[above] at (1.4,2) {$\scriptstyle b+r+1$};
      \node[left] at (0,1) {$\scriptstyle a+1$};
      \node[right] at (1.4,1) {$\scriptstyle b+1$};
      \node[above] at (0.7,1.5) {$\scriptstyle r$};
        \node at (0.7,0.1) {$\scriptstyle\blacklozenge$};
\end{tikzpicture}}
\;
&=
\;
\hackcenter{
\begin{tikzpicture}[scale=0.8]
\draw[thick, color=\clr] (0,0)--(0,2);
\draw[thick, color=\clr] (1.4,0)--(1.4,2);
  \draw[thick, color=\clr] (0,0)--(0,0.2) .. controls ++(0,0.35) and ++(0,-0.35)  .. (1.4,0.8)--(1.4,2); 
   \draw (1.4,1.8) arc (0:-180:0.7cm) [thick, color=\clr];
     \node[above] at (0,-0.4) { $\scriptstyle a$};
     \node[above] at (1.4,-0.4) {$\scriptstyle b$};
     \node[above] at (0,2) {$\scriptstyle a-r+1$};
      \node[above] at (1.4,2) {$\scriptstyle b+r+1$};
      \node[left] at (0,1) {$\scriptstyle a-r$};
      \node[right] at (1.4,1) {$\scriptstyle b+r$};
       \node[below] at (0.7,0.5) {$\scriptstyle r$};
         \node at (0.7,1.1) {$\scriptstyle\blacklozenge$};
\end{tikzpicture}};
\end{align}
\begin{align}\label{CupDiagSwitchRel2}
\hackcenter{}
\hackcenter{
\begin{tikzpicture}[scale=0.8]
\draw[thick, color=\clr] (0,0)--(0,2);
\draw[thick, color=\clr] (1.4,0)--(1.4,2);
  \draw[thick, color=\clr] (1.4,0)--(1.4,1.2) .. controls ++(0,0.35) and ++(0,-0.35)  .. (0,1.8)--(0,2); 
   \draw (1.4,0.8) arc (0:-180:0.7cm) [thick, color=\clr];
     \node[above] at (0,-0.4) { $\scriptstyle a$};
     \node[above] at (1.4,-0.4) {$\scriptstyle b$};
     \node[above] at (0,2) {$\scriptstyle a+r+1$};
      \node[above] at (1.4,2) {$\scriptstyle b-r+1$};
      \node[left] at (0,1) {$\scriptstyle a+1$};
      \node[right] at (1.4,1) {$\scriptstyle b+1$};
      \node[above] at (0.7,1.5) {$\scriptstyle r$};
        \node at (0.7,0.1) {$\scriptstyle\blacklozenge$};
\end{tikzpicture}}
\;
&=
\;
\hackcenter{
\begin{tikzpicture}[scale=0.8]
\draw[thick, color=\clr] (0,0)--(0,2);
\draw[thick, color=\clr] (1.4,0)--(1.4,2);
  \draw[thick, color=\clr] (1.4,0)--(1.4,0.2) .. controls ++(0,0.35) and ++(0,-0.35)  .. (0,0.8)--(0,2); 
   \draw (1.4,1.8) arc (0:-180:0.7cm) [thick, color=\clr];
     \node[above] at (0,-0.4) { $\scriptstyle a$};
     \node[above] at (1.4,-0.4) {$\scriptstyle b$};
     \node[above] at (0,2) {$\scriptstyle a+r+1$};
      \node[above] at (1.4,2) {$\scriptstyle b-r+1$};
      \node[left] at (0,1) {$\scriptstyle a+r$};
      \node[right] at (1.4,1) {$\scriptstyle b-r$};
       \node[below] at (0.7,0.5) {$\scriptstyle r$};
         \node at (0.7,1.1) {$\scriptstyle\blacklozenge$};
\end{tikzpicture}}.
\end{align}
\end{definition}

\begin{remark} The diamonds which decorate the odd cup and cap morphisms in the definition of $\pWeb$ are used to distinguish them from the even cup and cap morphisms used in the definition of the oriented web category $\pWeb_{\uparrow \downarrow}$  in  \cref{SS:DefinitionofPWebupdown}.   \end{remark}

\begin{remark}  We introduced the antenna in \cref{E:antenna} because this morphism appears frequently when studying applications for $\pWeb$.  It has the disadvantage of requiring $2$ to be invertible in $\k$.   One could define $\pWeb$ over $\Z$ by including the antenna as a generating morphism and instead imposing the relation obtained by scaling \cref{E:antenna} by $2$, along with relations for moving an antenna past other the generating morphisms.  This diagrammatic category would be related to the representation theory of the Kostant $\Z$-form $U(\fp )_{\Z}$ introduced in \cite{DKM}.   We opted not to do this as it would add complexity and was not needed for the applications considered here.
\end{remark}
%%%%%%%%%%%%%%%%%%%%%%%%%%%%%%%%%%%

%%%%%%%%%%%%%%%%%%%%%%%%%%%%%%%%%%%

\subsection{Implied Relations for \texorpdfstring{$\pWeb$}{p-Web}}
We first record a few additional relations which are implied by the defining relations of $\pWeb$. 

\begin{lemma}\label{DotSwitch}
For all \(a \in \Z_{\geq 2}\) we have
\begin{align*}
\hackcenter{}
\hackcenter{
}.
\end{align*}
\end{lemma}
\begin{proof}
Applying cap-rung swap relations \cref{CapDiagSwitchRel1,CapDiagSwitchRel2}  to push caps to the bottom of diagrams, we have:
\begin{align}\label{Claim1}
\hackcenter{}
\hackcenter{
%
}.\nonumber
\end{align}
Therefore we have
\begin{align*}
\hackcenter{}
\hspace{-1mm}
\hackcenter{
%
},
\end{align*}
proving the first claim. The second claim follows using analogous arguments.
\end{proof}

%%%%%%%%%%%%%%%%%%%%%%%%%%%%%%%%%%%%%%

\subsection{Braiding for  \texorpdfstring{$\pWeb$}{p-Web}}\label{SS:BraidingInPWeb} 

For \(a,b \in \Z_{\geq 0}\) we define the crossing morphism $(a,b) \to (b,a)$ 
%\(\hackcenter{
%%
}
%\) 
in $\pWeb$ as in \cref{CrossDef}.  The goal of this section is to prove the crossing morphism can be used to define a symmetric braiding on $\pWeb$.  

\begin{theorem}\label{BraidThmP}
For all \(a \in \Z_{\geq 0} \), we have:
\begin{align}\label{B5}
\hackcenter{}
\hackcenter{
%
},
\end{align*}
which implies that 
\begin{align*}
\hackcenter{}
\hackcenter{
%
},
\end{align*}
proving the first equality in \cref{B5}. Now using this equality, we may precompose with two cup morphisms to arrive at the equality
\begin{align*}
\hackcenter{}
\hackcenter{
%
}.
\end{align*}
Applying the straightening relation \cref{StraightRel} to both sides of this equation gives the second equality in \cref{B5}. 

\end{proof}

\cref{BraidThmP} demonstrates that cups and caps are natural with respect to the crossing.  Because the relations in $\aWeb$ also hold in $\pWeb$, the following is an immediate consequence of the above and \cref{BraidThm}.

\begin{corollary}\label{C:pwebbraiding}  The crossing morphisms defined in \cref{E:generalcrossing} define a symmetric braiding on $\pWeb$. 
\end{corollary}

\subsection{Additional Relations in  \texorpdfstring{$\pWeb$}{p-Web}}\label{SS:AdditionalRelationsInPWeb}  The crossing morphisms allow for the following additional relations on $\pWeb$.

\begin{lemma} For every $a \in \Z_{\geq 0}$, we have:
\begin{align}\label{DotThruCross}
\hackcenter{}
\hackcenter{
}.
\end{align}
\end{lemma}
\begin{proof} Follows immediately from the definition of the antenna, along with \cref{BraidThm} and \cref{B5}.
\end{proof}

\begin{lemma}\label{TwistMarkBubble}
We have:
\begin{align*}
\hackcenter{
%
}
=0.
\end{align*}
\end{lemma}
\begin{proof}
The first two relations are easily seen to hold by direct computation and applications of \cref{CapDiagSwitchRel1,CapDiagSwitchRel2,CupDiagSwitchRel1,CupDiagSwitchRel2}.  The third is obtained by combining the first two:
$$
\hackcenter{
%
}\; ,
$$
where the first equality is deduced by putting a cap on top of the first relation, and the second equality is deduced by putting a cup under the second relation.  Because $2$ is invertible in the ground ring $\k$, the third relation follows.
\end{proof}

\begin{lemma}\label{BCswap}
For all \(a,b \in \Z_{\geq 0}\), we have:
\begin{align*}
\hackcenter{}
\hackcenter{
%
}.
\end{align*}
\begin{proof}
First, note that:
\begin{align*}
\hackcenter{}
\hackcenter{
%
}.
\end{align*}
Therefore we have:
\begin{align*}
\hackcenter{}
\hackcenter{
%
}.
\end{align*}

\begin{proof}
For the first equality, we have:
\begin{align*}
\hackcenter{
%
},
\end{align*}
as desired. The second equality is proved in a similar fashion.
\end{proof}

\end{lemma}

\begin{lemma}\label{ZeroLem}
The following relation holds in $\pWeb$: 
\begin{align*}
\hackcenter{}
\hackcenter{
%
}
\;
,
\end{align*}
proving the first claim. The second claim is similar. For the third, we have
\begin{align*}
\hackcenter{}
\hackcenter{
%
}
,
\end{align*}
completing the proof.
\end{proof}

\subsection{A Basis for Morphism Spaces in  \texorpdfstring{$\aWeb$}{a-Web} and  \texorpdfstring{$\pWeb$}{p-Web}}\label{SS:BasisForPWebs}
In this section we construct $\k$-spanning sets for the morphism spaces in $\aWeb$ and \(\pWeb\).  In \cref{SS:BasisTheorems} we will show these are in fact bases.  The bases themselves are diagrammatically analogous to bases defined in a different setting in \cite[Section 5]{SW}.  

We will write multiple splits and merges in the form
\begin{align}\label{E:multisplitmerge}
\hackcenter{}
\hackcenter{
\begin{tikzpicture}[scale=0.8]
 \draw[thick, color=\clr] (0,-0.2)--(0,0);
 \draw[thick, color=\clr] (0,0) .. controls ++(0,0.5) and ++(0,-0.5) .. (-1.2,1);
  \draw[thick, color=\clr] (0,0) .. controls ++(0,0.5) and ++(0,-0.5) .. (-0.6,1);
    \draw[thick, color=\clr] (0,0) .. controls ++(0,0.5) and ++(0,-0.5) .. (0.6,1);
       \draw[thick, color=\clr] (0,0) .. controls ++(0,0.5) and ++(0,-0.5) .. (1.2,1);
          \node at (0,1) {$\scriptstyle \cdots $};
            \node[below] at (0.1,-0.2) {$\scriptstyle a_1 + \cdots + a_{n+1} $};
             \node[above] at (-1.2,1) {$\scriptstyle a_1 $};
              \node[above] at (-0.6,1) {$\scriptstyle a_2 $};
               \node[above] at (0.6,1) {$\scriptstyle a_n$};
                \node[above] at (1.2,0.97) {$\scriptstyle \;a_{n+1} $};
\end{tikzpicture}}
\qquad
\qquad
\textup{and}
\qquad
\qquad
\hackcenter{
\begin{tikzpicture}[scale=0.8]
 \draw[thick, color=\clr] (0,0.2)--(0,0);
 \draw[thick, color=\clr] (0,0) .. controls ++(0,-0.5) and ++(0,0.5) .. (-1.2,-1);
  \draw[thick, color=\clr] (0,0) .. controls ++(0,-0.5) and ++(0,0.5) .. (-0.6,-1);
    \draw[thick, color=\clr] (0,0) .. controls ++(0,-0.5) and ++(0,0.5) .. (0.6,-1);
       \draw[thick, color=\clr] (0,0) .. controls ++(0,-0.5) and ++(0,0.5) .. (1.2,-1);
          \node at (0,-1) {$\scriptstyle \cdots $};
            \node[above] at (0.1,0.2) {$\scriptstyle a_1 + \cdots + a_{n+1} $};
             \node[below] at (-1.2,-1) {$\scriptstyle a_1 $};
              \node[below] at (-0.6,-1) {$\scriptstyle a_2 $};
               \node[below] at (0.6,-1) {$\scriptstyle a_n $};
                \node[below] at (1.2,-1) {$\scriptstyle \;a_{n+1} $};
\end{tikzpicture}},
\end{align}
where the diagram should be interpreted as \(n\) vertically composed splits, or merges, respectively. By \cref{AssocRel} the resulting morphism is independent of the split (or merge) order. It will also be convenient to define: 
\begin{align}\label{greendot}
\hackcenter{}
  \hackcenter{
\begin{tikzpicture}[scale=0.7]
  \draw[thick, color=\clr] (0,0)--(0,2);
     \node[shape=coordinate](DOT) at (0,1) {};
     \filldraw  (DOT) circle (2.5pt);
  \draw[thick, color=\clr, fill=green]  (DOT) circle (5pt);
        \node[below] at (0,0) {$\scriptstyle a $};
         \node[above] at (0,2) {$\scriptstyle a-2 $};
      \end{tikzpicture}}
:=
     \hackcenter{
\begin{tikzpicture}[scale=0.7]
  \draw[thick, color=\clr] (0,0)--(0,2);
   \draw[thick, color=\clr] (0,0)--(0,0.2) .. controls ++(0,0.35) and ++(0,-0.35) .. (0.4,0.8)--(0.4,1);
     \node[below] at (0,0) {$\scriptstyle a$};
      \node[below] at (0.4,0.6) {$\scriptstyle 2$};
      \node[above] at (0,2) {$\scriptstyle a-2 $};
        \node[shape=coordinate](DOT) at (0.4,1) {};
     \filldraw  (DOT) circle (2.5pt);
      \end{tikzpicture}}. 
\end{align}

Given a matrix $A$ let us write $A^{T}$ for the transpose.  For any \(\bba \in \Z_{\geq 0}^t\), \(\bbb \in \Z_{\geq 0}^u\), let \(\chi(\bba, \bbb)\) be the set of tuples \((A,B,C,D)\), such that
\begin{align*}
A \in \textup{Mat}_{t \times t}(\{0,1\}), 
\qquad
B \in \textup{Mat}_{u \times u}(\{0,1\}), 
\qquad
C \in \textup{Mat}_{t \times u}(\Z_{\geq 0}), 
\qquad
D \in \{0,1\}^t
\end{align*}
\begin{align*}
A^{T} = A, \qquad B^{T} = B, \qquad A_{ii} = 0\,\, \textup{for all \(i=1, \ldots, t\)}, \qquad B_{ii} = 0\,\,\textup{for all \(i=1, \ldots, u\)},
\end{align*}
\begin{align*}
2D_{i} + \sum_{j=1}^tA_{ij} + \sum_{j=1}^u C_{ij} = a_i \qquad \textup{for all \(i =1, \ldots, t\)},
\end{align*}
\begin{align*}
\sum_{i=1}^uB_{ij} + \sum_{i=1}^t C_{ij} = b_j \qquad \textup{for all \(j =1, \ldots, u\)}.
\end{align*}

For any \((A,B,C,D) \in \chi(\bba, \bbb)\), we define an associated element \(\xi^{(A,B,C,D)} \in \Hom_{\pWeb}(\bba, \bbb)\) via
\begin{align*}
\xi^{(A,B,C,D)}:=
\hackcenter{}
\hackcenter{
\begin{tikzpicture}[scale=0.8]
 \draw[thick, color=\clr] (0,-1)--(0,0);
 \draw[thick, color=\clr] (0,0) .. controls ++(0,0.5) and ++(0,-0.5) .. (-1.2,1)--(-1.2,2);
  \draw[thick, color=\clr] (0,0) .. controls ++(0,0.5) and ++(0,-0.5) .. (-0.2,1)--(-0.2,2);
    \draw[thick, color=\clr] (0,0) .. controls ++(0,0.5) and ++(0,-0.5) .. (0.2,1)--(0.2,2);
       \draw[thick, color=\clr] (0,0) .. controls ++(0,0.5) and ++(0,-0.5) .. (1.2,1)--(1.2,2);
          \node[left] at (-1.2,1.5) {$\scriptstyle A_{11}$};
          \node[left] at (-0.2,1.5) {$\scriptstyle A_{1t}$};
          \node[left] at (-0.3,1) {$\scriptstyle \cdots $};
                     \node[right] at (0.2,1.5) {$\scriptstyle C_{11}$};
            \node[right] at (1.2,1.5) {$\scriptstyle C_{1u}$};
            \node[right] at (0.3,1) {$\scriptstyle \cdots $};
            \node[below] at (0,-1) {$\scriptstyle a_1 $};
            \node[shape=coordinate](DOT) at (0,-0.4) {};
     \draw[thick, color=\clr, fill=green]  (DOT) circle (5pt);
      \node[right] at (0.1,-0.4) {$\scriptstyle D_1 $};
  %%%%%%%%%%
  \node at (2.2,-0.4) {$\scriptstyle \cdots $};
   \node at (2.2,1) {$\scriptstyle \cdots $};
  %%%%%%%%%%
   \draw[thick, color=\clr] (4.4,-1)--(4.4,0);
 \draw[thick, color=\clr] (4.4,0) .. controls ++(0,0.5) and ++(0,-0.5) .. (3.2,1)--(3.2,2);
  \draw[thick, color=\clr] (4.4,0) .. controls ++(0,0.5) and ++(0,-0.5) .. (4.2,1)--(4.2,2);
    \draw[thick, color=\clr] (4.4,0) .. controls ++(0,0.5) and ++(0,-0.5) .. (4.6,1)--(4.6,2);
       \draw[thick, color=\clr] (4.4,0) .. controls ++(0,0.5) and ++(0,-0.5) .. (5.6,1)--(5.6,2);
          \node[left] at (3.2,1.5) {$\scriptstyle A_{t1}$};
          \node[left] at (4.2,1.5) {$\scriptstyle A_{tt}$};
          \node[left] at (4.1,1) {$\scriptstyle \cdots $};
                     \node[right] at (4.6,1.5) {$\scriptstyle C_{t1}$};
            \node[right] at (5.6,1.5) {$\scriptstyle C_{tu}$};
            \node[right] at (4.7,1) {$\scriptstyle \cdots $};
            \node[below] at (4.4,-1) {$\scriptstyle a_t $};
            \node[shape=coordinate](DOT) at (4.4,-0.4) {};
     \draw[thick, color=\clr, fill=green]  (DOT) circle (5pt);
      \node[right] at (4.5,-0.4) {$\scriptstyle D_t $};
  %%%%%%%%%%
   \draw[thick, color=\clr, fill=red] (-1.4,2)--(5.8,2)--(5.8,2.8)--(-1.4,2.8)--(-1.4,2);
    \node at (2.2,2.4) {$\scriptstyle X$};
    %%%%%%%%%
    \draw[thick, color=\clr] (0,5)--(0,4.8);
 \draw[thick, color=\clr] (0,4.8) .. controls ++(0,-0.5) and ++(0,0.5) .. (-1.2,3.8)--(-1.2,2.8);
  \draw[thick, color=\clr] (0,4.8) .. controls ++(0,-0.5) and ++(0,0.5) .. (-0.2,3.8)--(-0.2,2.8);
    \draw[thick, color=\clr] (0,4.8) .. controls ++(0,-0.5) and ++(0,0.5) .. (0.2,3.8)--(0.2,2.8);
       \draw[thick, color=\clr] (0,4.8) .. controls ++(0,-0.5) and ++(0,0.5) .. (1.2,3.8)--(1.2,2.8);
          \node[left] at (-1.2,3.3) {$\scriptstyle B_{11}$};
          \node[left] at (-0.2,3.3) {$\scriptstyle B_{u1}$};
          \node[left] at (-0.3,3.8) {$\scriptstyle \cdots $};
                     \node[right] at (0.2,3.3) {$\scriptstyle C_{11}$};
            \node[right] at (1.2,3.3) {$\scriptstyle C_{t1}$};
            \node[right] at (0.3,3.8) {$\scriptstyle \cdots $};
            \node[above] at (0,5) {$\scriptstyle b_1 $};
        %%%%%%%%%%    
         %%%%%%%%%
    \draw[thick, color=\clr] (4.4,5)--(4.4,4.8);
 \draw[thick, color=\clr] (4.4,4.8) .. controls ++(0,-0.5) and ++(0,0.5) .. (3.2,3.8)--(3.2,2.8);
  \draw[thick, color=\clr] (4.4,4.8) .. controls ++(0,-0.5) and ++(0,0.5) .. (4.2,3.8)--(4.2,2.8);
    \draw[thick, color=\clr] (4.4,4.8) .. controls ++(0,-0.5) and ++(0,0.5) .. (4.6,3.8)--(4.6,2.8);
       \draw[thick, color=\clr] (4.4,4.8) .. controls ++(0,-0.5) and ++(0,0.5) .. (5.6,3.8)--(5.6,2.8);
          \node[left] at (3.2,3.3) {$\scriptstyle B_{1u}$};
          \node[left] at (4.2,3.3) {$\scriptstyle B_{uu}$};
          \node[left] at (4.1,3.8) {$\scriptstyle \cdots $};
                     \node[right] at (4.6,3.3) {$\scriptstyle C_{1u}$};
            \node[right] at (5.6,3.3) {$\scriptstyle C_{tu}$};
            \node[right] at (4.7,3.8) {$\scriptstyle \cdots $};
            \node[above] at (4.4,5) {$\scriptstyle b_u $};   
              %%%%%%%%%%
  \node at (2.2,4.8) {$\scriptstyle \cdots $};
   \node at (2.2,3.8) {$\scriptstyle \cdots $};
\end{tikzpicture}}
\end{align*}
where \(X\) is any diagram composed only of crossings, cups, and caps, where no cup occurs below any cap, and in which: 
\begin{itemize}
\item the strands labeled by \(A_{ij}\) and \(A_{ji}\) meet in a cap in \(X\);
\item the strands labeled by \(B_{ij}\) and \(B_{ji}\) meet in a cup in \(X\), and;
\item the strand labeled by \(C_{ij}\) at the bottom of \(X\) meets the strand labeled by \(C_{ij}\) at the top of \(X\).
\end{itemize}
All such choices for \(X\) are equivalent up to sign because of \cref{BraidThm}.

It should be noted that the method of using splits and merges to ``explode'' or ``collapse'' the bottom and top of morphisms as done here can be found elsewhere in the literature (e.g., see \cite{RT}).

\begin{example}
Let \(\bba = (9,4,8)\), \(\bbb = (9,6)\), and set
\begin{align*}
A=
\begin{bmatrix}
0 & 0 & 1\\
0 & 0 & 1\\
1 & 1 & 0
\end{bmatrix}
\qquad
B=
\begin{bmatrix}
0 & 1\\ 
1 & 0
\end{bmatrix}
\qquad
C=
\begin{bmatrix}
2& 4 \\
3 & 0 \\
3 & 1 
\end{bmatrix}
\qquad
D=
\begin{bmatrix}
1 & 0 & 1
\end{bmatrix}.
\end{align*}
Then \((A,B,C,D) \in \chi(\bba, \bbb)\), and
\begin{align*}
\xi^{(A,B,C,D)}
=
\pm
\hackcenter{}
\hackcenter{
\begin{tikzpicture}[scale=0.8]
 \draw[thick, color=\clr] (0,0) .. controls ++(0,0.5) and ++(0,-0.5) .. (-0.8,0.6);
  \draw[thick, color=\clr] (0,0) .. controls ++(0,0.5) and ++(0,-0.5) .. (0,0.6);
    \draw[thick, color=\clr] (0,0) .. controls ++(0,0.5) and ++(0,-0.5) .. (0.8,0.6);
    %%%
   \draw[thick, color=\clr] (2,0) .. controls ++(0,0.35) and ++(0,-0.35) .. (1.6,0.6);
  \draw[thick, color=\clr] (2,0) .. controls ++(0,0.35) and ++(0,-0.35) .. (2.4,0.6);
  %%%%%%
  \draw[thick, color=\clr] (4.4,0) .. controls ++(0,0.5) and ++(0,-0.5) .. (3.2,0.6);
    \draw[thick, color=\clr] (4.4,0) .. controls ++(0,0.35) and ++(0,-0.35) .. (4.8,0.6);
     \draw[thick, color=\clr] (4.4,0) .. controls ++(0,0.5) and ++(0,-0.5) .. (5.6,0.6);
     %%%%%%
     \draw[thick, color=\clr] (-0.8,0.6)--(-0.8,1.1) .. controls ++(0,1) and ++(0,1) .. (3.2,1.1)--(3.2,0.6);
      \draw[thick, color=\clr] (1.6,0.6) .. controls ++(0,1.2) and ++(0,1)  .. (4.4,0);
      %%%%%%
       \draw[thick, color=\clr] (1.2,3.6) .. controls ++(0,-0.5) and ++(0,0.5) .. (0,3);
        \draw[thick, color=\clr] (1.2,3.6) .. controls ++(0,-0.35) and ++(0,0.35) .. (0.8,3);
        \draw[thick, color=\clr] (1.2,3.6) .. controls ++(0,-0.35) and ++(0,0.35) .. (1.6,3);
        \draw[thick, color=\clr] (1.2,3.6) .. controls ++(0,-0.5) and ++(0,0.5) .. (2.4,3);
        %%%%%%%%
             \draw[thick, color=\clr] (4,3.6) .. controls ++(0,-0.5) and ++(0,0.5) .. (3.2,3);
        \draw[thick, color=\clr] (4,3.6) .. controls ++(0,-0.5) and ++(0,0.5) .. (4,3);
        \draw[thick, color=\clr] (4,3.6) .. controls ++(0,-0.5) and ++(0,0.5) .. (4.8,3);
        %%%%%%%%%
         \draw[thick, color=\clr] (0,3) .. controls ++(0,-1) and ++(0,-1) .. (3.2,3);
         %%%%%%%%
         \draw[thick, color=\clr] (0,0.6) .. controls ++(0,.5) and ++(0,-.5) .. (-0.4,1.6)
                .. controls ++(0,0.5) and ++(0,-0.5) .. (0.8,3);
         \draw[thick, color=\clr] (0.8,0.6) .. controls ++(0,1) and ++(0,-1) .. (4,3);
         \draw[thick, color=\clr] (5.6,0.6) .. controls ++(0,1) and ++(0,-1) .. (4.8,3);
          \draw[thick, color=\clr] (4.8,0.6) .. controls ++(0,1) and ++(0,-1) .. (2.4,3);
           \draw[thick, color=\clr] (2.4,0.6) .. controls ++(0,.5) and ++(0,-.5) .. (0.4,1.6)
                .. controls ++(0,0.5) and ++(0,-0.5) .. (1.6,3);
                %%%%%%%%%
                \node at (1.6,2.25) {$\scriptstyle\blacklozenge$};
                \node at (1.3,1.85) {$\scriptstyle\blacklozenge$};
                \node at (2.45,1.2) {$\scriptstyle\blacklozenge$};
              %%%%%%%%%%%%%
               \draw[thick, color=\clr] (0,0)--(0,-0.8);
               \draw[thick, color=\clr] (2,0)--(2,-0.8);
                 \draw[thick, color=\clr] (4.4,0)--(4.4,-0.8);
                    \draw[thick, color=\clr] (1.2,3.6)--(1.2,3.8);
               \draw[thick, color=\clr] (4,3.6)--(4,3.8);
                 %%%%%%%%%%%
                  \draw[thick, color=\clr] (0,-0.7) .. controls ++(0,0.35) and ++(0,-0.35) .. (0.4,-0.1);
                    \node[shape=coordinate](DOT) at (0.4,-0.1) {};
     \filldraw  (DOT) circle (2.5pt);
     \draw[thick, color=\clr] (4.4,-0.7) .. controls ++(0,0.35) and ++(0,-0.35) .. (4.8,-0.1);
                    \node[shape=coordinate](DOT) at (4.8,-0.1) {};
     \filldraw  (DOT) circle (2.5pt);
     %%%%%%%%%%%
       \node[below] at (0,-0.8) {$\scriptstyle 9 $};   
       \node[below] at (2,-0.8) {$\scriptstyle 4 $};   
       \node[below] at (4.4,-0.8) {$\scriptstyle 8$};   
        \node[above] at (1.2,3.8) {$\scriptstyle 9 $};  
         \node[above] at (4,3.8) {$\scriptstyle 6 $};  
         %%%%%%
           \node at (-1,0.7) {$\scriptstyle 1 $}; 
         \node at (-0.22,0.7) {$\scriptstyle 2 $};
          \node at (2.6,0.7) {$\scriptstyle 3 $};  
            \node at (1.4,0.7) {$\scriptstyle 1 $};  
          \node at (0.6,0.7) {$\scriptstyle 4 $}; 
           \node at (5,0.7) {$\scriptstyle 3 $}; 
           \node at (3.4,0.7) {$\scriptstyle 1 $}; 
           \node at (4.3,0.7) {$\scriptstyle 1 $}; 
            \node at (5.8,0.7) {$\scriptstyle 1 $}; 
              \node at (0.4,-0.5) {$\scriptstyle 2 $}; 
               \node at (4.8,-0.5) {$\scriptstyle 2 $}; 
       \end{tikzpicture}}.
\end{align*}
\end{example}

\begin{proposition}\label{webspan}
The set 
\begin{align*}
\mathscr{B}:=\left\{ \left.  \xi^{(0,0,C,0)}\; \right|  (0,0,C,0) \in \chi(\bba, \bbb)\right\}
\end{align*}
is a \(\k\)-spanning set for \(\Hom_{\aWeb}(\bba,\bbb)\).

The set 
\begin{align*}
\mathscr{B}:=\left\{ \left. \xi^{(A,B,C,D)}\;   \right| (A,B,C,D) \in \chi(\bba, \bbb) \right\}
\end{align*}
is a \(\k\)-spanning set for \(\Hom_{\pWeb}(\bba,\bbb)\).

\end{proposition}
\begin{proof}
We will focus primarily on the statement for the spanning set of morphisms in $\pWeb$.   The statement for $\aWeb$ is simpler due to the lack of cups and caps, and should be considered known (e.g., see \cite[Theorem 3.11]{SW} or \cite[Lemma 4.9]{BEPO}).

Let \(f\) be a diagram in \(\Hom_{\pWeb}(\bba,\bbb)\). 
By inducting on the number of ``out of place'' parts, we may apply the defining relations \cref{AssocRel,CupDiagSwitchRel2} of \(\pWeb\),  \cref{BraidThm}, and \cref{RungsToCross}, to rewrite \(f\) as a linear combination of diagrams consisting of splits, merges, cups, caps, antennas, and crossings, where:
\begin{itemize}
\item no merge occurs below any split;
\item no cup occurs below any cap;
\item no crossing occurs above any merge or below any split;
\item no antenna occurs above any merge, cup, cap, or crossing.
\end{itemize}
Any diagram which satisfies all of the above is equivalent to a constant multiple of some diagram with the following form:
\begin{align}\label{spandiag1}
\hackcenter{}
\hackcenter{
\begin{tikzpicture}[scale=0.8]
 \draw[thick, color=\clr] (0,-0.2)--(0,0.4);
 \draw[thick, color=\clr] (0,0.4) .. controls ++(0,0.35) and ++(0,-0.35) .. (-0.4,1)--(-0.4,1.4);
  \draw[thick, color=\clr] (0,0.4) .. controls ++(0,0.35) and ++(0,-0.35) .. (0.4,1)--(0.4,1.4);
          \node[left] at (-0.4,1) {$\scriptstyle x^{(1)}_{1}$};
          \node  at (0,1) {$\scriptstyle \cdots $};
                     \node[right] at (0.4,1) {$\scriptstyle x^{(1)}_{r_1}$};
            \node[below] at (0,-0.2) {$\scriptstyle a_1 $};
            \node[shape=coordinate](DOT) at (0,0.2) {};
     \draw[thick, color=\clr, fill=green]  (DOT) circle (5pt);
      \node[right] at (0.1,0.2) {$\scriptstyle d_1 $};
  %%%%%%%%%%
     \node  at (1.3,0.2) {$\scriptstyle \cdots $};
  %%%%%%%%%%
 \draw[thick, color=\clr] (2.6,-0.2)--(2.6,0.4);
 \draw[thick, color=\clr] (2.6,0.4) .. controls ++(0,0.35) and ++(0,-0.35) .. (2.2,1)--(2.2,1.4);
  \draw[thick, color=\clr] (2.6,0.4) .. controls ++(0,0.35) and ++(0,-0.35) .. (3,1)--(3,1.4);
          \node[left] at (2.25,1) {$\scriptstyle x^{(t)}_{1}$};
          \node  at (2.6,1) {$\scriptstyle \cdots $};
                     \node[right] at (3,1) {$\scriptstyle x^{(t)}_{r_t}$};
            \node[below] at (2.6,-0.2) {$\scriptstyle a_t $};
            \node[shape=coordinate](DOT) at (2.6,0.2) {};
     \draw[thick, color=\clr, fill=green]  (DOT) circle (5pt);
      \node[right] at (2.7,0.2) {$\scriptstyle d_t $};
  %%%%%%%%%%
   \draw[thick, color=\clr, fill=red] (-0.6,1.4)--(3.2,1.4)--(3.2,2.2)--(-0.6,2.2)--(-0.6,1.4);
    \node at (1.3,1.8) {$\scriptstyle X$};
    %%%%%%%%%
   %%%%%%%%%% 
     \draw[thick, color=\clr] (0,3.4)--(0,3.2);
 \draw[thick, color=\clr] (0,3.2) .. controls ++(0,-0.35) and ++(0,0.35) .. (-0.4,2.6)--(-0.4,2.2);
  \draw[thick, color=\clr] (0,3.2) .. controls ++(0,-0.35) and ++(0,0.35) .. (0.4,2.6)--(0.4,2.2);
          \node[left] at (-0.4,2.6) {$\scriptstyle y^{(1)}_{1}$};
          \node  at (0,2.6) {$\scriptstyle \cdots $};
                     \node[right] at (0.4,2.6) {$\scriptstyle y^{(1)}_{s_1}$};
            \node[above] at (0,3.4) {$\scriptstyle b_1 $};
  %%%%%%%%%%
     \node  at (1.3,3.4) {$\scriptstyle \cdots $};
  %%%%%%%%%%
 \draw[thick, color=\clr] (2.6,3.4)--(2.6,3.2);
 \draw[thick, color=\clr] (2.6,3.2) .. controls ++(0,-0.35) and ++(0,0.35) .. (2.2,2.6)--(2.2,2.2);
  \draw[thick, color=\clr] (2.6,3.2) .. controls ++(0,-0.35) and ++(0,0.35) .. (3,2.6)--(3,2.2);
          \node[left] at (2.25,2.6) {$\scriptstyle y^{(u)}_{1}$};
          \node  at (2.6,2.6) {$\scriptstyle \cdots $};
                     \node[right] at (3,2.6) {$\scriptstyle y^{(u)}_{s_u}$};
            \node[above] at (2.6,3.4) {$\scriptstyle b_u $};
      \end{tikzpicture}},
\end{align}
for some labels \(d_i, x_j^{(i)}, y_j^{(i)} \in \Z_{\geq 0}\), and \(X\) is a diagram composed only of crossings, cups, and caps, where no cup occurs below any cap.

Now we consider \cref{spandiag1}. Note the following:
\begin{enumerate}
\item If two strands which split from \(a_i\) meet in a cap in \(X\), then by \cref{CapDiagSwitchRel1} the diagram can be rewritten by adding one to \(d_i\), deleting the strand, and multiplying by 2.
\item If two strands which merge in \(b_i\) meet in a cup in \(X\), then by \cref{CupDiagSwitchRel2} the diagram is zero.
\item If for some \(i \neq j\), there is more than one instance of a strand which splits from \(a_i\) and a strand which splits from \(a_j\) meeting in a cap in \(X\), then the diagram is zero by  \cref{ZeroLem}.
\item If for some \(i \neq j\), there is more than one instance of a strand which merges into \(b_i\) and a strand which merges into \(b_j\) meeting in a cup in \(X\), then the diagram is zero by  \cref{ZeroLem}.
\item If \(d_i >1\), then the diagram is zero by  \cref{ZeroLem}.
\item If there is more than one strand in \(X\) which splits from \(a_i\) and merges into \(b_j\), then by \cref{DiagSwitchRel}, the diagram can be written with a single strand which splits from \(a_i\) and merges into \(b_j\), multiplied by some constant.
 \end{enumerate}
 
After rewriting as above, we have via  \cref{CrossAbsorb} that \cref{spandiag1} is equivalent to a constant multiple of some diagram of the form \(\xi^{(A,B,C,D)}\), completing the proof for $\pWeb$.

An entirely analogous argument applies for $\aWeb$. Since there are no cups, caps or antennas it is easier and we leave it to the reader.
\end{proof}

\subsection{Generating Sets for the Morphism Spaces of  \texorpdfstring{$\aWeb$}{a-Web} and \texorpdfstring{$\pWeb$}{p-Web}}  In this section we describe generating sets for the morphism spaces of $\aWeb$ and $\pWeb$ using only the operations of composition and $\k$-linear combinations, and not the monoidal product.  These generators (and the relations among them, given in \cref{L:relationsforgenerators}) are used to establish a Howe duality in \cite{DKM}.

Let \(\pWeb_m\) be the full subcategory of \(\pWeb\) consisting of objects \(\ba \in \Z^m_{\geq 0}\).   We emphasize that the monoidal product in $\pWeb$ does not preserve the subcategory $\pWeb_m$.  Hence, $\pWeb_{m}$ is a supercategory and does not inherit a monoidal structure.  For \(t \in \Z_{\geq 0}\), \(\ba \in \Z_{\geq 0}^m\), \(1 \leq r < s \leq m\), we define the following morphisms in \(\pWeb_m\):
\begin{align*}
e^{(t)}_{[r,s],\ba}
&:=
\hackcenter{}
\hackcenter{
\begin{tikzpicture}[scale=0.8]
\draw[thick, color=\clr] (0,0)--(0,1);
\node at (0.5, 0.5) {$\scriptstyle \cdots $};
\draw[thick, color=\clr] (1,0)--(1,1);
\draw[thick, color=\clr] (1.8,0)--(1.8,1);
\node at (2.3, 0.25) {$\scriptstyle \cdots $};
\node at (2.3, 0.75) {$\scriptstyle t $};
\draw[thick, color=\clr] (2.8,0)--(2.8,1);
\draw[thick, color=\clr] (3.6,0)--(3.6,1);
\node at (4.1, 0.5) {$\scriptstyle \cdots $};
\draw[thick, color=\clr] (4.6,0)--(4.6,1);
  \draw[thick, color=\clr] (3.6,0.1)--(3.6,0.2) .. controls ++(0,0.35) and ++(0,-0.35) .. (1,0.8)--(1,0.9);
  \node[below] at (0, 0) {$\scriptstyle a_1 $};
    \node[below] at (1, 0) {$\scriptstyle a_r $};
        \node[below] at (1.8, 0) {$\scriptstyle a_{r+1} $};
            \node[below] at (2.8, 0) {$\scriptstyle a_{s-1} $};
               \node[below] at (3.6, 0) {$\scriptstyle a_s $};
               \node[below] at (4.6, 0) {$\scriptstyle a_m $};
                 \node[above] at (0, 1) {$\scriptstyle a_1 $};
    \node[above] at (1,1) {$\scriptstyle a_r +t$\;\;};
        \node[above] at (1.8, 1) {$\scriptstyle a_{r+1} $};
            \node[above] at (2.8, 1) {$\scriptstyle a_{s-1} $};
               \node[above] at (3.6, 1) {\;\;$\scriptstyle a_s -t$};
               \node[above] at (4.6, 1) {$\scriptstyle a_m $};
\end{tikzpicture}},
\qquad
&f^{(t)}_{[r,s],\ba}
:=
\hackcenter{}
\hackcenter{
\begin{tikzpicture}[scale=0.8]
\draw[thick, color=\clr] (0,0)--(0,1);
\node at (0.5, 0.5) {$\scriptstyle \cdots $};
\draw[thick, color=\clr] (1,0)--(1,1);
\draw[thick, color=\clr] (1.8,0)--(1.8,1);
\node at (2.3, 0.25) {$\scriptstyle \cdots $};
\node at (2.3, 0.75) {$\scriptstyle t $};
\draw[thick, color=\clr] (2.8,0)--(2.8,1);
\draw[thick, color=\clr] (3.6,0)--(3.6,1);
\node at (4.1, 0.5) {$\scriptstyle \cdots $};
\draw[thick, color=\clr] (4.6,0)--(4.6,1);
  \draw[thick, color=\clr] (1,0.1)--(1,0.2) .. controls ++(0,0.35) and ++(0,-0.35) .. (3.6,0.8)--(3.6,0.9);
  \node[below] at (0, 0) {$\scriptstyle a_1 $};
    \node[below] at (1, 0) {$\scriptstyle a_r $};
        \node[below] at (1.8, 0) {$\scriptstyle a_{r+1} $};
            \node[below] at (2.8, 0) {$\scriptstyle a_{s-1} $};
               \node[below] at (3.6, 0) {$\scriptstyle a_s $};
               \node[below] at (4.6, 0) {$\scriptstyle a_m $};
                 \node[above] at (0, 1) {$\scriptstyle a_1 $};
    \node[above] at (1,1) {$\scriptstyle a_r -t$\;\;};
        \node[above] at (1.8, 1) {$\scriptstyle a_{r+1} $};
            \node[above] at (2.8, 1) {$\scriptstyle a_{s-1} $};
               \node[above] at (3.6, 1) {\;\;$\scriptstyle a_s +t$};
               \node[above] at (4.6, 1) {$\scriptstyle a_m $};
\end{tikzpicture}},\\
b_{[r,s],\ba}
&:=
\hackcenter{}
\hackcenter{
\begin{tikzpicture}[scale=0.8]
\draw[thick, color=\clr] (0,0)--(0,1);
\node at (0.5, 0.5) {$\scriptstyle \cdots $};
\draw[thick, color=\clr] (1,0)--(1,1);
\draw[thick, color=\clr] (1.8,0)--(1.8,1);
\node at (2.3, 0.25) {$\scriptstyle \cdots $};
\draw[thick, color=\clr] (2.8,0)--(2.8,1);
\draw[thick, color=\clr] (3.6,0)--(3.6,1);
\node at (4.1, 0.5) {$\scriptstyle \cdots $};
\draw[thick, color=\clr] (4.6,0)--(4.6,1);
  \draw[thick, color=\clr] (1,0.1)--(1,0.2) .. controls ++(0,0.6) and ++(0,0.6) .. (3.6,0.2)--(3.6,0.1);
  \node[below] at (0, 0) {$\scriptstyle a_1 $};
    \node[below] at (1, 0) {$\scriptstyle a_r $};
        \node[below] at (1.8, 0) {$\scriptstyle a_{r+1} $};
            \node[below] at (2.8, 0) {$\scriptstyle a_{s-1} $};
               \node[below] at (3.6, 0) {$\scriptstyle a_s $};
               \node[below] at (4.6, 0) {$\scriptstyle a_m $};
                 \node[above] at (0, 1) {$\scriptstyle a_1 $};
    \node[above] at (1,1) {$\scriptstyle a_r -1$\;\;};
        \node[above] at (1.8, 1) {$\scriptstyle a_{r+1} $};
            \node[above] at (2.8, 1) {$\scriptstyle a_{s-1} $};
               \node[above] at (3.6, 1) {\;\;$\scriptstyle a_s -1$};
               \node[above] at (4.6, 1) {$\scriptstyle a_m $};
                  \node at (2.3,0.65) {$\scriptstyle\blacklozenge$};
\end{tikzpicture}},
\qquad
&c_{[r,s],\ba}
:=
\hackcenter{}
\hackcenter{
\begin{tikzpicture}[scale=0.8]
\draw[thick, color=\clr] (0,0)--(0,1);
\node at (0.5, 0.5) {$\scriptstyle \cdots $};
\draw[thick, color=\clr] (1,0)--(1,1);
\draw[thick, color=\clr] (1.8,0)--(1.8,1);
\node at (2.3, 0.75) {$\scriptstyle \cdots $};
\draw[thick, color=\clr] (2.8,0)--(2.8,1);
\draw[thick, color=\clr] (3.6,0)--(3.6,1);
\node at (4.1, 0.5) {$\scriptstyle \cdots $};
\draw[thick, color=\clr] (4.6,0)--(4.6,1);
  \draw[thick, color=\clr] (1,0.9)--(1,0.8) .. controls ++(0,-0.6) and ++(0,-0.6) .. (3.6,0.8)--(3.6,0.9);
  \node[below] at (0, 0) {$\scriptstyle a_1 $};
    \node[below] at (1, 0) {$\scriptstyle a_r $};
        \node[below] at (1.8, 0) {$\scriptstyle a_{r+1} $};
            \node[below] at (2.8, 0) {$\scriptstyle a_{s-1} $};
               \node[below] at (3.6, 0) {$\scriptstyle a_s $};
               \node[below] at (4.6, 0) {$\scriptstyle a_m $};
                 \node[above] at (0, 1) {$\scriptstyle a_1 $};
    \node[above] at (1,1) {$\scriptstyle a_r +1$\;\;};
        \node[above] at (1.8, 1) {$\scriptstyle a_{r+1} $};
            \node[above] at (2.8, 1) {$\scriptstyle a_{s-1} $};
               \node[above] at (3.6, 1) {\;\;$\scriptstyle a_s +1$};
               \node[above] at (4.6, 1) {$\scriptstyle a_m $};
                  \node at (2.3,0.35) {$\scriptstyle\blacklozenge$};
\end{tikzpicture}}.
\end{align*}

For \(1 \leq u \leq m\), we define the additional morphism
\begin{align*}
b_{[u],\ba}:=\hackcenter{}
\hackcenter{
\begin{tikzpicture}[scale=0.8]
\draw[thick, color=\clr] (0,0)--(0,1);
\node at (0.5, 0.5) {$\scriptstyle \cdots $};
\draw[thick, color=\clr] (1,0)--(1,1);
\node at (1.5, 0.5) {$\scriptstyle \cdots $};
\draw[thick, color=\clr] (2,0)--(2,1);
     \node[shape=coordinate](DOT) at (1,0.5) {};
     \filldraw  (DOT) circle (2.5pt);
  \draw[thick, color=\clr, fill=green]  (DOT) circle (5pt);
  \node[below] at (0, 0) {$\scriptstyle a_1 $};
    \node[below] at (1, 0) {$\scriptstyle a_u $};
        \node[below] at (2, 0) {$\scriptstyle a_m $};
          \node[above] at (0, 1) {$\scriptstyle a_1 $};
    \node[above] at (1, 1) {$\scriptstyle a_u-2 $};
        \node[above] at (2, 1) {$\scriptstyle a_m $};
\end{tikzpicture}},
\end{align*}
where again the dot denotes a combination of a split and antenna, as in \cref{greendot}.

To remove some of the clutter in the calculations which follow, we will sometimes write products for compositions (e.g., $fg = f \circ g$) and will occasionally omit the label \(\ba\) when the domain is clear from context (e.g., writing \(e_{[r,s]}^{(t)}\) instead of \(e_{[r,s],\ba}^{(t)}\)).  Let us write $1_{\ba}$ for the identity morphism of $\ba$.  Pre- or post-composing by these gives a convenient alternate method for specifying the domain or range of a morphism.  For example, $e_{[r,s],\ba}^{(t)}=e_{[r,s]}^{(t)} 1_{\ba}$.

In order to establish a generating set for the morphisms in $\pWeb_m$, we need the following technical lemma.  Because we are no longer in a monoidal supercategory, we only use composition and $\k$-linear combinations when generating morphisms in this category.

%%%%
\begin{lemma}\label{efLem}
Label the set of morphisms:
\begin{align*}
Y_{\{e,f\}}(m):=
\left\{ \left.  e_{[r,s],\ba}^{(t)}, f_{[r,s],\ba}^{(t)} \; \right|  t \in \Z_{\geq 0}, 1 \leq r < s \leq m, \ba \in \Z_{\geq 0}^m \right\}.
\end{align*}
Let \(S_{\{e,f\}}(m)\) represent the \(\k\)-linear subcategory of \(\pWeb_m\) consisting of all objects in \(\pWeb_m\), with morphisms generated by \(Y_{\{e,f\}}(m)\). 
Then for all \(\ba, \bb \in \Z_{\geq 0}^m\), \(k \in \Z_{\geq 0}\) and \(i,j \in [1,m]\), and morphisms \(y \in S_{\{e,f\}}(m)\), the morphism
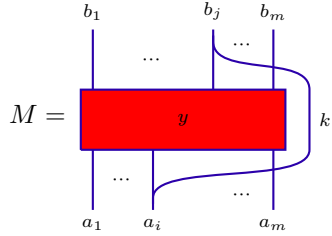
\begin{align*}
M=
\hackcenter{}
\hackcenter{
\begin{tikzpicture}[scale=0.8]
 \draw[thick, color=\clr] (0,0)--(0,1);
  \node at (0.5, 0.5) {$\scriptstyle \cdots $};
 \draw[thick, color=\clr] (1,0)--(1,1);
 \node at (2.5, 0.25) {$\scriptstyle \cdots $};
  \draw[thick, color=\clr] (3,0)--(3,1);
  %%%%%%%%%%
   \draw[thick, color=\clr] (0,2)--(0,3);
  \node at (1, 2.5) {$\scriptstyle \cdots $};
 \draw[thick, color=\clr] (2,2)--(2,3);
 \node at (2.5, 2.75) {$\scriptstyle \cdots $};
  \draw[thick, color=\clr] (3,2)--(3,3);
        %%%%%%%%%%
   \draw[thick, color=\clr, fill=red] (-0.2,1)--(3.2,1)--(3.2,2)--(-0.2,2)--(-0.2,1);
    \node at (1.5,1.5) {$\scriptstyle y$};
    %%%%%%%%%
   %%%%%%%%%% 
    \draw[thick, color=\clr] (1,0.1)--(1,0.2) .. controls ++(0,0.6) and ++(0,-0.6) .. (3.6,1)--(3.6,2)
    .. controls ++(0,0.6) and ++(0,-0.6) .. (2,2.8)--(2,2.9);
  %%%%%%%%%%
   \node[below] at (0, 0) {$\scriptstyle a_1 $};
    \node[below] at (1, 0) {$\scriptstyle a_i $};
    \node[below] at (3, 0) {$\scriptstyle a_m $};
    \node[above] at (0, 3) {$\scriptstyle b_1 $};
    \node[above] at (2, 3) {$\scriptstyle b_j $};
    \node[above] at (3, 3) {$\scriptstyle b_m $}; 
     \node[right] at (3.6, 1.5) {$\scriptstyle k $};  
   %%%%%%%%%%
      \end{tikzpicture}}
\end{align*}
is also in \(S_{\{e,f\}}(m)\).
\end{lemma}

\begin{proof}
We may assume that \(y\) is itself a composition of \(u\) many morphisms in \(Y_{\{e,f\}}(m)\). Our argument will go by nested induction on \(k\) and \(u\). First, note that if \(k=0\), the claim holds trivially, so we now fix \(k>0\) and assume that the claim holds for all \(y\) and \(k'<k\). 

Assume \(u=0\). If \(i=j\), then \(M\) is some multiple of the identity morphism by \cref{DiagSwitchRel}. If \(i>j\), then \(M=e_{j,i}^{(k)}\), and if \(i<j\), then \(M=f_{i,j}^{(k)}\). This proves the claim when \(u=0\). We now assume \(u>0\) and that the claim holds for all \(u'<u\). 

As \(u>0\), we may write \(y=y'e_{[r,s]}^{(t)}\) or \(y=y'f_{[r,s]}^{(t)}\) for some \(r,s,t\). We will assume the former, as the latter case is similar. If \(r \neq i\), then \(e_{[r,s]}^{(t)}\) moves freely below past the split on the \(i\)th strand in \(M\). Then the induction assumption on \(u\) may be used to complete the proof of the claim.

So we now assume that \(y=y'e_{[i,s]}^{(t)}\), so that \(M\) may be written:
\begin{align}\label{lemdiag5}
M=
\hackcenter{}
\hackcenter{
\begin{tikzpicture}[scale=0.8]
 \draw[thick, color=\clr] (0,-1)--(0,1);
  \node at (0.3, 0) {$\scriptstyle \cdots $};
 \draw[thick, color=\clr] (0.6,-1)--(0.6,1);
  \node at (1.5, 0.5) {$\scriptstyle \cdots $};
  \node at (1.5, -0.75) {$\scriptstyle \cdots $};
  \draw[thick, color=\clr] (2.4,-1)--(2.4,1);
  \node at (2.7, -0.5) {$\scriptstyle \cdots $};
  \draw[thick, color=\clr] (3,-1)--(3,1);
  %%%%%%%%%%
   \draw[thick, color=\clr] (0,2)--(0,3);
  \node at (1, 2.5) {$\scriptstyle \cdots $};
 \draw[thick, color=\clr] (2,2)--(2,3);
 \node at (2.5, 2.75) {$\scriptstyle \cdots $};
  \draw[thick, color=\clr] (3,2)--(3,3);
        %%%%%%%%%%
   \draw[thick, color=\clr, fill=red] (-0.2,1)--(3.2,1)--(3.2,2)--(-0.2,2)--(-0.2,1);
    \node at (1.5,1.5) {$\scriptstyle y'$};
    %%%%%%%%%
   %%%%%%%%%% 
    \draw[thick, color=\clr] (0.6,-.8)--(0.6,-0.7) .. controls ++(0,0.6) and ++(0,-1.2) .. (3.6,1.2)--(3.6,2)
    .. controls ++(0,0.6) and ++(0,-0.6) .. (2,2.8)--(2,2.9);
     \draw[thick, color=\clr] (2.4,-.8)--(2.4,-0.7) .. controls ++(0,0.6) and ++(0,-0.6) .. (0.6,0.8)--(0.6,0.9);
  %%%%%%%%%%
   \node[below] at (0, -1) {$\scriptstyle a_1 $};
    \node[below] at (0.6, -1) {$\scriptstyle a_i $};
     \node[below] at (2.4, -1) {$\scriptstyle a_s $};
    \node[below] at (3, -1) {$\scriptstyle a_m $};
    \node[above] at (0, 3) {$\scriptstyle b_1 $};
    \node[above] at (2, 3) {$\scriptstyle b_j $};
    \node[above] at (3, 3) {$\scriptstyle b_m $}; 
     \node[right] at (3.6, 1.5) {$\scriptstyle k $};  
   %%%%%%%%%%
      \end{tikzpicture}}.
\end{align}
Note that for clarity here we are omitting strands between the \(i\)th and \(s\)th strands.  Using \cref{C:pwebbraiding} any morphisms on strands between the $i$th and $s$th strands may be pulled all the way to the right side of the morphism $M$ by introducing crossings.  For this reason any morphisms between the $i$th and $s$th strands will not affect our calculations and can be safely ignored.

Using \cref{CrossDef}, \(M\) may be rewritten as a linear combination of diagrams of the form
\begin{align}\label{M6}
\hackcenter{}
\hackcenter{
\begin{tikzpicture}[scale=0.8]
 \draw[thick, color=\clr] (0,-1)--(0,1);
  \node at (0.3, 0) {$\scriptstyle \cdots $};
 \draw[thick, color=\clr] (0.6,-1)--(0.6,1);
  \node at (1.5, -0.75) {$\scriptstyle \cdots $};
  \draw[thick, color=\clr] (2.4,-1)--(2.4,1);
  \node at (2.7, -0.5) {$\scriptstyle \cdots $};
  \draw[thick, color=\clr] (3,-1)--(3,1);
  %%%%%%%%%%
   \draw[thick, color=\clr] (0,2)--(0,3);
  \node at (1, 2.5) {$\scriptstyle \cdots $};
 \draw[thick, color=\clr] (2,2)--(2,3);
 \node at (2.5, 2.75) {$\scriptstyle \cdots $};
  \draw[thick, color=\clr] (3,2)--(3,3);
        %%%%%%%%%%
   \draw[thick, color=\clr, fill=red] (-0.2,1)--(3.2,1)--(3.2,2)--(-0.2,2)--(-0.2,1);
    \node at (1.5,1.5) {$\scriptstyle y'$};
    %%%%%%%%%
   %%%%%%%%%% 
     \draw[thick, color=\clr] (0.6,-.9) .. controls ++(0,0.3) and ++(0,-0.3) .. (1.2,-0.3)--(1.2,0.4)
     .. controls ++(0,0.3) and ++(0,-0.3) .. (0.6,0.8)--(0.6,0.9);
     \draw[thick, color=\clr] (2.4,-.9) .. controls ++(0,0.3) and ++(0,-0.3) .. (1.8,-0.3)--(1.8,0.4)
     .. controls ++(0,0.6) and ++(0,-0.6) .. (3.6,1.1)--(3.6,2)
    .. controls ++(0,0.6) and ++(0,-0.6) .. (2,2.8)--(2,2.9);
    \draw[thick, color=\clr] (1.8,0) .. controls ++(0,0.3) and ++(0,-0.3) .. (1.2,0.4);
     \draw[thick, color=\clr] (1.2,-0.3) .. controls ++(0,0.3) and ++(0,-0.3) .. (1.8,0.1);
  %%%%%%%%%%
   \node[below] at (0, -1) {$\scriptstyle a_1 $};
    \node[below] at (0.6, -1) {$\scriptstyle a_i $};
     \node[below] at (2.4, -1) {$\scriptstyle a_s $};
    \node[below] at (3, -1) {$\scriptstyle a_m $};
    \node[above] at (0, 3) {$\scriptstyle b_1 $};
    \node[above] at (2, 3) {$\scriptstyle b_j $};
    \node[above] at (3, 3) {$\scriptstyle b_m $}; 
     \node[right] at (3.6, 1.5) {$\scriptstyle k $};  
   %%%%%%%%%%
      \end{tikzpicture}}.
\end{align}
Using \cref{AssocRel,DiagSwitchRel}, any diagram as in \cref{M6} can be written as a linear combination of diagrams of the form 
\begin{align}\label{M7}
\hackcenter{}
\hackcenter{
\begin{tikzpicture}[scale=0.8]
 \draw[thick, color=\clr] (0,-1)--(0,1);
  \node at (0.3, 0) {$\scriptstyle \cdots $};
 \draw[thick, color=\clr] (0.6,-1)--(0.6,1);
  \node at (1.5, -0.75) {$\scriptstyle \cdots $};
  \draw[thick, color=\clr] (2.4,-1)--(2.4,1);
  \node at (2.7, -0.5) {$\scriptstyle \cdots $};
  \draw[thick, color=\clr] (3,-1)--(3,1);
  %%%%%%%%%%
   \draw[thick, color=\clr] (0,2)--(0,3);
  \node at (1, 2.5) {$\scriptstyle \cdots $};
 \draw[thick, color=\clr] (2,2)--(2,3);
 \node at (2.5, 2.75) {$\scriptstyle \cdots $};
  \draw[thick, color=\clr] (3,2)--(3,3);
        %%%%%%%%%%
   \draw[thick, color=\clr, fill=red] (-0.2,1)--(3.2,1)--(3.2,2)--(-0.2,2)--(-0.2,1);
    \node at (1.5,1.5) {$\scriptstyle y'$};
    %%%%%%%%%
   %%%%%%%%%% 
     \draw[thick, color=\clr] (2.4,-.9) .. controls ++(0,0.3) and ++(0,-0.3) .. (1.8,-0.3)--(1.8,0.4)
     .. controls ++(0,0.6) and ++(0,-0.6) .. (3.6,1.1)--(3.6,2)
    .. controls ++(0,0.6) and ++(0,-0.6) .. (2,2.8)--(2,2.9);
    \draw[thick, color=\clr] (1.8,0.2) .. controls ++(0,0.3) and ++(0,-0.3) .. (0.6,0.7);
     \draw[thick, color=\clr] (0.6,-0.6) .. controls ++(0,0.3) and ++(0,-0.3) .. (1.8,-0.1);
  %%%%%%%%%%
   \node[below] at (0, -1) {$\scriptstyle a_1 $};
    \node[below] at (0.6, -1) {$\scriptstyle a_i $};
     \node[below] at (2.4, -1) {$\scriptstyle a_s $};
    \node[below] at (3, -1) {$\scriptstyle a_m $};
    \node[above] at (0, 3) {$\scriptstyle b_1 $};
    \node[above] at (2, 3) {$\scriptstyle b_j $};
    \node[above] at (3, 3) {$\scriptstyle b_m $}; 
     \node[right] at (3.6, 1.5) {$\scriptstyle k $};  
   %%%%%%%%%%
      \end{tikzpicture}}.
\end{align}
An application of \cref{DiagSwitchRel} allows us to write any diagram as in \cref{M7} as a linear combination of diagrams of the form
\begin{align}\label{M8}
\hackcenter{}
\hackcenter{
\begin{tikzpicture}[scale=0.8]
 \draw[thick, color=\clr] (0,-1)--(0,1);
  \node at (0.3, 0) {$\scriptstyle \cdots $};
 \draw[thick, color=\clr] (0.6,-1)--(0.6,1);
  \node at (1.5, -0.75) {$\scriptstyle \cdots $};
  \draw[thick, color=\clr] (2.4,-1)--(2.4,1);
  \node at (2.7, -0.5) {$\scriptstyle \cdots $};
  \draw[thick, color=\clr] (3,-1)--(3,1);
  %%%%%%%%%%
   \draw[thick, color=\clr] (0,2)--(0,3);
  \node at (1, 2.5) {$\scriptstyle \cdots $};
 \draw[thick, color=\clr] (2,2)--(2,3);
 \node at (2.5, 2.75) {$\scriptstyle \cdots $};
  \draw[thick, color=\clr] (3,2)--(3,3);
        %%%%%%%%%%
   \draw[thick, color=\clr, fill=red] (-0.2,1)--(3.2,1)--(3.2,2)--(-0.2,2)--(-0.2,1);
    \node at (1.5,1.5) {$\scriptstyle y'$};
    %%%%%%%%%
   %%%%%%%%%% 
     \draw[thick, color=\clr] (2.4,-.9) .. controls ++(0,0.3) and ++(0,-0.3) .. (1.8,-0.4)--(1.8,0.4)
     .. controls ++(0,0.6) and ++(0,-0.6) .. (3.6,1.1)--(3.6,2)
    .. controls ++(0,0.6) and ++(0,-0.6) .. (2,2.8)--(2,2.9);
    \draw[thick, color=\clr] (0.6,0.1) .. controls ++(0,0.3) and ++(0,-0.3) .. (1.8,0.4);
     \draw[thick, color=\clr] (1.8,-0.4) .. controls ++(0,0.3) and ++(0,-0.3) .. (0.6,-0.1);
  %%%%%%%%%%
   \node[below] at (0, -1) {$\scriptstyle a_1 $};
    \node[below] at (0.6, -1) {$\scriptstyle a_i $};
     \node[below] at (2.4, -1) {$\scriptstyle a_s $};
    \node[below] at (3, -1) {$\scriptstyle a_m $};
    \node[above] at (0, 3) {$\scriptstyle b_1 $};
    \node[above] at (2, 3) {$\scriptstyle b_j $};
    \node[above] at (3, 3) {$\scriptstyle b_m $}; 
     \node[right] at (3.6, 1.5) {$\scriptstyle k $};  
   %%%%%%%%%%
      \end{tikzpicture}}
      \;\;
\substack{ \textup{(\ref{AssocRel})} \\ =\\ \,}
\;\;
      \hackcenter{
\begin{tikzpicture}[scale=0.8]
 \draw[thick, color=\clr] (0,-1)--(0,1);
  \node at (0.3, 0) {$\scriptstyle \cdots $};
 \draw[thick, color=\clr] (0.6,-1)--(0.6,1);
  \node at (1.5, 0) {$\scriptstyle \cdots $};
  \draw[thick, color=\clr] (2.4,-1)--(2.4,1);
  \node at (2.7, -0.5) {$\scriptstyle \cdots $};
  \draw[thick, color=\clr] (3,-1)--(3,1);
  %%%%%%%%%%
   \draw[thick, color=\clr] (0,2)--(0,3);
  \node at (1, 2.5) {$\scriptstyle \cdots $};
 \draw[thick, color=\clr] (2,2)--(2,3);
 \node at (2.5, 2.75) {$\scriptstyle \cdots $};
  \draw[thick, color=\clr] (3,2)--(3,3);
        %%%%%%%%%%
   \draw[thick, color=\clr, fill=red] (-0.2,1)--(3.2,1)--(3.2,2)--(-0.2,2)--(-0.2,1);
    \node at (1.5,1.5) {$\scriptstyle y'$};
    %%%%%%%%%
   %%%%%%%%%% 
     \draw[thick, color=\clr] (0.6,0.1)
     .. controls ++(0,0.6) and ++(0,-0.8) .. (3.6,1.1)--(3.6,1.8)
    .. controls ++(0,0.6) and ++(0,-0.3) .. (2,2.5);
    \draw[thick, color=\clr] (2.4,0)
     .. controls ++(0,0.6) and ++(0,-0.6) .. (4.4,0.9)--(4.4,2)
    .. controls ++(0,0.6) and ++(0,-0.6) .. (2,2.8);
     \draw[thick, color=\clr] (2.4,-0.9) .. controls ++(0,0.6) and ++(0,-0.6) .. (0.6,-0.1);
  %%%%%%%%%%
   \node[below] at (0, -1) {$\scriptstyle a_1 $};
    \node[below] at (0.6, -1) {$\scriptstyle a_i $};
     \node[below] at (2.4, -1) {$\scriptstyle a_s $};
    \node[below] at (3, -1) {$\scriptstyle a_m $};
    \node[above] at (0, 3) {$\scriptstyle b_1 $};
    \node[above] at (2, 3) {$\scriptstyle b_j $};
    \node[above] at (3, 3) {$\scriptstyle b_m $}; 
     \node[right] at (3.6, 1.5) {$\scriptstyle k' $};  
     \node[right] at (4.4, 1.5) {$\scriptstyle k'' $};  
   %%%%%%%%%%
      \end{tikzpicture}},
\end{align}
where \(k'+k''=k\). If \(k' = 0\) or \(k''=0\), then the claim follows by the induction assumption on \(u\). If \(k',k''>0\), then applying the induction assumption on \(k\) to the \(k'\) strand, and subsequently to the \(k''\) strand proves the claim, and completes the proof.
\end{proof}

Given $\bba = (a_{1}, \dotsc , a_{r}) \in \Z^{r}$, we write $|\bba | = \sum_{i=1}^{r} |a_{i}|$.

\begin{lemma}\label{aWebGen}
Morphisms in \(\aWeb_m\) are generated under composition by \(Y_{\{e,f\}}(m)\).
\end{lemma}
\begin{proof}
Let \(\ba,\bb \in \Z^m_{\geq 0}\) and let \(\xi:=\xi^{(0,0,C,0)} \in \mathscr{B}\) be an element in \(\Hom_{\aWeb}(\ba,\bb)\) as in \cref{webspan}. We show by inducting on \(n=|\ba|+|\bb|\) that \(\xi\) belongs to the set  \(S_{\{e,f\}}(m)\) from \cref{efLem}.  Since by \cref{webspan} such elements \(\xi\) span \(\Hom_{\aWeb}(\ba,\bb)\), this will prove the lemma.

If \(n=0\), then \(\xi\) is the identity morphism. Thus we may assume \(n>0\), and that the claim holds for all \(n'<n\). If \(C=0\), then \(\xi\) is the identity morphism, so assume \(C_{ij} >0\) for some \(i,j\). Then, using \cref{AssocRel}, we may write
\begin{align*}
\xi=
\hackcenter{}
\hackcenter{
\begin{tikzpicture}[scale=0.8]
 \draw[thick, color=\clr] (0,0)--(0,1);
  \node at (0.5, 0.5) {$\scriptstyle \cdots $};
 \draw[thick, color=\clr] (1,0)--(1,1);
 \node at (2.5, 0.25) {$\scriptstyle \cdots $};
  \draw[thick, color=\clr] (3,0)--(3,1);
  %%%%%%%%%%
   \draw[thick, color=\clr] (0,2)--(0,3);
  \node at (1, 2.5) {$\scriptstyle \cdots $};
 \draw[thick, color=\clr] (2,2)--(2,3);
 \node at (2.5, 2.75) {$\scriptstyle \cdots $};
  \draw[thick, color=\clr] (3,2)--(3,3);
        %%%%%%%%%%
   \draw[thick, color=\clr, fill=red] (-0.2,1)--(3.2,1)--(3.2,2)--(-0.2,2)--(-0.2,1);
    \node at (1.5,1.5) {$\scriptstyle \xi'$};
    %%%%%%%%%
   %%%%%%%%%% 
    \draw[thick, color=\clr] (1,0.1)--(1,0.2) .. controls ++(0,0.6) and ++(0,-0.6) .. (3.6,1)--(3.6,2)
    .. controls ++(0,0.6) and ++(0,-0.6) .. (2,2.8)--(2,2.9);
  %%%%%%%%%%
   \node[below] at (0, 0) {$\scriptstyle a_1 $};
    \node[below] at (1, 0) {$\scriptstyle a_i $};
    \node[below] at (3, 0) {$\scriptstyle a_m $};
    \node[above] at (0, 3) {$\scriptstyle b_1 $};
    \node[above] at (2, 3) {$\scriptstyle b_j $};
    \node[above] at (3, 3) {$\scriptstyle b_m $}; 
     \node[right] at (3.6, 1.5) {$\scriptstyle C_{ij} $};  
   %%%%%%%%%%
      \end{tikzpicture}},
\end{align*}
for some basis element \(\xi' \in \Hom_{\pWeb}(\ba',\bb')\), where \(|\ba'|<|\ba|\) and \(|\bb'|<|\bb|\). Applying the induction assumption on \(n\), we have that \(\xi' \in S_{\{e,f\}}(m)\). Then, applying \cref{efLem}, it follows that \(\xi \in S_{\{e,f\}}(m)\), as desired.
\end{proof}

\begin{theorem}\label{Ygen}
Morphisms in \(\pWeb_m\) are generated under composition by the set \(\{b_{[1],a} \mid a \in \ZZ_{\geq 0}\}\) when \(m =1\), and by
\begin{align*}
Y(m):=
\left\{ \left. e_{[i,i+1],\ba}^{(t)}, f_{[i,i+1],\ba}^{(t)}, b_{[1,2],\ba}, c_{[1,2],\ba} \; \right| t \in \Z_{\geq 0}, 1 \leq i < m,  \ba \in \Z_{\geq 0}^m \right\},
\end{align*}
when \(m\geq 2\).
\end{theorem}
\begin{proof}
The statement for \(m=1\) follows directly from \cref{webspan}, so assume \(m \geq 2\).
First, let
\begin{align*}
Y'(m):=
\left\{ \left.  e_{[r,s],\ba}^{(t)}, f_{[r,s],\ba}^{(t)}, b_{[r,s],\ba}, c_{[r,s],\ba}, b_{[u],\ba} \; \right| t \in \Z_{\geq 0}, 1 \leq r < s \leq m, u \in [1,m], \ba \in \Z_{\geq 0}^m \right\}.
\end{align*}
We first prove a preliminary claim; that morphisms in \(\pWeb_m\) are generated under composition by \(Y'(m)\). Write \(S_m\) (resp.\  \(S'_m\)) for the subcategory of \(\pWeb_m\) generated by morphisms in \(Y(m)\) (resp.\  \(Y'(m)\)).
Let \(\ba,\bb \in \Z^m_{\geq 0}\). By \cref{webspan}, \(\Hom_{\pWeb}(\ba,\bb)\) is spanned by elements of the form \(\xi:=\xi^{(A,B,C,D)} \in \mathscr{B}\). We show by induction on \(n=|\ba|+|\bb|\) that \(\xi \in S'_m\). 

If \(n=0\), then \(\xi\) is the identity morphism. Thus we may assume \(n>0\), and that the claim holds for all \(n'<n\). Note the following:
\begin{itemize}
\item If \(A_{ij} >0\) for any \(i<j\), then by \cref{AssocRel}, \(\xi = \xi'b_{[i,j],\ba}\) for some basis element \(\xi' \in \Hom_{\pWeb}(\ba', \bb)\), where \(|\ba'|<|\ba|\). 
\item If \(B_{ij} >0\) for any \(i<j\), then by \cref{AssocRel}, \(\xi = c_{[i,j],\bb'}\xi'\) for some basis element \(\xi' \in \Hom_{\pWeb}(\ba, \bb')\), where \(|\bb'|<|\bb|\). 
\item If \(D_{i} >0\) for any \(i\), then by \cref{AssocRel}, \(\xi = \xi'b_{[i],\ba}\) for some basis element \(\xi' \in \Hom_{\pWeb}(\ba', \bb)\), where \(|\ba'|<|\ba|\). 
\end{itemize}
In any of these cases, applying the induction assumption on \(n\) to \(\xi'\) completes the proof. Therefore we may assume \(A=B=D=0\). Then the preliminary claim follows by \cref{aWebGen}, so morphisms in \(\pWeb_m\) are generated under composition by \(Y'(m)\).

Since \(S_m\) contains \(e_{[i,i+1],\ba}^{(t)}\) and \(f_{[i,i+1],\ba}^{(t)}\) for all \(1 \leq i < m, \ba \in \ZZ_{\geq 0}^m\), it follows by \cref{CrossDef} that \(S_m\) contains all crossing morphisms which transpose neighboring strands. Then we have that \(e_{[r,s],\ba}^{(t)}\), \(f_{[r,s],\ba}^{(t)}\), \(b_{[r,s],\ba}\), \(c_{[r,s],\ba}\) belong to \(S_m\) for all \(1 \leq r < s \leq m\), \(\ba \in \ZZ^m_{\geq 0}\), as one may generate these elements by pre- and post-composing the morphisms \(e_{[1,2],\ba}^{(t)}, f_{[1,2],\ba}^{(t)}, b_{[1,2],\ba}, c_{[1,2],\ba}\) with sequences of crossing morphisms. Finally, we have that \(b_{[u],\ba} = (1/2) [ b_{[u,u+1]}, e_{[u,u+1]}^{(1)}]1_{\ba}\) for \(1 \leq u < m\) and \(b_{[m],\ba} = (1/2) [ b_{[m-1,m]}, f_{[m-1,m]}^{(1)}]1_{\ba}\) by \cref{CapDiagSwitchRel1,CapDiagSwitchRel2}, so it follows that \(S_m' \subseteq S_m\), completing the proof.
\end{proof}

\begin{corollary}
If \(\k\) is a field of characteristic zero, then morphisms in \(\pWeb_m\) are generated under composition by the set \(\{b_{[1],a} \mid a \in \ZZ_{\geq 0}\}\) when \(m =1\), and by
\begin{align*}
Y_0(m):=
\left\{ \left. e_{[i,i+1],\ba}^{(1)}, f_{[i,i+1],\ba}^{(1)}, b_{[1,2],\ba}, c_{[1,2],\ba}\;  \right| t \in \Z_{\geq 0}, 1 \leq i < m,  \ba \in \Z_{\geq 0}^m \right\},
\end{align*}
when \(m \geq 2\).
\end{corollary}
\begin{proof}
It can be deduced from \cref{AssocRel,KnotholeRel} (or alternatively, using (2.9) in \cite{CKM} at $q = 1$) that \((e_{[i,i+1]}^{(1)})^t1_\ba\) and \((f_{[i,i+1]}^{(1)})^t1_\ba\) are nonzero multiples of \(e_{[i,i+1],\ba}^{(t)}\) and \(f_{[i,i+1],\ba}^{(t)}\), respectively. Thus the result follows from \cref{Ygen}.
\end{proof}

\subsection{Relations for Morphisms in \texorpdfstring{\(\pWeb_m\)}{p-Webm}.}
We now establish a number of relations which hold among the generators in \(Y(m)\). While not utilized in this paper, these relations will be key in establishing the Howe duality result in \cite{DKM}.

\begin{lemma} \label{L:relationsforgenerators}
The following relations hold in \(\pWeb_m\), for all valid \(1 \leq i, j \leq m\), and \(\ba \in \ZZ_{\geq 0}^m\). 
\begin{align}
&[e_{[i,i+1]}^{(1)}, f_{[j,j+1]}^{(1)}]1_\ba = \delta_{i,j}(a_i - a_{i+1})1_\ba;\label{ef0}\\
&[e_{[i,i+1]}^{(1)}, e_{[j,j+1]}^{(1)}]1_\ba = 0 &\textup{if } j \neq i \pm 1;\label{ser1}\\
&[e_{[i,i+1]}^{(1)},[e_{[i,i+1]}^{(1)},e_{[j,j+1]}^{(1)}]1_\ba = 0 &\textup{if }j = i \pm 1;\label{ser2}\\
&[f_{[i,i+1]}^{(1)}, f_{[j,j+1]}^{(1)}]1_\ba = 0 &\textup{if } j \neq i \pm 1;\label{ser3}\\
&[f_{[i,i+1]}^{(1)},[f_{[i,i+1]}^{(1)},f_{[j,j+1]}^{(1)}]1_\ba = 0 &\textup{if }j = i \pm 1;\label{ser4}\\
&[b_{[i,i+1]},b_{[j,j+1]}]1_\ba =0;\label{bbrel}\\
&[c_{[i,i+1]},c_{[j,j+1]}]1_\ba =0;\label{ccrel}\\
&[b_{[i,i+1]},c_{[j,j+1]}]1_\ba
=
\begin{cases}
(a_i - a_{i+1})1_\ba & \textup{if \(j=i\)};\\
[e_{[i-1,i]},e_{[i,i+1]}]1_\ba & \textup{if \(j=i-1\)};\\
[f_{[i,i+1]},f_{[i+1,i+2]}]1_\ba &\textup{if \(j=i+1\)};\\
0 &\textup{otherwise}
\end{cases}\label{bcswaprel}\\
&[b_{[i,i+1]},e_{[j,j+1]}^{(1)}]1_\ba = 0 & \textup{if }j \neq i,i+1;\label{beswapeasy}\\
&[b_{[i,i+1]},f_{[j,j+1]}^{(1)}]1_\ba = 0 & \textup{if }j \neq i,i-1; \label{bfswapeasy}\\
&[b_{[i,i+1]},e_{[i,i+1]}^{(1)}]1_\ba =  2b_{[i+1]}    = [b_{[i+1,i+2]},f_{[i+1,i+2]}^{(1)}]1_\ba;
\label{be1rel}
\\
&[b_{[i,i+1]},e_{[i+1,i+2]}^{(1)}]1_\ba = b_{[i,i+2]} =  [b_{[i+1,i+2]},f_{[i,i+1]}^{(1)}]1_\ba; \label{be2rel}\\
&[e_{[j,j+1]}^{(1)},c_{[i,i+1]}]1_\ba = 0 & \textup{if }j \neq i-1; \label{ce1}\\
&[f_{[j,j+1]}^{(1)},c_{[i,i+1]}]1_\ba = 0 & \textup{if }j \neq i+1;\label{cf1}\\
&[e_{[i,i+1]}^{(1)},c_{[i+1,i+2]}]1_\ba =  c_{[i,i+2]}  = [f_{[i+1,i+2]}^{(1)},c_{[i,i+1]}]1_\ba; \label{ec1rel}\\
&[[b_{[i,i+1]},e_{[i,i+1]}^{(1)}],e_{[j,j+1]}^{(1)}]1_\ba =
\begin{cases}
2b_{[i+1,i+2]} 1_\ba & \textup{if }j=i+1;\\
0&\textup{otherwise};
\end{cases}
\label{bee1rel}
\\
&[[b_{[1,2]},e_{[1,2]}^{(1)}], f_{[1,2]}^{(1)}] 1_\ba= 2b_{[1,2]}1_\ba; \label{bef1rel}\\
&[[b_{[1,2]},f_{[1,2]}^{(1)}], f_{[j,j+1]}^{(1)}] 1_\ba= 0; \label{bff1rel} \\
&[[c_{[i+1,i]},e_{[i,i+1]}^{(1)}],e_{[j,j+1]}^{(1)}] 1_\ba = 
\begin{cases}
c_{[i,i+1]}1_\ba & \textup{if }j=i+1;\\
0 & \textup{if }j \neq i+1,i-1.
\end{cases}
\label{cee2rel}
\end{align}
\end{lemma}
\begin{proof}
These are straightforward calculations which follow quickly from relations already known to hold in \(\pWeb\). In particular:
\begin{itemize}
	\item Relation
\cref{ef0} follows from \cref{DiagSwitchRel};
	\item Relations
\cref{ser1,ser2,ser3,ser4} follow from \cref{CoxeterWeb};
	\item Relations 
\cref{bbrel,ccrel,beswapeasy,bfswapeasy} follow from \cref{AssocRel}; 
\item Relation
\cref{bcswaprel} follows from \cref{BCswap,BCE}; 
	\item Relation
\cref{be1rel} follows from \cref{CapDiagSwitchRel1,CapDiagSwitchRel2}; 
\item Relations \cref{be2rel,ec1rel} follow from  \cref{RungsToCross};
	\item Relations
\cref{ce1,cf1} follow from  \cref{CupDiagSwitchRel1,CupDiagSwitchRel2}; 
	\item Relations 
\cref{bee1rel,bef1rel,bff1rel} follow from \cref{be1rel,CapDiagSwitchRel1,CapDiagSwitchRel2,AntRel}; 
	\item Relation \cref{cee2rel} follows from \cref{ec1rel,RungsToCross}. \qedhere
\end{itemize}
\end{proof}

%%%%%%%%%%%%%%%%%%%%%%%%%%%%%%%%%%%%%%

\section{The  \texorpdfstring{$\pWeb_{\uparrow\downarrow}$}{p-Webupdown} Category}\label{S:Pwebupdowndefinition}

\subsection{Definition of  \texorpdfstring{$\pWeb_{\uparrow\downarrow}$}{p-Webupdown}}\label{SS:DefinitionofPWebupdown}  We now introduce an oriented version of $\pWeb$.  In this section we continue to assume $\k$ is an integral domain in which $2$ is invertible.  It will again be a diagrammatic supercategory given by generators and relations as outlined in \cref{SS:aWeb}.   Because many of the constructions and arguments in this section are similar to those given in the previous section, so we will sometimes be brief in explanations.

\begin{definition}
The category \(\pWeb_{\uparrow \downarrow}\) is the $\k$-linear strict monoidal supercategory defined as follows.  Objects are words (including the empty word) from the set
\begin{align*}
\left\{ \uparrow_k, \downarrow_k \mid k \in \Z_{\geq 1}\right\}.
\end{align*}

%Concatenation provides the monodial structure on objects.

The morphisms are generated as a $\k$-linear monoidal supercategory by the diagrams:
\begin{align*}
\hackcenter{
{}
}
\hackcenter{
\begin{tikzpicture}[scale=0.8]
\draw[thick, color=\clr,->] (0,-0.1)--(0,0.2);
  \draw[thick, color=\clr,->] (0,0)--(0,0.2) .. controls ++(0,0.35) and ++(0,-0.35) .. (-0.4,0.9)--(-0.4,1);
  \draw[thick, color=\clr,->] (0,0)--(0,0.2) .. controls ++(0,0.35) and ++(0,-0.35) .. (0.4,0.9)--(0.4,1);
      \node[above] at (-0.4,1) {$\scriptstyle a$};
      \node[above] at (0.4,1) {$\scriptstyle b$};
      \node[below] at (0,-0.1) {$\scriptstyle a+b$};
\end{tikzpicture}},
\qquad
\qquad
\hackcenter{
\begin{tikzpicture}[scale=0.8]
 \draw[thick, color=\clr,->] (-0.4,-0.1)--(-0.4,0.2);
  \draw[thick, color=\clr,->] (0.4,-0.1)--(0.4,0.2);
  \draw[thick, color=\clr,->] (-0.4,0)--(-0.4,0.1) .. controls ++(0,0.35) and ++(0,-0.35) .. (0,0.8)--(0,1);
\draw[thick, color=\clr,->] (0.4,0)--(0.4,0.1) .. controls ++(0,0.35) and ++(0,-0.35) .. (0,0.8)--(0,1);
      \node[below] at (-0.4,-0.1) {$\scriptstyle a$};
      \node[below] at (0.4,-0.1) {$\scriptstyle b$};
      \node[above] at (0,1) {$\scriptstyle a+b$};
\end{tikzpicture}},
\qquad
\qquad
\hackcenter{
\begin{tikzpicture}[scale=0.8]
 \draw(0,0) arc (0:180:0.4cm) [thick, color=\clr,->];
    \node[below] at (0,0) {$\scriptstyle a$};
      \node[below] at (-0.8,0) {$\scriptstyle a$};
\end{tikzpicture}},
\qquad
\qquad
\hackcenter{
\begin{tikzpicture}[scale=0.8]
 \draw(0,0) arc (0:-180:0.4cm) [thick, color=\clr,->];
    \node[above] at (0,0) {$\scriptstyle a$};
      \node[above] at (-0.8,0) {$\scriptstyle a$};
\end{tikzpicture}},
\qquad
\qquad
\hackcenter{
\begin{tikzpicture}[scale=0.8]
 \draw[thick, color=\clr,->] (0,0)--(0,0.4);
  \draw[thick, color=\clr,<-] (0,0.6)--(0,1);
    \node[below] at (0,0) {$\scriptstyle 1$};
     \node[above] at (0,1) {$\scriptstyle 1$};
 \node[shape=coordinate](DOT) at (0,0.5) {};
  \draw[thick, color=\clr, fill=yellow]  (DOT) circle (1mm);
\end{tikzpicture}},
\qquad
\qquad
\hackcenter{
\begin{tikzpicture}[scale=0.8]
 \draw[thick, color=\clr,<-] (0,0)--(0,0.4);
  \draw[thick, color=\clr,->] (0,0.6)--(0,1);
    \node[below] at (0,0) {$\scriptstyle 1$};
     \node[above] at (0,1) {$\scriptstyle 1$};
 \node[shape=coordinate](DOT) at (0,0.5) {};
  \draw[thick, color=\clr, fill=blue]  (DOT) circle (1mm);
\end{tikzpicture}},
\qquad
\qquad
\hackcenter{
\begin{tikzpicture}[scale=0.8]
  \draw[thick, color=\clr,<-] (0.4,0)--(0.4,0.1) .. controls ++(0,0.35) and ++(0,-0.35) .. (-0.4,0.9)--(-0.4,1);
  \draw[thick, color=\clr,->] (-0.4,0)--(-0.4,0.1) .. controls ++(0,0.35) and ++(0,-0.35) .. (0.4,0.9)--(0.4,1);
      \node[above] at (-0.4,1) {$\scriptstyle b$};
      \node[above] at (0.4,1) {$\scriptstyle a$};
            \node[below] at (-0.4,0) {$\scriptstyle a$};
      \node[below] at (0.4,0) {$\scriptstyle b$};
\end{tikzpicture}},
\end{align*}
for all \(a,b \in \Z_{\geq 0}\). We call these morphisms {\em upward split}, {\em upward merge}, {\em leftward cap}, {\em leftward cup}, {\em tag-in}, {\em tag-out}, and {\em rightward crossing}, respectively. The parity is given by declaring the tag-in and tag-out morphisms to be odd and the rest of the generating morphisms to be even.

%These generate the morphisms of $\pWeb_{\uparrow \downarrow}$ in the same sense as $\pWeb$ and $\aWeb$.

To describe the defining relations it will be convenient to first set a diagrammatic shorthand for certain additional morphisms in \(\pWeb_{\uparrow \downarrow}\). We define morphisms:
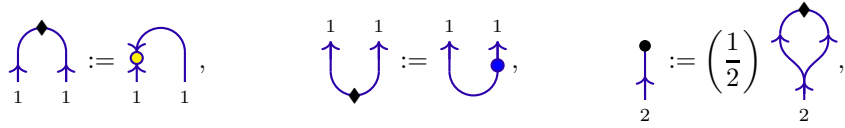
\begin{align*}
\hackcenter{}
\hackcenter{
\begin{tikzpicture}[scale=0.8]
 \draw (0,0) arc (0:180:0.4cm) [thick, color=\clr];
   \draw[thick, color=\clr] (-0.8,-0.5)--(-0.8,0);
     \draw[thick, color=\clr] (0,-0.5)--(0,0);
      \draw[thick, color=\clr,->] (-0.8,-0.5)--(-0.8,-0.2);
     \draw[thick, color=\clr,->] (0,-0.5)--(0,-0.2);
         \node[below] at (0,-0.5) {$\scriptstyle 1$};
      \node[below] at (-0.8,-0.5) {$\scriptstyle 1$};
  \node at (-0.4,0.4) {$\scriptstyle\blacklozenge$};
\end{tikzpicture}}
:=
\hackcenter{
\begin{tikzpicture}[scale=0.8]
 \node[white] at (-0.4,0.4) {$\scriptstyle\blacklozenge$};
 \draw(0,0) arc (0:180:0.4cm) [thick, color=\clr,->];
    \node[below] at (0,-0.5) {$\scriptstyle 1$};
      \node[below] at (-0.8,-0.5) {$\scriptstyle 1$};
      \draw[thick, color=\clr,->] (-0.8,-0.5)--(-0.8,-0.2);
      \draw[thick, color=\clr] (0,-0.5)--(0,0);
       \node[shape=coordinate](DOT) at (-0.8,-0.1) {};
  \draw[thick, color=\clr, fill=yellow]  (DOT) circle (1mm);
\end{tikzpicture}},
\qquad
\qquad
\hackcenter{
\begin{tikzpicture}[scale=0.8]
 \draw (0,0) arc (0:-180:0.4cm) [thick, color=\clr];
   \draw[thick, color=\clr] (-0.8,0.5)--(-0.8,0);
     \draw[thick, color=\clr] (0,0.5)--(0,0);
      \draw[thick, color=\clr,<-] (-0.8,0.5)--(-0.8,0.2);
     \draw[thick, color=\clr,<-] (0,0.5)--(0,0.2);
         \node[above] at (0,0.5) {$\scriptstyle 1$};
      \node[above] at (-0.8,0.5) {$\scriptstyle 1$};
  \node at (-0.4,-0.4) {$\scriptstyle\blacklozenge$};
\end{tikzpicture}}
:=
\hackcenter{
\begin{tikzpicture}[scale=0.8]
  \node[white] at (-0.4,-0.4) {$\scriptstyle\blacklozenge$};
 \draw (0,0) arc (0:-180:0.4cm) [thick, color=\clr];
   \draw[thick, color=\clr] (-0.8,0.5)--(-0.8,0);
     \draw[thick, color=\clr] (0,0.5)--(0,0);
      \draw[thick, color=\clr,<-] (-0.8,0.5)--(-0.8,0.2);
     \draw[thick, color=\clr,<-] (0,0.5)--(0,0.2);
         \node[above] at (0,0.5) {$\scriptstyle 1$};
      \node[above] at (-0.8,0.5) {$\scriptstyle 1$};
       \node[shape=coordinate](DOT) at (0,0.1) {};
  \draw[thick, color=\clr, fill=blue]  (DOT) circle (1mm);
\end{tikzpicture}},
\qquad
\qquad
\hackcenter{
\begin{tikzpicture}[scale=0.8]
 \draw[thick, color=\clr,->] (0,-0.1)--(0,0.3);
 \draw[thick, color=\clr] (0,-0.1)--(0,0.8);
    \node[below] at (0,-0.1) {$\scriptstyle 2$};
    \node[shape=coordinate](DOT) at (0,0.8) {};
     \filldraw  (DOT) circle (2.5pt);
      \node[white] at (0,1.4) {$\scriptstyle\blacklozenge$};
\end{tikzpicture}}
:=
\left(\frac{1}{2} \right)
\,
\hackcenter{
\begin{tikzpicture}[scale=0.8]
\draw[thick, color=\clr,->] (0,-0.1)--(0,0.2);
  \draw[thick, color=\clr,->] (0,0)--(0,0.2) .. controls ++(0,0.35) and ++(0,-0.35) .. (-0.4,0.9)--(-0.4,1);
  \draw[thick, color=\clr,->] (0,0)--(0,0.2) .. controls ++(0,0.35) and ++(0,-0.35) .. (0.4,0.9)--(0.4,1);
      \node[below] at (0,-0.1) {$\scriptstyle 2$};
       \draw(0.4,1) arc (0:180:0.4cm) [thick, color=\clr];
         \node at (0,1.4) {$\scriptstyle\blacklozenge$};
\end{tikzpicture}},
\end{align*}
which we call \emph{upward cap}, \emph{upward cup}, and \emph{upward antenna}, respectively. We also define, for all \(a,b \in \Z_{\geq 0}\), \emph{rightward cap} and \emph{rightward cup} morphisms:
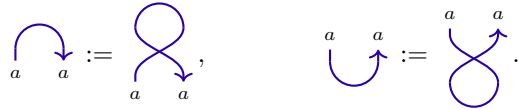
\begin{align*}
\hackcenter{}
\hackcenter{
\begin{tikzpicture}[scale=0.8]
 \draw(0,0) arc (0:180:0.4cm) [thick, color=\clr];
 \draw[thick, color=\clr] (-0.8,-0.2)--(-0.8,0);
  \draw[thick, color=\clr,<-] (0,-0.2)--(0,0);
    \node[below] at (0,-0.2) {$\scriptstyle a$};
      \node[below] at (-0.8,-0.2) {$\scriptstyle a$};
\end{tikzpicture}}
:=
\hackcenter{
\begin{tikzpicture}[scale=0.8]
  \draw[thick, color=\clr,<-] (0.4,0)--(0.4,0.1) .. controls ++(0,0.35) and ++(0,-0.35) .. (-0.4,0.9);
  \draw[thick, color=\clr] (-0.4,0)--(-0.4,0.1) .. controls ++(0,0.35) and ++(0,-0.35) .. (0.4,0.9);
   \draw(0.4,0.9) arc (0:180:0.4cm) [thick, color=\clr];
            \node[below] at (-0.4,0) {$\scriptstyle a$};
      \node[below] at (0.4,0) {$\scriptstyle a$};
\end{tikzpicture}},
\qquad
\qquad
\hackcenter{
\begin{tikzpicture}[scale=0.8]
 \draw(0,0) arc (0:-180:0.4cm) [thick, color=\clr];
 \draw[thick, color=\clr] (-0.8,0.2)--(-0.8,0);
  \draw[thick, color=\clr,<-] (0,0.2)--(0,0);
    \node[above] at (0,0.2) {$\scriptstyle a$};
      \node[above] at (-0.8,0.2) {$\scriptstyle a$};
\end{tikzpicture}}
:=
\hackcenter{
\begin{tikzpicture}[scale=0.8]
  \draw[thick, color=\clr,<-] (0.4,0)--(0.4,-0.1) .. controls ++(0,-0.35) and ++(0,0.35) .. (-0.4,-0.9);
  \draw[thick, color=\clr] (-0.4,0)--(-0.4,-0.1) .. controls ++(0,-0.35) and ++(0,0.35) .. (0.4,-0.9);
   \draw(0.4,-0.9) arc (0:-180:0.4cm) [thick, color=\clr];
            \node[above] at (-0.4,0) {$\scriptstyle a$};
      \node[above] at (0.4,0) {$\scriptstyle a$};
\end{tikzpicture}}.
\end{align*}
For all \(a,b \in \Z_{\geq 0}\), define the {\em upward crossing}
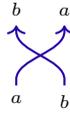
\begin{align*}
\hackcenter{}
\hackcenter{
\begin{tikzpicture}[scale=0.8]
  \draw[thick, color=\clr,->] (0.4,0)--(0.4,0.1) .. controls ++(0,0.35) and ++(0,-0.35) .. (-0.4,0.9)--(-0.4,1);
  \draw[thick, color=\clr,->] (-0.4,0)--(-0.4,0.1) .. controls ++(0,0.35) and ++(0,-0.35) .. (0.4,0.9)--(0.4,1);
      \node[above] at (-0.4,1) {$\scriptstyle b$};
      \node[above] at (0.4,1) {$\scriptstyle a$};
            \node[below] at (-0.4,0) {$\scriptstyle a$};
      \node[below] at (0.4,0) {$\scriptstyle b$};
\end{tikzpicture}}
\end{align*}
as in \cref{CrossDef}, with all strands oriented upward. 
We then define the {\em leftward crossing} by composing this with the leftward cap and cup morphisms:
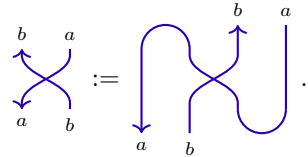
\begin{align*}
\hackcenter{}
\hackcenter{
\begin{tikzpicture}[scale=0.8]
  \draw[thick, color=\clr,->] (0.4,0)--(0.4,0.1) .. controls ++(0,0.35) and ++(0,-0.35) .. (-0.4,0.9)--(-0.4,1);
  \draw[thick, color=\clr,<-] (-0.4,0)--(-0.4,0.1) .. controls ++(0,0.35) and ++(0,-0.35) .. (0.4,0.9)--(0.4,1);
      \node[above] at (-0.4,1) {$\scriptstyle b$};
      \node[above] at (0.4,1) {$\scriptstyle a$};
            \node[below] at (-0.4,0) {$\scriptstyle a$};
      \node[below] at (0.4,0) {$\scriptstyle b$};
\end{tikzpicture}}
:=
\hackcenter{
\begin{tikzpicture}[scale=0.8]
  \draw[thick, color=\clr] (0.4,0)--(0.4,0.1) .. controls ++(0,0.35) and ++(0,-0.35) .. (-0.4,0.9)--(-0.4,1);
  \draw[thick, color=\clr,->] (-0.4,-0.4)--(-0.4,0.1) .. controls ++(0,0.35) and ++(0,-0.35) .. (0.4,0.9)--(0.4,1.4);
    \draw[thick, color=\clr,<-] (-1.2,-0.4)--(-1.2,1);
       \draw[thick, color=\clr] (1.2,0)--(1.2,1.4);
       \draw (-0.4,1) arc (0:180:0.4cm) [thick, color=\clr];
         \draw (1.2,0) arc (0:-180:0.4cm) [thick, color=\clr];
             \node[above] at (0.4,1.4) {$\scriptstyle b$};
      \node[above] at (1.2,1.4) {$\scriptstyle a$};
            \node[below] at (-0.4,-0.4) {$\scriptstyle b$};
      \node[below] at (-1.2,-0.4) {$\scriptstyle a$};
\end{tikzpicture}}.
\end{align*}

With this notation established we can now give the defining relations for $\pWeb_{\uparrow\downarrow}$.  The defining relations of \(\pWeb_{\uparrow \downarrow}\) are:\\

\noindent \emph{Up-arrow relations.} Relations \cref{AssocRel,CupDiagSwitchRel2} hold in \(\pWeb_{\uparrow \downarrow}\), where we interpret the diagrams as having all strands oriented upward.\\

\noindent \emph{Leftward straightening.} For all \(a \in \Z_{\geq 0}\) we have:
\begin{align}\label{OrStraightRel}
\hackcenter{}
\hackcenter{
\begin{tikzpicture}[scale=0.8]
  \draw[thick, color=\clr] (0,0)--(0,-0.8); 
 \draw (0,0) arc (0:180:0.4cm) [thick, color=\clr];
   \draw[thick, color=\clr] (-0.8,0)--(-0.8,-0.2); 
  \draw (-0.8,-0.2) arc (0:-180:0.4cm) [thick, color=\clr];
   \draw[thick, color=\clr,->] (-1.6,-0.2)--(-1.6,0.6); 
     \node[below] at (0,-0.8) {$\scriptstyle a$};
\end{tikzpicture}}
\;
=
\;
\hackcenter{
\begin{tikzpicture}[scale=0.8]
   \draw[thick, color=\clr,->] (0,-0.8)--(0,0.6); 
     \node[below] at (0,-0.8) {$\scriptstyle a$};
\end{tikzpicture}},
\qquad
\qquad
\hackcenter{
\begin{tikzpicture}[scale=0.8]
   \draw[thick, color=\clr,->] (-0.8,0)--(-0.8,-0.8); 
 \draw (0,0) arc (0:180:0.4cm) [thick, color=\clr];
   \draw[thick, color=\clr] (0,0)--(0,-0.2); 
  \draw (0.8,-0.2) arc (0:-180:0.4cm) [thick, color=\clr];
   \draw[thick, color=\clr] (0.8,-0.2)--(0.8,0.6); 
     \node[below] at (-0.8,-0.8) {$\scriptstyle a$};
\end{tikzpicture}}
\;
=
\;
\hackcenter{
\begin{tikzpicture}[scale=0.8]
   \draw[thick, color=\clr,<-] (0,-0.8)--(0,0.6); 
     \node[below] at (0,-0.8) {$\scriptstyle a$};
\end{tikzpicture}}.
\end{align}

\noindent \emph{Left/right crossing inversion.} For all \(a,b \in \Z_{\geq 0}\), we have:
\begin{align}\label{LRCrossRel}
\hackcenter{}
\hackcenter{
\begin{tikzpicture}[scale=0.8]
  \draw[thick, color=\clr,<-] (0,0)--(0,0.2) .. controls ++(0,0.35) and ++(0,-0.35)  .. (0.8,0.8)--(0.8,1) .. controls ++(0,0.35) and ++(0,-0.35)  .. (0,1.6)--(0,1.8); 
    \draw[thick, color=\clr,->] (0.8 ,0)--(0.8, 0.2) .. controls ++(0,0.35) and ++(0,-0.35)  .. (0,0.8)--(0,1) .. controls ++(0,0.35) and ++(0,-0.35)  .. (0.8,1.6)--(0.8,1.8);
     \node[above] at (0,-0.4) {$\scriptstyle a$};
     \node[above] at (0.8,-0.4) {$\scriptstyle b$};
\end{tikzpicture}}
\;
=
\;
\hackcenter{
\begin{tikzpicture}[scale=0.8]
  \draw[thick, color=\clr,<-] (0,0)--(0,1.8); 
    \draw[thick, color=\clr,->] (0.8 ,0)--(0.8,1.8);
     \node[above] at (0,-0.4) {$\scriptstyle a$};
     \node[above] at (0.8,-0.4) {$\scriptstyle b$};
\end{tikzpicture}},
\qquad
\qquad
\hackcenter{
\begin{tikzpicture}[scale=0.8]
  \draw[thick, color=\clr,->] (0,0)--(0,0.2) .. controls ++(0,0.35) and ++(0,-0.35)  .. (0.8,0.8)--(0.8,1) .. controls ++(0,0.35) and ++(0,-0.35)  .. (0,1.6)--(0,1.8); 
    \draw[thick, color=\clr,<-] (0.8 ,0)--(0.8, 0.2) .. controls ++(0,0.35) and ++(0,-0.35)  .. (0,0.8)--(0,1) .. controls ++(0,0.35) and ++(0,-0.35)  .. (0.8,1.6)--(0.8,1.8);
     \node[above] at (0,-0.4) {$\scriptstyle a$};
     \node[above] at (0.8,-0.4) {$\scriptstyle b$};
\end{tikzpicture}}
\;
=
\;
\hackcenter{
\begin{tikzpicture}[scale=0.8]
  \draw[thick, color=\clr,->] (0,0)--(0,1.8); 
    \draw[thick, color=\clr,<-] (0.8 ,0)--(0.8,1.8);
     \node[above] at (0,-0.4) {$\scriptstyle a$};
     \node[above] at (0.8,-0.4) {$\scriptstyle b$};
\end{tikzpicture}}.
\end{align}

\noindent \emph{Bubble annihilation.} For all \(a \in \Z_{>0}\) we have:
\begin{align}\label{BubbleRel}
\hackcenter{}
\hackcenter{
\begin{tikzpicture}[scale=0.8]
   \draw[thick, color=\clr,->] (0.4,-1) arc (0:180:0.4cm) [thick, color=\clr];
    \draw (0.4,-1) arc (0:-180:0.4cm) [thick, color=\clr];
    \node[right] at (0.4,-1) {$\scriptstyle a$};
\end{tikzpicture}}
= 0.
\end{align}
\end{definition}

We remark that including the rightward crossing generator along with the left/right crossing inversion relation is equivalent to imposing the relation that the leftward crossing is invertible.  While this latter approach is sometimes used in the literature, we chose to make the inverse morphism explicitly part of the definition.

Going forward it will be convenient to sometimes write $\uparrow_{0}$ or $\downarrow_{0}$ for the empty word (i.e., the monoidal unit object).

\subsection{Additional Morphisms} We define {\em downward splits} and {\em downward merges} by:
\begin{align*}
\hackcenter{
{}
}
\hackcenter{
\begin{tikzpicture}[scale=0.8]
\draw[thick, color=\clr,->] (0,0.1)--(0,-0.2);
  \draw[thick, color=\clr,->] (0,0)--(0,-0.2) .. controls ++(0,-0.35) and ++(0,0.35) .. (-0.4,-0.9)--(-0.4,-1);
  \draw[thick, color=\clr,->] (0,0)--(0,-0.2) .. controls ++(0,-0.35) and ++(0,0.35) .. (0.4,-0.9)--(0.4,-1);
      \node[below] at (-0.4,-1) {$\scriptstyle a$};
      \node[below] at (0.4,-1) {$\scriptstyle b$};
      \node[above] at (0,0.1) {$\scriptstyle a+b$};
\end{tikzpicture}}
:=
\hackcenter{
\begin{tikzpicture}[scale=0.8]
\draw[thick, color=\clr,->] (0,-1)--(0,-0.8);
\draw[thick, color=\clr,->] (0.8,1.2)--(0.8,0.8);
  \draw[thick, color=\clr,<-] (-0.4,0)--(-0.4,-0.1) .. controls ++(0,-0.35) and ++(0,0.35) .. (0,-0.8)--(0,-1);
\draw[thick, color=\clr,<-] (0.4,0)--(0.4,-0.1) .. controls ++(0,-0.35) and ++(0,0.35) .. (0,-0.8)--(0,-1);
    \draw (-0.4,0) arc (0:180:0.4cm) [thick, color=\clr];
        \draw (0.4,0) arc (0:180:1.2cm) [thick, color=\clr];
          \draw (0.8,-1) arc (0:-180:0.4cm) [thick, color=\clr];
          \draw[thick, color=\clr] (0.8,-1)--(0.8,1.2);
          \draw[thick, color=\clr,<-] (-1.2,-1.4)--(-1.2,0);
            \draw[thick, color=\clr,<-] (-2,-1.4)--(-2,0);
                 \node[below] at (-1.2,-1.4) {$\scriptstyle b$};
      \node[below] at (-2,-1.4) {$\scriptstyle a$};
      \node[above] at (0.8,1.2) {$\scriptstyle a+b$};
\end{tikzpicture}},
\qquad
\qquad
\hackcenter{
\begin{tikzpicture}[scale=0.8]
 \draw[thick, color=\clr,->] (-0.4,0.1)--(-0.4,-0.2);
  \draw[thick, color=\clr,->] (0.4,0.1)--(0.4,-0.2);
  \draw[thick, color=\clr,->] (-0.4,0)--(-0.4,-0.1) .. controls ++(0,-0.35) and ++(0,0.35) .. (0,-0.8)--(0,-1);
\draw[thick, color=\clr,->] (0.4,0)--(0.4,-0.1) .. controls ++(0,-0.35) and ++(0,0.35) .. (0,-0.8)--(0,-1);
      \node[above] at (-0.4,0.1) {$\scriptstyle a$};
      \node[above] at (0.4,0.1) {$\scriptstyle b$};
      \node[below] at (0,-1) {$\scriptstyle a+b$};
\end{tikzpicture}}
:=
\hackcenter{
\begin{tikzpicture}[scale=0.8]
\draw[thick, color=\clr,->] (1.2,1.4)--(1.2,1);
\draw[thick, color=\clr,->] (2,1.4)--(2,1);
\draw[thick, color=\clr,<-] (-0.8,-1.2)--(-0.8,-0.8);
  \draw[thick, color=\clr] (0.4,0)--(0.4,0.1) .. controls ++(0,0.35) and ++(0,-0.35) .. (0,0.8)--(0,1);
\draw[thick, color=\clr,->] (-0.4,0)--(-0.4,0.1) .. controls ++(0,0.35) and ++(0,-0.35) .. (0,0.8)--(0,1);
    \draw (1.2,0) arc (0:-180:0.4cm) [thick, color=\clr,->];
        \draw (2,0) arc (0:-180:1.2cm) [thick, color=\clr,->];
          \draw (0,1) arc (0:180:0.4cm) [thick, color=\clr];
          \draw[thick, color=\clr] (-0.8,1)--(-0.8,-1.2);
          \draw[thick, color=\clr] (1.2,1.4)--(1.2,0);
            \draw[thick, color=\clr] (2,1.4)--(2,0);
                 \node[above] at (1.2,1.4) {$\scriptstyle a$};
      \node[above] at (2,1.4) {$\scriptstyle b$};
      \node[below] at (-0.8,-1.2) {$\scriptstyle a+b$};
\end{tikzpicture}}.
\end{align*}
We define {\em downward crossings} like so:
\begin{align*}
\hackcenter{}
\hackcenter{
\begin{tikzpicture}[scale=0.8]
  \draw[thick, color=\clr,<-] (0.4,0)--(0.4,0.1) .. controls ++(0,0.35) and ++(0,-0.35) .. (-0.4,0.9)--(-0.4,1);
  \draw[thick, color=\clr,<-] (-0.4,0)--(-0.4,0.1) .. controls ++(0,0.35) and ++(0,-0.35) .. (0.4,0.9)--(0.4,1);
      \node[above] at (-0.4,1) {$\scriptstyle b$};
      \node[above] at (0.4,1) {$\scriptstyle a$};
            \node[below] at (-0.4,0) {$\scriptstyle a$};
      \node[below] at (0.4,0) {$\scriptstyle b$};
\end{tikzpicture}}
:=
\hackcenter{
\begin{tikzpicture}[scale=0.8]
  \draw[thick, color=\clr,->] (0.4,0)--(0.4,0.1) .. controls ++(0,0.35) and ++(0,-0.35) .. (-0.4,0.9)--(-0.4,1);
  \draw[thick, color=\clr,->] (-0.4,0)--(-0.4,0.1) .. controls ++(0,0.35) and ++(0,-0.35) .. (0.4,0.9)--(0.4,1);
    \draw[thick, color=\clr,<-] (-1.2,-1.2)--(-1.2,1);
       \draw[thick, color=\clr] (1.2,0)--(1.2,2.2);
           \draw[thick, color=\clr,<-] (-2,-1.2)--(-2,1);
       \draw[thick, color=\clr] (2,0)--(2,2.2);
       \draw (-0.4,1) arc (0:180:0.4cm) [thick, color=\clr];
        \draw (0.4,1) arc (0:180:1.2cm) [thick, color=\clr];
         \draw (1.2,0) arc (0:-180:0.4cm) [thick, color=\clr];
           \draw (2,0) arc (0:-180:1.2cm) [thick, color=\clr];
             \node[above] at (1.2,2.2) {$\scriptstyle b$};
      \node[above] at (2,2.2) {$\scriptstyle a$};
            \node[below] at (-1.2,-1.2) {$\scriptstyle b$};
      \node[below] at (-2,-1.2) {$\scriptstyle a$};
\end{tikzpicture}}.
\end{align*}

\subsection{Implied Relations for \texorpdfstring{$\pWeb_{\uparrow\downarrow}$}{p-webupdown}} The following theorem establishes a number of additional relations which follow from the defining relations of $\pWeb_{\uparrow \downarrow}$.  In particular, they show that diagrams which are the same as oriented graphs (which may have edges with tag-in and tag-out diagrams) are equal, up to a sign.  In particular, up to a sign, tag-in and tag-out morphisms move freely through crossings and along strands.

\begin{theorem}\label{OrBraidThm}
The following equalities hold in \(\pWeb_{\uparrow \downarrow}\) for all \(a,b,c \in \mathbb{Z}_{\geq 0}\) and all admissible strand orientations:
\begin{align}\label{O1}
\hackcenter{}
\hackcenter{
}.
\end{align}
\end{theorem}
\begin{proof}
Relations \cref{O1,O2,O3,O4} hold when all strands are oriented upward, as show in \cref{BraidThm}. Using this fact, together with (\ref{OrStraightRel}--\ref{BubbleRel}), it is a routine exercise then to prove that the equalities (\ref{O1}--\ref{O9}) hold for all admissible orientations. The equalities (\ref{O10}--\ref{O14})  can be seen to hold thanks to (\ref{O1}--\ref{O9}), \cref{UpFun}, and \cref{BraidThm}, after noting that 
\begin{align}\label{TagToCupCap}
\hackcenter{}
\hackcenter{
%
},
\end{align}
which completes the proof.
\end{proof}

For a nonnegative integer $a$ it will be convenient to adopt the notation \(|_{a}:= \uparrow_{a}\) and \(|_{-a}:= \downarrow_{a}\). More generally, for \(r \in \Z_{\geq 0}\) and \(\bba =(a_1, \ldots, a_{r}) \in \Z^{r}\), we will write
\begin{align*}
|_{\bba}:= |_{a_{1}} \cdots |_{a_{r}}
\end{align*} as a shorthand for the latter as an object of $\pWeb_{\uparrow\downarrow}$.  For example, $|_{(6,2,-9)} = \uparrow_{6}\uparrow_{2}\downarrow_{9}$.

The up, down, left, and right crossings can be used to define crossing morphisms,
\begin{align*}
\beta_{|_a |_b}: |_a |_b \to |_b |_a,
\end{align*}
for arbitrary \(a,b \in \Z\).  Just as with $\aWeb$ and $\pWeb$, we can use these morphisms to make oriented versions of \cref{E:generalcrossing} and to verify these make \(\pWeb_{\uparrow \downarrow}\) into a symmetric braided monoidal supercategory.

The relations \cref{OrStraightRel}, \cref{O7} show that $\uparrow_{a}$ and $\downarrow_{a}$ are left and right duals to each other with the cups and caps as the evaluation and coevaluation morphisms.  More generally, using ``rainbows'' constructed from leftward and rightward caps and cups we can also construct evaluation and coevaluation morphisms for general objects in $\pWeb_{\uparrow\downarrow}$.  For example, the evaluation and coevaluation morphisms for $\downarrow_{a}\uparrow_{b}\downarrow_{c}$ are

\begin{equation*}
\hackcenter{
\begin{tikzpicture}[scale=0.35]
		\node   (0) at (-4, 0) {};
		\node  (1) at (-4, -0.5) {$\scriptstyle c$};
		\node   (2) at (-3, 0) {};
		\node  (3) at (-3, -0.5) {$\scriptstyle b$};
		\node   (4) at (-2, 0) {};
		\node  (5) at (-2, -0.5) {$\scriptstyle a$};
		\node   (6) at (2, 0) {};
		\node  (7) at (2, -0.5) {$\scriptstyle a$};
		\node   (8) at (3, 0) {};
		\node  (9) at (3, -0.5) {$\scriptstyle b$};
		\node   (10) at (4, 0) {};
		\node  (11) at (4, -0.5) {$\scriptstyle c$};
		\draw [thick, color=\clr, ->, looseness=1.5]  (4.center) to [out=90,in=90] (6.center);
		\draw [thick, color=\clr, <-, looseness=1.5]  (2.center) to [out=90,in=90] (8.center);
		\draw [thick, color=\clr, ->, looseness=1.5]  (0.center) to [out=90,in=90] (10.center);
\end{tikzpicture}} \hspace{0.25in} \text{ and } \hspace{0.25in}
\hackcenter{
\begin{tikzpicture}[scale=0.35, yscale=-1]
		\node   (0) at (-4, 0) {};
		\node  (1) at (-4, -0.5) {$\scriptstyle a$};
		\node   (2) at (-3, 0) {};
		\node  (3) at (-3, -0.5) {$\scriptstyle b$};
		\node   (4) at (-2, 0) {};
		\node  (5) at (-2, -0.5) {$\scriptstyle c$};
		\node   (6) at (2, 0) {};
		\node  (7) at (2, -0.5) {$\scriptstyle c$};
		\node   (8) at (3, 0) {};
		\node  (9) at (3, -0.5) {$\scriptstyle b$};
		\node   (10) at (4, 0) {};
		\node  (11) at (4, -0.5) {$\scriptstyle a$};
		\draw [thick, color=\clr, ->, looseness=1.5]  (4.center) to [out=90,in=90] (6.center);
		\draw [thick, color=\clr, <-, looseness=1.5]  (2.center) to [out=90,in=90] (8.center);
		\draw [thick, color=\clr, ->, looseness=1.5]  (0.center) to [out=90,in=90] (10.center);
\end{tikzpicture}} \; .
\end{equation*}
Using these we see that $\pWeb_{\uparrow\downarrow}$ is in fact a rigid category.  Altogether we have the following result.

\begin{corollary}\label{C:pwebupdownbraiding}  The oriented crossing morphisms endow  \(\pWeb_{\uparrow \downarrow}\) with the structure of a symmetric braided monoidal supercategory and the evaluation and coevaluation morphisms endow $\pWeb_{\uparrow\downarrow}$ with the structure of a rigid supercategory.
\end{corollary}

\subsection{Isomorphic Morphism Spaces in  \texorpdfstring{$\pWeb_{\uparrow\downarrow}$}{p-Webupdown}}\label{SS:IsomorphicHoms}  We next remind the reader of well-known arguments (e.g. \cite[Proposition 2.10.8]{EGNO}) which use the existence of the braiding morphisms in $\pWeb_{\uparrow\downarrow}$ to define isomorphisms between various morphism spaces.  Entirely analogous results obviously hold by the same arguments for $\aWeb$ and $\pWeb$.

The symmetric group \(\mathfrak{S}_r\) acts on \(\Z^r\) by place permutation:
\begin{align*}
\sigma \cdot \ba = \sigma \cdot (a_1, \ldots, a_r) = (a_{\sigma^{-1}(1)}, \ldots, a_{\sigma^{-1}(r)}).
\end{align*}
 The braiding morphism defines an associated invertible morphism
\begin{align*}
\beta_{\sigma,|_\ba} \in \Hom_{\pWeb_{\uparrow \downarrow}}(|_\ba, |_{\sigma \cdot \ba}),
\end{align*}
where \(\beta^{-1}_{\sigma, |_\ba} = \beta_{\sigma^{-1}, |_{\sigma \cdot \ba}} \). More generally, for \(r_1,r_2 \in \Z_{\geq 0}\), \(\bba \in \Z^{r_1}\), \(\bbb \in \Z^{r_2}\), \(\sigma \in \mathfrak{S}_{r_1}\), and \(\omega \in \mathfrak{S}_{r_2}\), we have an isomorphism of morphism spaces:
\begin{align}\label{SymIsom}
\Hom_{\pWeb_{\uparrow \downarrow}}(|_{\bba}, |_{\bbb}) \xrightarrow{\sim}
\Hom_{\pWeb_{\uparrow \downarrow}}(|_{ \sigma \cdot \bba}, |_{\omega \cdot \bbb}),
\end{align}
given by
\begin{align*}
f \mapsto     \beta_{\omega, |_\bbb} \circ f \circ  \beta_{\sigma^{-1}, |_{\sigma \cdot \bba}}.
\end{align*}

Let \(\omega_0 \in \mathfrak{S}_r\) be the longest element, so that \(\omega_{0}\cdot(a_1, \ldots,a_r) = (a_r, \ldots, a_1)\). Then for \(\bba \in \Z_{\geq 0}^{r_1}\), \(\bbb \in \Z_{\geq 0}^{r_2}\), \(\bbc \in \Z_{\geq 0}^{r_3}\), and \(\bbd \in \Z_{\geq 0}^{r_4}\), 
we have an isomorphism of Hom spaces:
\begin{align}\label{SplitToUp}
\Hom_{\pWeb_{\uparrow \downarrow}}(\downarrow_\bba \uparrow_\bbb, \uparrow_\bbc \downarrow_\bbd)
\xrightarrow{\sim}
\Hom_{\pWeb_\uparrow}(\uparrow_{\bbb} \uparrow_{\omega_0\bbd} ,\uparrow_{\omega_0 \bba} \uparrow_{\bbc}),
\end{align}
given by
\begin{align*}
\hackcenter{}
\hackcenter{
\begin{tikzpicture}[scale=0.8]
\draw[thick, color=\clr] (0,0)--(0,0.5)--(2.2,0.5)--(2.2,0)--(0,0);
 \draw[thick, color=\clr,->] (0.1,0.5)--(0.1,1);
  \draw[thick, color=\clr,->] (0.8,0.5)--(0.8,1);
    \draw[thick, color=\clr,<-] (1.4,0.5)--(1.4,1);
      \draw[thick, color=\clr,<-] (2.1,0.5)--(2.1,1);
      \draw[thick, color=\clr,<-] (0.1,-0.5)--(0.1,0);
  \draw[thick, color=\clr,<-] (0.8,-0.5)--(0.8,0);
    \draw[thick, color=\clr,->] (1.4,-0.5)--(1.4,0);
      \draw[thick, color=\clr,->] (2.1,-0.5)--(2.1,0);
      \node at (1.1,0.25) {$\scriptstyle f$};
         \node at (0.45,0.75) {$\scriptstyle \cdots$};
            \node at (1.75,0.75) {$\scriptstyle \cdots$};
               \node at (0.45,-0.25) {$\scriptstyle \cdots$};
                \node at (1.75,-0.25) {$\scriptstyle \cdots$};
                \node[above] at (0.1,1) {$\scriptstyle c_1$};
                \node[above] at (0.8,0.94) {$\scriptstyle c_{r_3}$};
                 \node[above] at (1.4,1) {$\scriptstyle d_1$};
                \node[above] at (2.1,0.92) {$\scriptstyle d_{r_4}$};
                 \node[below] at (0.1,-0.5) {$\scriptstyle a_1$};
                \node[below] at (0.8,-0.5) {$\scriptstyle a_{r_1}$};
                 \node[below] at (1.4,-0.44) {$\scriptstyle b_1$};
                \node[below] at (2.1,-0.44) {$\scriptstyle b_{r_2}$};
\end{tikzpicture}}
\;
\mapsto
\hackcenter{
\begin{tikzpicture}[scale=0.8]
\draw[thick, color=\clr] (0,0)--(0,0.5)--(2.2,0.5)--(2.2,0)--(0,0);
 \draw[thick, color=\clr,->] (0.1,0.5)--(0.1,2);
  \draw[thick, color=\clr,->] (0.8,0.5)--(0.8,2);
    \draw[thick, color=\clr,<-] (1.4,0.5)--(1.4,1);
      \draw[thick, color=\clr,<-] (2.1,0.5)--(2.1,1);
      \draw[thick, color=\clr] (0.1,-0.5)--(0.1,0);
  \draw[thick, color=\clr] (0.8,-0.5)--(0.8,0);
    \draw[thick, color=\clr,->] (1.4,-1.5)--(1.4,0);
      \draw[thick, color=\clr,->] (2.1,-1.5)--(2.1,0);
          \draw[thick, color=\clr,->] (-1.2,-0.5)--(-1.2,2);
           \draw[thick, color=\clr,->] (-0.5,-0.5)--(-0.5,2);
            \draw[thick, color=\clr] (2.7,-1.5)--(2.7,1);
            \draw[thick, color=\clr] (3.4,-1.5)--(3.4,1);
      \node at (1.1,0.25) {$\scriptstyle f$};
         \node at (0.45,1.25) {$\scriptstyle \cdots$};
            \node at (1.75,0.75) {$\scriptstyle \cdots$};
               \node at (0.45,-0.25) {$\scriptstyle \cdots$};
                \node at (1.75,-0.75) {$\scriptstyle \cdots$};
                 \node at (3.05,0.25) {$\scriptstyle \cdots$};
                 \node at (-0.85,0.25) {$\scriptstyle \cdots$};
                 \draw[thick, color=\clr] (2.7,1) arc (0:180:0.3cm) [thick, color=\clr];
                  \draw[thick, color=\clr] (3.4,1) arc (0:180:1cm) [thick, color=\clr];
                  \draw (0.1,-0.5) arc (0:-180:0.3cm) [thick, color=\clr];
                    \draw (0.8,-0.5) arc (0:-180:1cm) [thick, color=\clr];
                      \node[above] at (0.1,2) {$\scriptstyle c_1$};
                \node[above] at (0.8,1.94) {$\scriptstyle c_{r_3}$};
                 \node[below] at (3.4,-1.5) {$\scriptstyle d_1$};
                \node[below] at (2.7,-1.5) {$\scriptstyle d_{r_4}$};
                 \node[above] at (-0.5,2) {$\scriptstyle a_1$};
                \node[above] at (-1.2,1.94) {$\scriptstyle a_{r_1}$};
                 \node[below] at (1.4,-1.5) {$\scriptstyle b_1$};
                \node[below] at (2.1,-1.5) {$\scriptstyle b_{r_2}$};
\end{tikzpicture}}
\;,
\end{align*}
with inverse given by an entirely similar map, thanks to \cref{OrStraightRel}.

 Let \(\pWeb_{\uparrow}\) (resp.\  \(\pWeb_{\downarrow}\)) be the full subcategory of \(\pWeb_{\uparrow \downarrow}\) consisting of objects of the form \(\uparrow_{\bba}:=\uparrow_{a_1} \cdots \uparrow_{a_r}\) (resp.\   \(\downarrow_{\bba}:= \downarrow_{a_1} \cdots \downarrow_{a_r}\)) for all \(r \in \Z_{\geq 0}\) and \(\bba =(a_1, \ldots, a_r) \in \Z_{\geq 0}^r\).  Combining \cref{SymIsom} and \cref{SplitToUp} yields the following lemma.

\begin{lemma}\label{MixToUp}
Let \(\bba \in \Z^{r_1}\) and  \(\bbb \in \Z^{r_2}\). Then there exists \(\bbc \in \Z_{\geq 0}^{r_3}\), \(\bbd \in \Z_{\geq 0}^{r_4}\) such that \(|\bbc| + |\bbd| = |\bba| + |\bbb|\) and  there is a parity preserving isomorphism of morphism spaces
\begin{align*}
\Phi:\Hom_{\pWeb_{\uparrow \downarrow}}(|_{\bba}, |_{\bbb}) 
\xrightarrow{\sim}
\Hom_{\pWeb_{\uparrow}}(\uparrow_{\bbc}, \uparrow_{\bbd})
\end{align*}
given by
\begin{align*}
\Phi:f \mapsto \varphi_2 \circ f \circ \varphi_1,
\end{align*}
for some invertible morphisms
\begin{align*}
\varphi_1 \in \Hom_{\pWeb_{\uparrow \downarrow}}(\uparrow_{\bbc}, |_{\bba}) 
\qquad
\textup{and}
\qquad
\varphi_2 \in \Hom_{\pWeb_{\uparrow \downarrow}}(|_{\bbb}, \uparrow_{\bbd}).
\end{align*}
\end{lemma}

Given a supercategory $\mathcal{B}$, let $\mathcal{B}^{\textup{sop}}$ be the category with the same objects and morphisms as $\mathcal{B}$ but with composition given by $\alpha \bullet \beta := (-1)^{\p{\alpha}\p{\beta}}\beta \circ \alpha$ for all homogeneous morphisms $\alpha$ and $\beta$ in $\mathcal{B}$. 
Let \(\textup{refl}:\pWeb_{\uparrow \downarrow} \to \pWeb_{\uparrow \downarrow}\) be the involutive contravariant superfunctor (in the sense of \cite{BE, LEP}) given by \(D \mapsto (-1)^{k(k-1)/2}D'\) on diagrams, where \(D'\) is the reflection of \(D\) along a horizontal axis, and \(k\) is the number of tag-in and tag-out generators in \(D\). It is easily checked that this is well-defined using \cref{OrBraidThm}. The following lemma is immediate.

\begin{lemma}
The contravariant superfunctor \(\textup{refl}\) is an equivalence of supercategories \(\pWeb_{\uparrow \downarrow} \to \pWeb_{\uparrow \downarrow}^{\textup{sop}}\) and restricts to an equivalence \(\pWeb_{\uparrow} \to \pWeb_{\downarrow}^{\textup{sop}}\).
\end{lemma}

\subsection{Connecting  \texorpdfstring{$\pWeb$}{p-Web} to  \texorpdfstring{$\pWeb_{\uparrow \downarrow}$}{p-Webupdown}}

\begin{lemma}\label{UpFun}
There is a well-defined functor of monoidal supercategories
\begin{align*}
\iota_{\uparrow}: \pWeb \to \pWeb_{\uparrow}
\end{align*}
given on objects by \(\iota_\uparrow(k) = \uparrow_k\) and on morphisms by
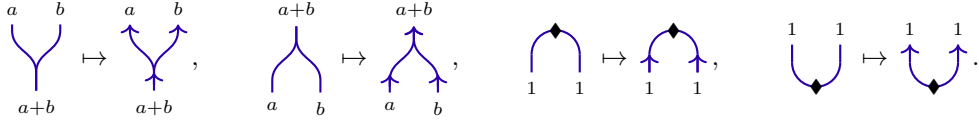
\begin{align*}
\hackcenter{
{}
}
\hackcenter{
\begin{tikzpicture}[scale=0.8]
  \draw[thick, color=\clr] (0,-0.1)--(0,0.2) .. controls ++(0,0.35) and ++(0,-0.35) .. (-0.4,0.9)--(-0.4,1);
  \draw[thick, color=\clr] (0,-0.1)--(0,0.2) .. controls ++(0,0.35) and ++(0,-0.35) .. (0.4,0.9)--(0.4,1);
      \node[above] at (-0.4,1) {$\scriptstyle a$};
      \node[above] at (0.4,1) {$\scriptstyle b$};
      \node[below] at (0,-0.1) {$\scriptstyle a+b$};
\end{tikzpicture}}
\mapsto 
\hackcenter{
\begin{tikzpicture}[scale=0.8]
\draw[thick, color=\clr,->] (0,-0.1)--(0,0.2);
  \draw[thick, color=\clr,->] (0,0)--(0,0.2) .. controls ++(0,0.35) and ++(0,-0.35) .. (-0.4,0.9)--(-0.4,1);
  \draw[thick, color=\clr,->] (0,0)--(0,0.2) .. controls ++(0,0.35) and ++(0,-0.35) .. (0.4,0.9)--(0.4,1);
      \node[above] at (-0.4,1) {$\scriptstyle a$};
      \node[above] at (0.4,1) {$\scriptstyle b$};
      \node[below] at (0,-0.1) {$\scriptstyle a+b$};
\end{tikzpicture}},
\qquad
\hackcenter{
\begin{tikzpicture}[scale=0.8]
  \draw[thick, color=\clr] (-0.4,-0.1)--(-0.4,0.1) .. controls ++(0,0.35) and ++(0,-0.35) .. (0,0.8)--(0,1);
\draw[thick, color=\clr] (0.4,-0.1)--(0.4,0.1) .. controls ++(0,0.35) and ++(0,-0.35) .. (0,0.8)--(0,1);
      \node[below] at (-0.4,-0.1) {$\scriptstyle a$};
      \node[below] at (0.4,-0.1) {$\scriptstyle b$};
      \node[above] at (0,1) {$\scriptstyle a+b$};
\end{tikzpicture}}
\mapsto
\hackcenter{
\begin{tikzpicture}[scale=0.8]
 \draw[thick, color=\clr,->] (-0.4,-0.1)--(-0.4,0.2);
  \draw[thick, color=\clr,->] (0.4,-0.1)--(0.4,0.2);
  \draw[thick, color=\clr,->] (-0.4,0)--(-0.4,0.1) .. controls ++(0,0.35) and ++(0,-0.35) .. (0,0.8)--(0,1);
\draw[thick, color=\clr,->] (0.4,0)--(0.4,0.1) .. controls ++(0,0.35) and ++(0,-0.35) .. (0,0.8)--(0,1);
      \node[below] at (-0.4,-0.1) {$\scriptstyle a$};
      \node[below] at (0.4,-0.1) {$\scriptstyle b$};
      \node[above] at (0,1) {$\scriptstyle a+b$};
\end{tikzpicture}},
\qquad
\hackcenter{
\begin{tikzpicture}[scale=0.8]
 \draw (0,0) arc (0:180:0.4cm) [thick, color=\clr];
   \draw[thick, color=\clr] (-0.8,-0.3)--(-0.8,0);
     \draw[thick, color=\clr] (0,-0.3)--(0,0);
         \node[below] at (0,-0.3) {$\scriptstyle 1$};
      \node[below] at (-0.8,-0.3) {$\scriptstyle 1$};
  \node at (-0.4,0.4) {$\scriptstyle\blacklozenge$};
\end{tikzpicture}}
\mapsto
\hackcenter{
\begin{tikzpicture}[scale=0.8]
 \draw (0,0) arc (0:180:0.4cm) [thick, color=\clr];
   \draw[thick, color=\clr] (-0.8,-0.3)--(-0.8,0);
     \draw[thick, color=\clr] (0,-0.3)--(0,0);
      \draw[thick, color=\clr,->] (-0.8,-0.3)--(-0.8,0);
     \draw[thick, color=\clr,->] (0,-0.3)--(0,0);
         \node[below] at (0,-0.3) {$\scriptstyle 1$};
      \node[below] at (-0.8,-0.3) {$\scriptstyle 1$};
  \node at (-0.4,0.4) {$\scriptstyle\blacklozenge$};
\end{tikzpicture}},
\qquad
\hackcenter{
\begin{tikzpicture}[scale=0.8]
 \draw (0,0) arc (0:-180:0.4cm) [thick, color=\clr];
   \draw[thick, color=\clr] (-0.8,0.3)--(-0.8,0);
     \draw[thick, color=\clr] (0,0.3)--(0,0);
         \node[above] at (0,0.3) {$\scriptstyle 1$};
      \node[above] at (-0.8,0.3) {$\scriptstyle 1$};
  \node at (-0.4,-0.4) {$\scriptstyle\blacklozenge$};
\end{tikzpicture}}
\mapsto
\hackcenter{
\begin{tikzpicture}[scale=0.8]
 \draw (0,0) arc (0:-180:0.4cm) [thick, color=\clr];
   \draw[thick, color=\clr] (-0.8,0.3)--(-0.8,0);
     \draw[thick, color=\clr] (0,0.3)--(0,0);
      \draw[thick, color=\clr,<-] (-0.8,0.3)--(-0.8,0.2);
     \draw[thick, color=\clr,<-] (0,0.3)--(0,0.2);
         \node[above] at (0,0.3) {$\scriptstyle 1$};
      \node[above] at (-0.8,0.3) {$\scriptstyle 1$};
  \node at (-0.4,-0.4) {$\scriptstyle\blacklozenge$};
\end{tikzpicture}}.
\end{align*}
\end{lemma}

\begin{proof}The theorem follows immediately from the defining relations of $\pWeb$ and $\pWeb_{\uparrow\downarrow}$. 
\end{proof}

\begin{theorem}\label{iotaFull}
The functor \(\iota_\uparrow:\pWeb \to \pWeb_\uparrow\) is full.
\end{theorem}
\begin{proof}
We begin by proving a claim.

{\em Claim:} Let \(\bba \in \Z_{\geq 0}^{r}\), \(\bbb \in \Z_{\geq 0}^s\), with \(a_1 = b_1=k\). If \(f \in \Hom_{\pWeb_\uparrow}(\uparrow_\bba, \uparrow_\bbb)\) is in the image of \(\iota_\uparrow\), then the morphism 
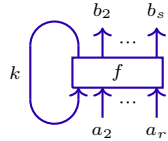
\begin{align*}
\hackcenter{}
\hackcenter{
\begin{tikzpicture}[scale=0.8]
\draw[thick, color=\clr] (0,0)--(0,0.5)--(1.5,0.5)--(1.5,0)--(0,0);
 \draw[thick, color=\clr] (0.1,0.5)--(0.1,0.7);
  \draw[thick, color=\clr,->] (0.5,0.5)--(0.5,1);
    \draw[thick, color=\clr,->] (1.4,0.5)--(1.4,1);
      \draw[thick, color=\clr,->] (0.1,-0.2)--(0.1,0);
  \draw[thick, color=\clr,->] (0.5,-0.5)--(0.5,0);
    \draw[thick, color=\clr,->] (1.4,-0.5)--(1.4,0);
         \draw (0.1,0.7) arc (0:180:0.4cm) [thick, color=\clr];
          \draw (0.1,-0.2) arc (0:-180:0.4cm) [thick, color=\clr];
            \draw[thick, color=\clr] (-0.7,-0.2)--(-0.7,0.7);
      \node at (0.75,0.25) {$\scriptstyle f$};
         \node at (0.95,0.75) {$\scriptstyle \cdots$};
            \node at (0.95,-0.25) {$\scriptstyle \cdots$};
                     \node[left] at (-0.7,0.25) {$\scriptstyle k$};
                \node[above] at (0.5,1) {$\scriptstyle b_{2}$};
                 \node[above] at (1.4,1) {$\scriptstyle b_s$};
                \node[below] at (0.5,-0.5) {$\scriptstyle a_2$};
                 \node[below] at (1.4,-0.5) {$\scriptstyle a_r$};
\end{tikzpicture}}
\end{align*}
is also in the image of \(\iota_\uparrow\).

We prove the Claim by induction on \(k\), with the base case \(k=0\) being trivial. Let \(k>0\) and assume the claim holds for all \(m<k\). By  \cref{BasisThm}, we may assume that \(f\) is of the form \(\iota_\uparrow(\xi)\) for some \(\xi \in \mathscr{B}\). After an isotopy of the strands in \(\iota_\uparrow(\xi)\), we may write
\begin{align}\label{Lk1}
\hackcenter{}
\hackcenter{
\begin{tikzpicture}[scale=0.8]
\draw[thick, color=\clr] (0,0)--(0,0.5)--(1.5,0.5)--(1.5,0)--(0,0);
 \draw[thick, color=\clr] (0.1,0.5)--(0.1,0.7);
  \draw[thick, color=\clr,->] (0.5,0.5)--(0.5,1);
    \draw[thick, color=\clr,->] (1.4,0.5)--(1.4,1);
      \draw[thick, color=\clr,->] (0.1,-0.2)--(0.1,0);
  \draw[thick, color=\clr,->] (0.5,-0.5)--(0.5,0);
    \draw[thick, color=\clr,->] (1.4,-0.5)--(1.4,0);
         \draw (0.1,0.7) arc (0:180:0.4cm) [thick, color=\clr];
          \draw (0.1,-0.2) arc (0:-180:0.4cm) [thick, color=\clr];
            \draw[thick, color=\clr] (-0.7,-0.2)--(-0.7,0.7);
      \node at (0.75,0.25) {$\scriptstyle f$};
         \node at (0.95,0.75) {$\scriptstyle \cdots$};
            \node at (0.95,-0.25) {$\scriptstyle \cdots$};
                     \node[left] at (-0.7,0.25) {$\scriptstyle k$};
                \node[above] at (0.5,1) {$\scriptstyle b_{2}$};
                 \node[above] at (1.4,1) {$\scriptstyle b_s$};
                \node[below] at (0.5,-0.5) {$\scriptstyle a_2$};
                 \node[below] at (1.4,-0.5) {$\scriptstyle a_r$};
\end{tikzpicture}}
=
\pm 
\hackcenter{}
\hackcenter{
\begin{tikzpicture}[scale=0.8]
\draw[thick, color=\clr, fill=white] (0.6,-1)--(0.6,-0.5)--(2.1,-0.5)--(2.1,-1)--(0.6,-1);
\draw[thick, color=\clr, fill=white] (1,-0.5)--(1,1)--(2.1,1)--(2.1,-0.5)--(1,-0.5);
\draw[thick, color=\clr, fill=white] (0.6,1)--(0.6,1.5)--(2.1,1.5)--(2.1,1)--(0.6,1);
  \draw[thick, color=\clr] (-0.1,-0.5)--(-0.1,-0.3) .. controls ++(0,0.35) and ++(0,-0.35) .. (0.7,0.8)--(0.7,1);
   \draw[thick, color=\clr] (0.7,-0.5)--(0.7,-0.3) .. controls ++(0,0.35) and ++(0,-0.35) .. (-0.1,0.8)--(-0.1,1);
  \draw[thick, color=\clr,->] (0.7,0.8)--(0.7,1);
  \draw[thick, color=\clr,->] (0.7,-0.5)--(0.7,-0.3);
      \draw[thick, color=\clr,->] (0.7,-1.3)--(0.7,-1);
       \draw[thick, color=\clr,->] (2,-1.3)--(2,-1);
             \draw[thick, color=\clr,->] (0.7,1.5)--(0.7,1.8);
       \draw[thick, color=\clr,->] (2,1.5)--(2,1.8);
         \draw (-0.1,1) arc (0:180:0.4cm) [thick, color=\clr];
          \draw (-0.1,-0.5) arc (0:-180:0.4cm) [thick, color=\clr];
            \draw[thick, color=\clr,<-] (-0.9,-0.5)--(-0.9,1);
               \draw[thick, color=\clr] (-0.1,-0.5)--(-0.1,1);
                \draw[thick, color=\clr,->] (-0.1,-0.5)--(-0.1,0.35);
                \node at (1.35,-0.75) {$\scriptstyle f_1$};
      \node at (1.55,0.25) {$\scriptstyle f_2$};
       \node at (1.35,1.25) {$\scriptstyle f_3$};
         \node at (1.35,1.65) {$\scriptstyle \cdots$};
            \node at (1.35,-1.15) {$\scriptstyle \cdots$};
                     \node[left] at (-0.9,0.25) {$\scriptstyle k$};
                     \node[left] at (-0.1,0.25) {$\scriptstyle k'$};
                                 \node[above] at (0.7,1.8) {$\scriptstyle b_{2}$};
                 \node[above] at (2,1.8) {$\scriptstyle b_s$};
                \node[below] at (0.7,-1.3) {$\scriptstyle a_2$};
                 \node[below] at (2,-1.3) {$\scriptstyle a_r$};
\end{tikzpicture}}\,,
\end{align}
for some morphisms \(f_1,f_2, f_3\) in the image of \(\iota_\uparrow\), and some \(k' \leq k\). If \(k'=k\), then the diagram has a bubble, and is thus zero by \cref{BubbleRel}. If \(k'=0\), then the loop can be untwisted, using \cref{OrBraidThm}. So we assume now that \(0<k'<k\). Using \cref{O8,O9}, we may rewrite \cref{Lk1} as:
\begin{align*}
\pm 
\hackcenter{}
\hackcenter{
\begin{tikzpicture}[scale=0.8]
\draw[thick, color=\clr, fill=white] (0.6,-1)--(0.6,-0.5)--(2.1,-0.5)--(2.1,-1)--(0.6,-1);
\draw[thick, color=\clr, fill=white] (1,-0.5)--(1,1)--(2.1,1)--(2.1,-0.5)--(1,-0.5);
\draw[thick, color=\clr, fill=white] (0.6,1)--(0.6,1.5)--(2.1,1.5)--(2.1,1)--(0.6,1);
  \draw[thick, color=\clr] (-0.1,-0.5)--(-0.1,-0.3) .. controls ++(0,0.35) and ++(0,-0.35) .. (0.7,0.8)--(0.7,1);
   \draw[thick, color=\clr] (0.7,-0.5)--(0.7,-0.3) .. controls ++(0,0.35) and ++(0,-0.35) .. (-0.1,0.8)--(-0.1,1);
     \draw[thick, color=\clr] (-0.1,-1.9) .. controls ++(0,0.35) and ++(0,-0.35) ..  (-0.5,-1.3)--(-0.5,-1.1) .. controls ++(0,0.35) and ++(0,-0.35) .. (-0.1,-0.5);
      \draw[thick, color=\clr] (-0.9,-1.9) .. controls ++(0,0.35) and ++(0,-0.35) ..  (-0.5,-1.3)--(-0.5,-1.1) .. controls ++(0,0.35) and ++(0,-0.35) .. (-0.9,-0.5);
        \draw[thick, color=\clr,->] (-0.5,-1.2)--(-0.5,-1.1);
  \draw[thick, color=\clr,->] (0.7,0.8)--(0.7,1);
  \draw[thick, color=\clr,->] (0.7,-0.5)--(0.7,-0.3);
      \draw[thick, color=\clr,->] (0.7,-3.1)--(0.7,-1);
       \draw[thick, color=\clr,->] (2,-3.1)--(2,-1);
             \draw[thick, color=\clr,->] (0.7,1.5)--(0.7,2.2);
       \draw[thick, color=\clr,->] (2,1.5)--(2,2.2);
         \draw (-0.9,1) arc (0:180:0.4cm) [thick, color=\clr];
           \draw (-0.1,1) arc (0:180:1.2cm) [thick, color=\clr];
            \draw (-0.1,-1.9) arc (0:-180:1.2cm) [thick, color=\clr];
          \draw (-0.9,-1.9) arc (0:-180:0.4cm) [thick, color=\clr];
            \draw[thick, color=\clr] (-0.9,-0.5)--(-0.9,1);
                \draw[thick, color=\clr] (-0.9,-0.5)--(-0.9,1);
                \draw[thick, color=\clr,->] (-0.9,-0.5)--(-0.9,0.35);
                \draw[thick, color=\clr,<-] (-1.7,-1.9)--(-1.7,1);
                \draw[thick, color=\clr,<-] (-2.5,-1.9)--(-2.5,1);
                \node at (1.35,-0.75) {$\scriptstyle f_1$};
      \node at (1.55,0.25) {$\scriptstyle f_2$};
       \node at (1.35,1.25) {$\scriptstyle f_3$};
         \node at (1.35,1.85) {$\scriptstyle \cdots$};
            \node at (1.35,-2.1) {$\scriptstyle \cdots$};
                     \node[left] at (-1.7,-1.6) {$\scriptstyle k'$};
                      \node[left] at (-2.5,-1.6) {$\scriptstyle k-k'$};
                     \node[left] at (-0.5,-1.1) {$\scriptstyle k$};
                          \node[above] at (0.7,2.2) {$\scriptstyle b_{2}$};
                 \node[above] at (2,2.2) {$\scriptstyle b_s$};
                \node[below] at (0.7,-3.1) {$\scriptstyle a_2$};
                 \node[below] at (2,-3.1) {$\scriptstyle a_r$};
\end{tikzpicture}}
=
\pm
\hackcenter{
\begin{tikzpicture}[scale=0.8]
\draw[thick, color=\clr] (-0.8,0)--(-0.8,0.5)--(1.5,0.5)--(1.5,0)--(-0.8,0);
 \draw[thick, color=\clr] (0.1,0.5)--(0.1,0.7);
       \draw[thick, color=\clr,->] (0.1,-0.2)--(0.1,0);
        \draw[thick, color=\clr] (-0.7,0.5)--(-0.7,0.7);
       \draw[thick, color=\clr,->] (-0.7,-0.2)--(-0.7,0);
  \draw[thick, color=\clr,->] (0.5,0.5)--(0.5,1.9);
    \draw[thick, color=\clr,->] (1.4,0.5)--(1.4,1.9);
  \draw[thick, color=\clr,->] (0.5,-1.4)--(0.5,0);
    \draw[thick, color=\clr,->] (1.4,-1.4)--(1.4,0);
         \draw (0.1,0.7) arc (0:180:1.2cm) [thick, color=\clr];
          \draw (0.1,-0.2) arc (0:-180:1.2cm) [thick, color=\clr];
              \draw (-0.7,0.7) arc (0:180:0.4cm) [thick, color=\clr];
          \draw (-0.7,-0.2) arc (0:-180:0.4cm) [thick, color=\clr];
            \draw[thick, color=\clr,<-] (-1.5,-0.2)--(-1.5,0.7);
            \draw[thick, color=\clr,<-] (-2.3,-0.2)--(-2.3,0.7);
      \node at (0.35,0.25) {$\scriptstyle f$};
         \node at (0.95,1.2) {$\scriptstyle \cdots$};
            \node at (0.95,-0.7) {$\scriptstyle \cdots$};
                     \node[left] at (-1.5,0.25) {$\scriptstyle k'$};
                         \node[left] at (-2.3,0.25) {$\scriptstyle k-k'$};
                \node[above] at (0.5,1.9) {$\scriptstyle b_{2}$};
                 \node[above] at (1.4,1.9) {$\scriptstyle b_s$};
                \node[below] at (0.5,-1.4) {$\scriptstyle a_2$};
                 \node[below] at (1.4,-1.4) {$\scriptstyle a_r$};
\end{tikzpicture}}\,,
\end{align*}
where \(g\) is a morphism in the image of \(\iota_\uparrow\). Now, applying the inductive assumption for \(k'\), and then for \(k-k'\) gives the result, proving the Claim.

Now we prove the lemma. Let \(f\) be a diagram in \(\pWeb_{\uparrow}\). Using \cref{TagToCupCap}, we may assume \(f\) is composed only of upward splits, upward merges, leftward/rightward/upward cups, leftward/rightward/upward caps, and crossings of all orientations. Let \(c\) be the number of leftward/rightward cups in \(f\). We prove by induction on \(c\) that \(f\) is in the image of \(\iota_\uparrow\).

If \(c=0\), then since the domain and codomain are composed only of up-arrows, it must be that there are no downward strands in \(f\), so \(f\) is composed only of upward splits, upward merges, upward cups, upward caps, and upward crossings, and hence \(f\) is in the image of \(\iota_\uparrow\).

Now, for the induction step, assume \(c>0\). Select any leftward/rightward cup in \(f\). Again, since the domain and codomain are composed only of up-arrows, it must be that the downward strand leaving from the cup must lead into a leftward/rightward cap in \(f\). Then, using  \cref{OrBraidThm}, we may pull the downward strand to the left side of the diagram, giving a diagram of the form
\begin{align*}
\hackcenter{}
\hackcenter{
\begin{tikzpicture}[scale=0.8]
\draw[thick, color=\clr] (0,0)--(0,0.5)--(1.5,0.5)--(1.5,0)--(0,0);
 \draw[thick, color=\clr] (0.1,0.5)--(0.1,0.7);
  \draw[thick, color=\clr,->] (0.5,0.5)--(0.5,1);
    \draw[thick, color=\clr,->] (1.4,0.5)--(1.4,1);
      \draw[thick, color=\clr,->] (0.1,-0.2)--(0.1,0);
  \draw[thick, color=\clr,->] (0.5,-0.5)--(0.5,0);
    \draw[thick, color=\clr,->] (1.4,-0.5)--(1.4,0);
         \draw (0.1,0.7) arc (0:180:0.4cm) [thick, color=\clr];
          \draw (0.1,-0.2) arc (0:-180:0.4cm) [thick, color=\clr];
            \draw[thick, color=\clr] (-0.7,-0.2)--(-0.7,0.7);
      \node at (0.75,0.25) {$\scriptstyle g$};
         \node at (0.95,0.75) {$\scriptstyle \cdots$};
            \node at (0.95,-0.25) {$\scriptstyle \cdots$};
                     \node[left] at (-0.7,0.25) {$\scriptstyle k$};
                \node[above] at (0.5,1) {$\scriptstyle b_{1}$};
                 \node[above] at (1.4,1) {$\scriptstyle b_s$};
                \node[below] at (0.5,-0.5) {$\scriptstyle a_1$};
                 \node[below] at (1.4,-0.5) {$\scriptstyle a_r$};
\end{tikzpicture}}\,,
\end{align*}
where \(g\) is a diagram in \(\pWeb_\uparrow\) with \(c-1\) leftward/rightward cups. By the induction assumption, \(g\) is in the image of \(\iota_\uparrow\). Thus, by the Claim \(f\) itself is in the image of \(\iota_\uparrow\), completing the proof.
\end{proof}
We will show in \cref{T:pWebtopWebup} that $\iota_{\uparrow}$ is also faithful.

\section{The Lie Superalgebra of Type P}\label{S:LieSuperalgebraofTypeP}

\subsection{The Lie Superalgebra of Type P}\label{SS:LSAoftypeP}

In this section let $\k$ be a field of characteristic different from two.  Let $I=I_{n|n}$ be the index set $\left\{1, \dotsc ,n, -1, \dotsc , -n \right\}$ with fixed order $1 < \dotsb < n < -1 < \dotsb < -n$.  Let $\p{ \cdot }: I \to \Z_{2}$ be the function defined by $\p{i} = \0$ if $i >0$ and $\p{i}=\1$ if $i <0$. Let $V=V_{n}$ be the vector space with distinguished basis $\left\{v_{i}  \mid  i \in I \right\}$.  We define a $\Z_{2}$-grading on $V$ by declaring $\p{v_{i}}=\p{i}$ for all $i \in I$.  Let $\gl (V) = \gl(n|n)$ denote the superspace of all linear endomorphisms of $V$.  Then $\gl (V)$ is a Lie superalgebra via the graded version of the commutator bracket: 
\[
[f,g] = f \circ g - (-1)^{\p{f}\; \p{g}}g \circ f
\]
for all homogeneous $f,g \in \gl (V)$.  As done here, we frequently only give a formula for homogeneous elements and leave it understood that the general case is obtained via linearity.

Define an odd supersymmetric nondegenerate bilinear form on $V$ by declaring 
\begin{equation}\label{E:bilinear}
(v_{i},v_{j}) = 
(v_{j}, v_{i}) = \delta_{i,-j}
\end{equation}
for $i,j \in I$.  Here, odd means that the associated linear map $V \otimes V \to \k$ is an odd map of superspaces, while supersymmetric means that $(v,w) = (-1)^{\p{v} \p{w}} (w,v)$ for all homogeneous $v,w \in V$.   

Define a Lie superalgebra $\fg = \fp (n) \subseteq \gl (V)$ consisting of all linear maps which preserve the bilinear form given in \cref{E:bilinear}.  That is, for all homogeneous $v,w \in V$,
\[
\fp(n) = \left\{f \in \gl (V) \, \left| \, (f(v),w) + (-1)^{\p{f}  \p{v}}(v,f(w))=0 \right. \right\}.
\]   The supercommutator restricts to define a Lie superalgebra structure on $\fp (n)$.

With respect to our choice of basis we can describe $\fp(n)$ as the $2n \times 2n$ matrices defined over $\k$ of the form
\begin{equation}\label{E:matrixform}
\fp (n) = \left\{\left(\begin{matrix} A & B \\
                                  C & -A^{t} 
\end{matrix} \right) \right\},
\end{equation}
where $A,B,C$ are $n \times n$ matrices with entries from $\k$ with $B$ symmetric, $C$ skew-symmetric, and where $A^{t}$ denotes the transpose of $A$.  In terms of \cref{E:matrixform} the $\Z_{2}$-grading is given by declaring $\fg_{\0}$ as the subspace of all such matrices where $B=C=0$ and $\fg_{\1 }$ as the subspace of all such matrices where $A=0$.
%In fact, $\fp (n)$ admits a $\Z$-grading where $\fg_{1}$ is the subspace of all matrices with $A=C=0$, $\fg_{0}=\fg_{\0}$, and $\fg_{1}$ is the subspace of all matrices with $A=B=0$.  Then $\fg = \fg_{-1} \oplus \fg_{0} \oplus \fg_{1}$ and this is a decomposition as $\fg_{0}$-modules under the adjoint action.

A (left) $\fp (n)$-supermodule is a $\Z_{2}$-graded $\k$-vector space with a left $\k$-linear action of $\fp (n)$ which respects the $\Z_{2}$-grading and which satisfies graded versions of the usual axioms required of a module for a Lie algebra.  For example, the \emph{natural supermodule} is $V_{n}$ with  $\fp(n)$-supermodule structure given by matrix multiplication.   Since we will only consider supermodules we usually leave the prefix ``super'' implicit going forward. 

We allow for all (not just parity preserving) $\fp(n)$-module homomorphisms.  Consequently, the set of all $\fp(n)$-homomorphisms between two modules is naturally a $\Z_{2}$-graded vector space.  Explicitly, $f: M \to N$ is a homogeneous $\fp (n)$-module homomorphism if $f$ is a linear map which satisfies $f(M_{s}) \subseteq N_{s+\p{f}}$ for $s\in \Z_{2}$ and $f(x.m) = (-1)^{\p{x}\p{f}}x.f(m)$ for all homogeneous $x \in \fp (n)$ and $m\in M$.

Since the enveloping superalgebra $U(\fp (n))$ is a Hopf superalgebra, the category of $\fp (n)$-modules is a monoidal supercategory in the sense of \cite{BE}.  In what follows, we study particular monoidal sub-supercategories of this category.  For every $k \geq 0$ let $S^{k}(V_{n})$ denote the $k$th symmetric power of the natural module $V_{n}$ (by convention, $S^{0}(V_{n}) = \k$, the trivial module).  Let $\pmodS $ denote the full monoidal sub-supercategory of $\fp(n)$-modules generated by $\left\{S^{k}(V_{n}) \mid k \geq 0 \right\}$ and let $\pmodSS$ denote the full monoidal sub-supercategory of $\fp(n)$-modules generated by $\left\{S^{k}(V_{n}), S^{k}(V_{n})^{*} \mid k \geq 0 \right\}$.   That is, $\pmodS$ is the full subcategory of $\fp (n)$-modules consisting of objects of the form
\begin{align*}
S^{a_1}(V_n) \otimes \cdots \otimes S^{a_k}(V_n),
\end{align*}
ranging over all \(k \in \Z_{\geq 0}\) and \(\bba = (a_1, \ldots, a_k) \in \Z_{\geq 0}^k\).  The objects of $\pmodSS$ are similar except some symmetric powers are replaced by their duals.

\subsection{Basic  \texorpdfstring{$\fp(n)$}{p(n)}-Module Maps}\label{SS:BasicHomomorphisms}

To connect our diagrammatic categories to the representation theory of $\fp (n)$ we introduce certain explicit $\fp(n)$-module homomorphisms in the categories $\pmodS$ and $\pmodSS$.

We can view the symmetric bisuperalgebra 
\begin{equation}\label{E:symmalg}
S(V_n) = \bigoplus_{k \in \Z_{\geq 0}}S^k(V_n) 
\end{equation}
as the enveloping superalgebra for the abelian Lie superalgebra $V_{n}$.  This endows $S(V_{n})$ with the structure of a $\Z$-graded Hopf superalgebra.  In particular, it admits an associative product \(m: S(V_n) \otimes S(V_n) \to S(V_n)\) and coassociative coproduct \(\Delta: S(V_n) \to S(V_n) \otimes S(V_n)\).  The product is the usual concatenation product and the coproduct is given on generators $v \in V_{n}$ by $\Delta (v) = v \otimes 1 + 1 \otimes v$.   We stress that the multiplication in $S(V_n)$ is supercommutative, meaning that $v w = (-1)^{\p{v}\p{w}} w v$ for all homogeneous $v,w \in V_n$.  Corresponding to the direct sum decomposition \cref{E:symmalg} there are also projections \(p_k:S(V_n) \to S^k(V_n)\) and inclusions \(\iota_k: S^k(V_n) \to S(V_n)\) for all \(k \in \Z_{\geq 0}\).  Each of the maps \(\Delta, m, p_k, \iota_k\) is a \(U(\fp(n))\)-module homomorphism and we use them to construct several module homomorphisms which will be used in the sequel.

We define the {\em split} \(U(\fp(n))\)-module morphism 
\begin{align*}
\textup{spl}^{a,b}_{a+b}: S^{a+b}(V_n) \to S^a(V_n) \otimes S^b(V_n)
\end{align*}
via \(\textup{spl}^{a,b}_{a+b}:= (p_a \otimes p_b) \circ \Delta \circ \iota_{a+b}\). Explicitly, we have
\begin{align*}
\textup{spl}^{a,b}_{a+b}(x_1 \cdots x_{a+b})= \sum_{\substack{T = \{t_1< \ldots< t_a\} \\ U = \{u_1< \ldots<u_b\} \\ T \cup U = \{1,\ldots, a+b\} }} (-1)^{\varepsilon(T,U)}x_{t_1} \cdots x_{t_a} \otimes x_{u_1} \cdots x_{u_b},
\end{align*}
for all homogeneous \(x_1, \ldots, x_{a+b} \in V_n\), where \(\varepsilon(T,U) \in \Z_2\) is defined by
\begin{align*}
\varepsilon(T,U)=\#\left\{ (t,u) \in T \times U \mid t>u, \,\bar{x}_t = \bar{x}_u = \bar 1\right\}.
\end{align*}

Similarly, define the {\em merge} \(U(\fp(n))\)-module morphism
\begin{align*}
\operatorname{mer}_{a,b}^{a+b}: S^a(V_n) \otimes S^b(V_n) \to S^{a+b}(V_n)
\end{align*}
via \(\textup{mer}_{a,b}^{a+b}:= p_{a+b} \circ m \circ (\iota_a \otimes \iota_b)\), or, explicitly,
\begin{align*}
\operatorname{mer}_{a,b}^{a+b}(x_1 \cdots x_a \otimes y_1 \dotsc y_b) = x_1 \cdots x_a y_1 \cdots y_b,
\end{align*}
for all \(x_1, \ldots, x_a, y_1, \ldots, y_b \in V_n\).  Both the split and merge maps are even (i.e., parity preserving).

As the odd bilinear form used to define $\fp (n)$ is supersymmetric, it factors through to define the odd  \emph{antenna} \(U(\fp(n))\)-module homomorphism \(\textup{ant}{}:S^2(V_n) \to \k\) given by
\begin{align*}
\operatorname{ant} (x_1x_2) = (x_1,x_2)
\end{align*}
for all \(x_1, x_2 \in V_n\). 

For any \(k \in \Z_{\geq 0}\) we have the \emph{evaluation} \(U(\fp(n))\)-module homomorphism 
\begin{gather*}
\operatorname{eval}_k: S^k(V_n)^* \otimes S^k(V_n) \to \k\\
 f \otimes x \mapsto f(x).
\end{gather*}
Dualizing the evaluation map yields the \emph{coevaluation} \(U(\fp(n))\)-module homomorphism 
\[
\textup{coeval}_k: \k \to S^k(V_n) \otimes S^k(V_n)^*.
\]
In particular, we have
\begin{gather*}
\operatorname{coeval}_1: \k \to V_n \otimes V_n^*, \\
 1 \mapsto \sum_{i \in I_m} v_i \otimes v_i^*,
\end{gather*}
where \(\{v_i^* \mid i \in I\}\) is the dual basis for \(V_n^*\) defined by $v_{i}^{*}(v_{j})=\delta_{i,j}$.

The odd nondegenerate bilinear form \((\cdot, \cdot)\) induces an odd \(U(\fp(n))\)-module isomorphism
\begin{gather*}
D:V_n \to V_{n}^*, \\
v_i \mapsto (v_{i}, -)= v_{-i}^*.
\end{gather*}
Using this isomorphism we define the odd \emph{cap} and \emph{cup} \(U(\fp(n))\)-module homomorphisms by
\begin{align*}
\cap:= & \operatorname{eval}_1 \circ (D \otimes \operatorname{id}) : V_n^{\otimes 2} \to \k,\\
\cup:= & (\operatorname{id} \otimes D^{-1})\circ \operatorname{coeval}_1 : \k \to V_n^{\otimes 2}.
\end{align*}
On our basis for $V_{n}$ these maps are given by
\begin{align*}
\cap(v_i \otimes v_j) = \delta_{i,-j}
\qquad
\qquad
\textup{and}
\qquad
\qquad
\cup(1)= \sum_{i \in I_m} (-1)^{\bar i} v_i \otimes v_{-i}.
\end{align*}

Finally, for any two $\fp (n)$-modules $M$ and $N$ we have the even ``tensor swap'' homomorphism
\begin{gather*}
\tau_{M,N}:  M \otimes N \to N \otimes M, \\
m \otimes n  \mapsto (-1)^{\p{m} \p{n}}n \otimes m,
\end{gather*}
for all homogeneous $m\in M$ and $n\in N$. Note that 
\begin{align}\label{tswaprel}
\tau_{V_n,V_n} = (\operatorname{spl}_{2}^{1,1} \circ \operatorname{mer}_{1,1}^2) - (\operatorname{id}_1 \otimes \operatorname{id}_1).
\end{align} Compare with \cref{CrossDef} when $a=b=1$.  

Note that $\spl^{a,b}_{a+b}$, $\mer^{a+b}_{a,b}$, $\ev_{k}$, $\coev_{k}$, and the tensor swap are in fact $\gl (V)$-equivariant.  On the other hand, $\textup{ant}$, the odd cup, and the odd cap are only $\fp (n)$-equivariant. 
\section{From Webs to \texorpdfstring{$\fp(n)$\text{-Modules}}{p(n)-Modules}}\label{SS:WebstoP(n)modules}

\subsection{The Functor  \texorpdfstring{$G: \pWeb \to \fp(n)\text{-mod}_{\mathcal{S}}$}{G: p-Web -> p(n)-mod}}\label{SS:GFunctorWebs}  Unless otherwise stated, in this section $\k$ is a field of characteristic not two.  Recall $\pmodS$ denotes the monoidal supercategory of $\fp (n)$-modules generated by symmetric powers of the natural module \(V_n\). 

\begin{theorem}\label{Gthm}
There is a well-defined functor 
\begin{align*}
G: \pWeb \to \fp(n)\textup{-mod}_{\mathcal{S}}
\end{align*}
given on objects by \(G(k) = S^k(V_n)\) and on morphisms by
\begin{align*}
\hackcenter{
{}
}
\hackcenter{
\begin{tikzpicture}[scale=0.8]
  \draw[thick, color=\clr] (0,0)--(0,0.2) .. controls ++(0,0.35) and ++(0,-0.35) .. (-0.4,0.9)--(-0.4,1);
  \draw[thick, color=\clr] (0,0)--(0,0.2) .. controls ++(0,0.35) and ++(0,-0.35) .. (0.4,0.9)--(0.4,1);
      \node[above] at (-0.4,1) {$\scriptstyle a$};
      \node[above] at (0.4,1) {$\scriptstyle b$};
      \node[above] at (0,-.5) {$\scriptstyle a+b$};
\end{tikzpicture}}
\mapsto 
\textup{spl}_{a+b}^{a,b}
\qquad
\hackcenter{
\begin{tikzpicture}[scale=0.8]
  \draw[thick, color=\clr] (-0.4,0)--(-0.4,0.1) .. controls ++(0,0.35) and ++(0,-0.35) .. (0,0.8)--(0,1);
\draw[thick, color=\clr] (0.4,0)--(0.4,0.1) .. controls ++(0,0.35) and ++(0,-0.35) .. (0,0.8)--(0,1);
      \node[above] at (-0.4,-0.47) {$\scriptstyle a$};
      \node[above] at (0.4,-0.47) {$\scriptstyle b$};
      \node[above] at (0,1) {$\scriptstyle a+b$};
\end{tikzpicture}}
\mapsto
\textup{mer}_{a,b}^{a+b}
\qquad
\hackcenter{
\begin{tikzpicture}[scale=0.8]
 \draw (0,0) arc (0:180:0.4cm) [thick, color=\clr];
    \node[below] at (0,0) {$\scriptstyle 1$};
      \node[below] at (-0.8,0) {$\scriptstyle 1$};
                \node at (-0.4,0.4) {$\scriptstyle\blacklozenge$};
\end{tikzpicture}}
\mapsto
\cap
\qquad
\hackcenter{
\begin{tikzpicture}[scale=0.8]
 \draw (0,0) arc (0:-180:0.4cm) [thick, color=\clr];
    \node[above] at (0,0) {$\scriptstyle 1$};
      \node[above] at (-0.8,0) {$\scriptstyle 1$};
       \node at (-0.4,-0.4) {$\scriptstyle\blacklozenge$};
\end{tikzpicture}}
\mapsto
\cup .
\end{align*}
\end{theorem}
\begin{proof}
We simply check that images of relations \cref{AssocRel}--\cref{CupDiagSwitchRel2} are preserved by \(G\). This is routine, but requires some care in managing signs. Details are included in the {\tt arXiv} version of this paper, as explained in \cref{SS:ArxivVersion}.
\begin{answer}

We check that images of relations (\ref{AssocRel}--\ref{CupDiagSwitchRel2}) are preserved by \(G\). First, note that we have:
\begin{align*}
\hackcenter{}
G:
\hackcenter{
\begin{tikzpicture}[scale=0.8]
 \draw[thick, color=\clr] (0,0)--(0,0.5);
    \node[below] at (0,0) {$\scriptstyle 2$};
    \node[shape=coordinate](DOT) at (0,0.5) {};
     \filldraw  (DOT) circle (2.5pt);
\end{tikzpicture}}
\mapsto
\textup{ant}.
\end{align*}

{\em Relation \cref{AssocRel}}. Note that since \(\Delta\) is a {\em graded} map, we have 
\begin{align*}
(p_a \otimes p_b) \circ \Delta \circ \iota_{a+b} \circ p_{a+b} = (p_a \otimes p_b) \circ \Delta,
\end{align*}
for all \(a,b \in \Z_{\geq 0}\).
It follows then that the split morphisms satisfy a `coassociativity' condition:
\begin{align*}
(\textup{id}_{S^a(V_n)} \otimes \textup{spl}_{b+c}^{b,c}) \circ \textup{spl}_{a+b+c}^{a,b+c}
&=
(p_a \otimes p_b \otimes p_c) \circ (\textup{id}_{S(V_n)} \otimes (\Delta \circ \iota_{b+c} \circ p_{b+c})) \circ \Delta \circ \iota_{a+b+c}\\
&=
(p_a \otimes p_b \otimes p_c) \circ (\textup{id}_{S(V_n)} \otimes \Delta) \circ \Delta \circ \iota_{a+b+c}\\
&=(p_a \otimes p_b \otimes p_c) \circ (\Delta \otimes \textup{id}_{S(V_n)}) \circ \Delta \circ \iota_{a+b+c}\\
&=(p_a \otimes p_b \otimes p_c) \circ ((\Delta \circ \iota_{a+b} \circ p_{a+b}) \otimes \textup{id}_{S(V_n)}) \circ \Delta \circ \iota_{a+b+c}\\
&=(\textup{spl}_{a+b}^{a,b} \otimes \textup{id}_{S^c(V_n)}) \circ \textup{spl}_{a+b+c}^{a+b,c}.
\end{align*}
Similarly, `associativity' of the merge morphisms,
\begin{align*}
\textup{mer}_{a,b+c}^{a+b+c} \circ (\textup{id}_{S^a(V_n)} \otimes \textup{mer}_{b,c}^{b+c})
=
\textup{mer}_{a+b,c}^{a+b+c} \circ (\textup{mer}_{a,b}^{a+b} \otimes \textup{id}_{S^c(V_n)})
\end{align*}
follows from the associativity of \(m\) in a similar fashion, so we have that \(G\) preserves \cref{AssocRel}.

{\em Relation \cref{StraightRel}}. For all \(i \in I_n\) we have
\begin{align*}
(\textup{id}_1 \otimes \cap) \circ (\cup \otimes \textup{id}_1)(v_i )
&=(\textup{id}_1 \otimes \cap)\left(
\sum_{k \in I_m} (-1)^{\bar k} v_k \otimes v_{-k} \otimes v_i
\right)
=
\sum_{k \in I_m}  \delta_{i,k}v_k
=
v_i,
\end{align*}
and
\begin{align*}
(\cap \otimes \textup{id}_1) \circ (\textup{id}_1 \otimes \cup)(v_i)
&=
(\cap \otimes \textup{id}_1)
\left(
\sum_{k \in I_m} (-1)^{\bar i + \bar k}v_i \otimes v_k \otimes v_{-k}
\right)
=
\sum_{k \in I_m}
(-1)^{\bar i + \bar k}\delta_{k,-i} v_{-k}
=
-v_i.
\end{align*}

{\em Relation \cref{AntRel}}. For all \(i,j \in I_m\) we have
\begin{align*}
\textup{ant}\circ \textup{mer}_{1,1}^2(v_i \otimes v_j)
=
\textup{ant}(v_iv_j)
=
(v_i,v_j)
=
\delta_{i,-j}
=
\cap(v_i \otimes v_j).
\end{align*}

{\em Relation \cref{DiagSwitchRel}}. Write \(A \) (resp.\  \(B\)) for the image of the morphism on the left side (resp.\  right side) of \cref{DiagSwitchRel} under the functor \(G\).
Let \(x_1, \ldots, x_a, y_1, \ldots, y_b\) be homogeneous elements in \(V_n\). Then 
\begin{align*}
A=
(\textup{mer}_{a-s,r}^{a-s+r} \otimes \textup{id}_{S^{b+s-r}(V_n)})
\circ
(\textup{id}_{S^{a-s}(V_n)} \otimes \textup{spl}_{b+s}^{r,b+s-r})
\circ
(\textup{id}_{S^{a-s}(V_n)} \otimes \textup{mer}_{s,b}^{b+s})
\circ
(\textup{spl}_a^{a-s,s} \otimes \textup{id}_{S^b(V_n)})
\end{align*}
maps \(x_1 \cdots x_a \otimes y_{1} \cdots y_{b}\) to
\begin{align}\label{RSum1}
\sum
(-1)^{\varepsilon(T,U) + \varepsilon(L,M) + \varepsilon(P,Q) + \nu(M,P)}
x_{t_1} \cdots x_{t_{a-s}}x_{l_1} \cdots x_{l_c} y_{p_1} \cdots y_{p_{r-c}}\otimes x_{m_1} \cdots x_{m_{s-c}}y_{q_1} \cdots y_{q_{b-r+c}}
\end{align}
where the sum ranges over
\begin{align*}
c\in \Z_{\geq 0}; \;\;\;\;
T=\{t_1 < \cdots < t_{a-s}\}; \;\;\;\;
U=\{u_1 < \cdots < u_{s}\}; \;\;\;\;
\end{align*}
\begin{align*}
L=\{l_1 < \cdots < l_c\}; \;\;\;\;
M=\{m_1 < \cdots < m_{s-c}\}; \;\;\;\;
P=\{p_1 < \cdots < p_{r-c}\}; \;\;\;\;
Q=\{q_1 < \cdots < q_{b-r+c}\};
\end{align*}
\begin{align*}
T\cup U = \{1, \ldots, a\}; \;\;\;\;
L\cup M = U; \;\;\;\;
P \cup Q = \{1, \ldots, b\},
\end{align*}
and \(\nu(M,P)\) is defined by
\begin{align*}
\nu(M,P) = \sum_{\substack{m \in M \\ p \in P}} \bar{x}_m \bar{y}_p.
\end{align*}
Setting \(J = \{j_1, \ldots, j_{a-s+c}\} = T \cup L\), we have that
\begin{align*}
\varepsilon(T,U) + \varepsilon(L,M) &= \varepsilon(T,L \cup M) + \varepsilon(L,M) 
= \varepsilon(T,L)+\varepsilon(T,M) + \varepsilon(L,M) \\
&=\varepsilon(T,L) + \varepsilon(T \cup L, M) 
=   \varepsilon(T,L)+\varepsilon(J,M) , 
\end{align*}
and 
\begin{align*}
 (-1)^{\varepsilon(T,L)}x_{j_1} \cdots x_{j_{a-s+c}} = x_{t_1} \cdots x_{t_{a-s}} x_{l_1} \cdots x_{l_c},
\end{align*}
so it follows that we may rewrite \cref{RSum1} as 
\begin{align}\label{RSum2}
\sum
\binom{a-s+c}{c}
(-1)^{\varepsilon(J,M) + \varepsilon(P,Q)+ \nu(M,P)}
x_{j_1} \cdots x_{j_{a-s+c}} y_{p_1} \cdots y_{p_{r-c}}\otimes x_{m_1} \cdots x_{m_{s-c}}y_{q_1} \cdots y_{q_{b-r+c}},
\end{align}
where the sum ranges over
\begin{align} \label{RSum3}
c\in \Z_{\geq 0}; \;\;\;\;
J=\{j_1< \ldots< j_{a-s+c}\}; \;\;\;\;
M=\{m_1< \ldots< m_{s-c}\};
\end{align}
\begin{align*}
P=\{p_1< \ldots< p_{r-c}\}; \;\;\;\;
Q=\{q_1< \ldots< q_{b-r+c}\}; \;\;\;\;
J\cup M = \{1,\ldots, a\}; \;\;\;\;
P \cup Q = \{1, \ldots, b\}.
\end{align*}

Now, for any \(t \in \Z_{\geq 0}\), we compute in a similar fashion that 
\begin{align*}
(\textup{id}_{a-s+r} \otimes \textup{mer}_{s-t,b-r+t}^{b+s-r})
\circ
(\textup{spl}_{a+r-t}^{a-s+r,s-t} \otimes \textup{id}_{b-r+t})
\circ
(\textup{mer}_{a,r-t}^{a+r-t} \otimes \textup{id}_{b-r+t})
\circ
(\textup{id}_a \otimes \textup{spl}_b^{r-t,b-r+t})
\end{align*}
sends \(x_1 \cdots x_a \otimes y_1 \cdots y_b\) to
\begin{align}\label{RSum4}
\sum
\binom{b-r+t+k}{k}
(-1)^{\varepsilon(J,M)+\varepsilon(P,Q)+ \nu(M,P)}
x_{g_1} \cdots x_{g_{a-s+t+k}} y_{p_1} \cdots y_{p_{r-t-k}}\otimes x_{m_1} \cdots x_{m_{s-t-k}}y_{q_1} \cdots y_{q_{b-r+t+k}},
\end{align}
where the sum ranges over
\begin{align*}
k \in \Z_{\geq 0}; \;\;\;\;
J=\{j_1 < \cdots < j_{a-s+t+k}\}; \;\;\;\;
M=\{m_1 < \cdots < m_{s-t-k}\}; \;\;\;\;
\end{align*}
\begin{align*}
P=\{p_1 < \cdots <p_{r-t-k}\};\;\;\;\;
Q = \{q_1 < \cdots < q_{b-r+t+k}\}\;\;\;\;
J\cup M = \{1, \ldots, a\}\;\;\;\;
P \cup Q = \{1, \ldots, b\}.
\end{align*}
Let \(c,J,M,P,Q\) be  as in \cref{RSum3}. Then by \cref{RSum4}, the coefficient of 
\begin{align*}
x_{j_1} \cdots x_{j_{a-s+c}} y_{p_1} \cdots y_{p_{r-c}}\otimes x_{m_1} \cdots x_{m_{s-c}}y_{q_1} \cdots y_{q_{b-r+c}}
\end{align*}
in \(B(x_1 \cdots x_a \otimes y_1 \cdots y_b)\) is given by
\begin{align*}
(-1)^{\varepsilon(J,M) + \varepsilon(P,Q) + \nu(M,P)}
\sum_{d=0}^{c}
\binom{a-b + r -s}{c-d}
\binom{b-r+c}{d}.
\end{align*}
By the Chu-Vandermonde identity, this is equal to 
\begin{align*}
(-1)^{\varepsilon(J,M) + \varepsilon(P,Q) + \nu(M,P)} \binom{a-s + c}{c},
\end{align*}
which by \cref{RSum2} is the coefficient of the same term in \(A(x_1 \cdots x_a \otimes y_1 \cdots y_b)\). It follows that \(G\) preserves \cref{DiagSwitchRel}.

{\em Relation (\ref{CapDiagSwitchRel1}--\ref{CapDiagSwitchRel2})}.
Thanks to relation \cref{AssocRel}, we only need check that \cref{CapDiagSwitchRel1} holds in the case \(a=r+1\). The map
\(
(\textup{id}_1 \otimes \textup{mer}_{r,b}^{b+r})\circ(\textup{spl}_{r+1}^{1,r} \otimes \textup{id}_{b})
\)
sends \(x_1 \cdots x_{r+1} \otimes y_1 \cdots y_b\) to 
\begin{align*}
\sum_{t=1}^{r+1}
(-1)^{\bar{x}_t(\bar{x}_1 + \cdots + \bar{x}_{t-1})}x_t \otimes x_1 \cdots x_{t-1} x_{t+1} \cdots x_{r+1}y_1 \cdots y_b.
\end{align*}
Applying \(  
 ( \cap \otimes \textup{id}_{b+r-1}  )  \circ (\textup{id}_1 \otimes \textup{spl}_{b+r}^{1,b+r-1})
\)
to this sum gives
\begin{align*}
&\sum_{
\substack{
1 \leq t \leq r+1\\
1 \leq p \leq b
}}
(-1)^{\bar{x}_t(\bar{x}_1 + \cdots + \bar{x}_{t-1})}
(-1)^{\bar{y}_p(\bar{x}_1 + \cdots + \hat{\bar{x}}_t + \cdots + \bar{x}_{r+1} +\bar{y}_1 + \cdots + \bar{y}_{p-1})}
(x_t,y_p) x_1 \cdots \hat{x}_t \cdots x_{r+1}y_1 \cdots \hat{y}_p \cdots y_b\\
& \hspace{5mm}+
\sum_{1 \leq t > u\leq r+1}
(-1)^{\bar{x}_t(\bar{x}_1 + \cdots + \bar{x}_{t-1})}
(-1)^{\bar{x}_u(\bar{x}_1+ \cdots + \bar{x}_{u-1})}
(x_t, x_u) x_1 \cdots \hat{x}_u \cdots \hat{x}_t \cdots x_{r+1} y_1 \cdots y_b\\
& \hspace{10mm}+
\sum_{1 \leq t < u\leq r+1}
(-1)^{\bar{x}_t(\bar{x}_1 + \cdots + \bar{x}_{t-1})}
(-1)^{\bar{x}_u(\bar{x}_1+ \cdots + \hat{\bar{x}}_t + \cdots + \bar{x}_{u-1})}
(x_t, x_u) x_1 \cdots \hat{x}_t \cdots \hat{x}_u \cdots x_{r+1} y_1 \cdots y_b
\end{align*}
Exchanging \(t\) and \(u\) in the second line and using the fact that \((x_u,x_t) = (-1)^{\bar{x}_t \bar{x}_u}(x_t,x_u)\) and \((x,y) = 0\) unless \(\bar y = \bar x + \bar 1\), we may rewrite this as 
\begin{align}\label{RSum11}
&\sum_{
\substack{
1 \leq t \leq r+1\\
1 \leq p \leq b
}}
(-1)^{(\bar{x}_t + \bar 1)(\bar{x}_{t+1} + \cdots + \bar{x}_{r+1})   + (\bar{x}_t + \bar 1)(\bar{y}_1 + \cdots + \bar{y}_{p-1}) + \bar{x}_1 + \cdots + \bar{x}_{t-1}
}
(x_t,y_p) x_1 \cdots \hat{x}_t \cdots x_{r+1}y_1 \cdots \hat{y}_p \cdots y_b\\
& \hspace{5mm}+
2\hspace{-3mm}\sum_{1 \leq t < u\leq r+1}
(-1)^{\bar{x}_t(\bar{x}_{t+1} + \cdots + \bar{x}_{u-1}) + \bar{x}_1 + \cdots + \hat{\bar{x}}_{t} + \cdots + \bar{x}_{u-1})}
(x_t, x_u) x_1 \cdots \hat{x}_t \cdots \hat{x}_u \cdots x_{r+1} y_1 \cdots y_b. \nonumber
\end{align}

On the other hand, the map
\begin{align*}
\textup{mer}_{r,b-1}^{b+r-1} \circ \cap \circ (\textup{spl}_{r+1}^{r,1} \otimes \textup{spl}_{b}^{1,b-1})
\end{align*}
sends \(x_1 \cdots x_{r+1} \otimes y_1 \cdots y_b\) to 
\begin{align*}
\sum_{t,p} (-1)^{\bar{x}_t(\bar{x}_{t+1}+ \cdots + \bar{x}_{r+1})} 
(-1)^{\bar{y}_p(\bar{y}_1 + \cdots + \bar{y}_{p-1})}
(-1)^{\bar{x}_1 + \cdots + \hat{\bar{x}}_t + \cdots + \bar{x}_{r+1}}
(x_t, y_p) x_1 \cdots \hat{x}_t \cdots x_{r+1} y_1 \cdots \hat{y}_p \cdots y_b.
\end{align*}
After rewriting using the fact that \((x_t, y_p) = 0\) unless \(\bar{y}_p = \bar{x}_t + \bar 1\), this is equal to the first line in \cref{RSum11}.

The map 
\begin{align*}
\textup{mer}_{r-1,b}^{b+r-1} \circ  ( \textup{id}_{r-1} \otimes \textup{ant} \otimes \textup{id}_b)  \circ ( \textup{spl}_{r+1}^{r-1,2} \otimes \textup{id}_b)
\end{align*}
sends \(x_1 \cdots x_{r+1} \otimes y_1 \cdots y_b\) to 
\begin{align*}
\sum_{1 \leq t<u \leq r+1}
(-1)^{\mu(t,u)}
(x_t,x_u)x_1 \cdots \hat{x}_t \cdots \hat{x}_u \cdots x_{r+1} y_1 \cdots y_b,
\end{align*}
where
\begin{align*}
\mu(t,u)=
(-1)^{\bar{x}_u(\bar{x}_{u+1} + \cdots + \bar{x}_{r+1})}
(-1)^{\bar{x}_t(\bar{x}_{t+1} + \cdots + \hat{\bar{x}}_u + \cdots + \bar{x}_{r+1})}
(-1)^{\bar{x}_1 + \cdots + \hat{\bar{x}}_t + \cdots + \hat{\bar{x}}_u + \cdots + \bar{x}_{r+1}}
\end{align*}
After rewriting using the fact that \((x_t, x_u) = 0\) unless \(\bar{x}_u = \bar{x}_t + \bar 1\), this is equal to half the second line in \cref{RSum11}. This completes the proof of relation \cref{CapDiagSwitchRel1}. Relation \cref{CapDiagSwitchRel2} is similar.

{\em Relation (\ref{CupDiagSwitchRel1}--\ref{CupDiagSwitchRel2})}. Thanks to relation \cref{AssocRel}, we only need check that \cref{CupDiagSwitchRel1} holds in the case \(b=0\). The map 
\(
(\textup{mer}_{a,1}^{a+1} \otimes \textup{id}_1)\circ (\textup{id}_a \otimes \cup)
\)
sends \(x_1 \cdots x_a\) to 
\begin{align*}
\sum_{i \in I_m} (-1)^{\bar{x}_1 + \cdots + \bar{x}_a}(-1)^{\bar i}x_1 \cdots x_a  v_i\otimes v_{-i}.
\end{align*}
Applying 
\(
(\textup{id}_{a-r+1} \otimes \textup{mer}_{r,1}^{r+1})\circ (\textup{spl}_{a+1}^{a-r+1,r} \otimes \textup{id}_1)
\) to this sum gives
\begin{align}\label{RSum21}
&\sum_{
\substack{
i \in I_m\\
T = \{t_1 < \cdots < t_{a-r}\}\\
U=\{u_1 < \cdots < u_r\}\\
T \cup U = \{1, \ldots ,a\}
}}
(-1)^{\varepsilon(T,U)}(-1)^{\bar{i}(\bar{x}_{u_1} + \cdots + \bar{x}_{u_r})}(-1)^{\bar{x}_1 + \cdots + \bar{x}_a}(-1)^{\bar i}x_{t_1} \cdots x_{t_{a-r}}  v_i\otimes 
x_{u_1} \cdots x_{u_r}
v_{-i}\\
&\hspace{5mm}+
\sum_{
\substack{
i \in I_m \\
T=\{t_1 < \cdots <t_{a-r+1}\}\\
U=\{u_1 < \cdots < u_{r-1}\}\\
T \cup U = \{1, \ldots, a\}
}}
(-1)^{\varepsilon(T,U)}
(-1)^{\bar{x}_1 + \cdots + \bar{x}_a}(-1)^{\bar i}
x_{t_1} \cdots x_{t_{a-r+1}} \otimes x_{u_1} \cdots x_{u_{r-1}} v_i v_{-i}.\nonumber
\end{align}
The second line is zero, since \((-1)^{\bar i}v_i v_{-i} = -(-1)^{\bar{i}+ \bar 1}v_{-i}v_{i}\) for all \(i \in I_m\).

On the other hand, applying
\begin{align*}
(\textup{mer}_{a-r,1}^{a-r+1} \otimes \textup{mer}_{r,1}^{r+1})
\circ
(\textup{id}_{a-r} \otimes \cup \otimes \textup{id}_r )\circ \textup{spl}_a^{a-r,r}
\end{align*}
to \(x_1 \cdots x_a\) yields
\begin{align}\label{RSum22}
\sum_{
\substack{
i \in I_m\\
T = \{t_1 < \cdots < t_{a-r}\}\\
U=\{u_1 < \cdots < u_{r}\}\\
T \cup U = \{1, \ldots, a\}
}}
(-1)^{\varepsilon(T,U)}
(-1)^{\bar{x}_{t_1} + \cdots + \bar{x}_{t_{a-r}}}
(-1)^{\bar i}
x_{t_1} \cdots x_{t_{a-r}}v_i \otimes v_{-i}x_{u_1} \cdots x_{u_r}.
\end{align}
Using the fact that 
\begin{align*}
v_{-i}x_{u_1} \cdots x_{u_r} = (-1)^{(\bar i + \bar 1)(\bar{x}_{u_1} + \cdots + \bar{x}_{u_r})} x_{u_1} \cdots x_{u_r} v_{-i},
\end{align*}
we may rewrite \cref{RSum22} to see that it is equivalent to the first line of \cref{RSum21}, completing the proof of relation \cref{CupDiagSwitchRel1}. The proof of  \cref{CupDiagSwitchRel2} is similar.
\end{answer}
\end{proof}

\subsection{The Crossing Morphism in  \texorpdfstring{$\pmodS$}{p(n)-mods}}  
For short, let 
\[
\tau_{a,b}: S^{a}(V_{n}) \otimes S^{b}(V_{n}) \to  S^{b}(V_{n}) \otimes S^{a}(V_{n})
\] be the tensor swap map introduced in \cref{SS:BasicHomomorphisms}.
\begin{lemma}\label{CrossTwist}
For all \(a,b \in \Z_{\geq 0}\), we have 
\begin{align*}
\hackcenter{}
G \left(
\hackcenter{
\begin{tikzpicture}[scale=0.8]
  \draw[thick, color=\clr] (0.4,0)--(0.4,0.2) .. controls ++(0,0.35) and ++(0,-0.35) .. (-0.4,0.9)--(-0.4,1);
  \draw[thick, color=\clr] (-0.4,0)--(-0.4,0.2) .. controls ++(0,0.35) and ++(0,-0.35) .. (0.4,0.9)--(0.4,1);
      \node[below] at (-0.4,0) {$\scriptstyle a$};
      \node[below] at (0.4,0.05) {$\scriptstyle b$};
      \node[above,white] at (-0.5,1) {$\scriptstyle a$};
      \node[above, white] at (0.5,.1) {$\scriptstyle b$};
\end{tikzpicture}} \right)
=
\tau_{a,b}.
\end{align*}
\end{lemma}
\begin{proof}
This is well-known $\aWeb$ (e.g., see \cite{TVW} where it is done for quantum $\gl (V)$).  It is also routine to check directly.  Details are included in the {\tt arXiv} version.
\begin{answer}
By \cref{CrossDef} and \cref{RSum2}, we have that \(G\Big(\hackcenter{}\hackcenter{
\begin{tikzpicture}[scale=0.7]
  \draw[thick, color=\clr] (0.2,0)--(0.2,0.1) .. controls ++(0,0.15) and ++(0,-0.15) .. (-0.2,0.45)--(-0.2,0.5);
  \draw[thick, color=\clr] (-0.2,0)--(-0.2,0.1) .. controls ++(0,0.15) and ++(0,-0.15) .. (0.2,0.45)--(0.2,0.5);
      \node[below] at (-0.2,0) {$\scriptstyle a$};
      \node[below] at (0.2,0.05) {$\scriptstyle b$};
\end{tikzpicture}}
\Big)
\)
sends \(x_1 \cdots x_a \otimes y_1 \cdots y_b\) to 
\begin{align}\label{CSum3}
\sum
\binom{a-s+c}{c}
(-1)^{\varepsilon(J,M) + \varepsilon(P,Q)+ \nu(M,P)+a-s}
x_{j_1} \cdots x_{j_{a-s+c}} y_{p_1} \cdots y_{p_{r-c}}\otimes x_{m_1} \cdots x_{m_{s-c}}y_{q_1} \cdots y_{q_{b-r+c}},
\end{align}
where the sum ranges over \(s-r = a-b\) and \cref{RSum3}. Now, letting \(z=s-c\), we have that the coefficient of 
\begin{align*}
x_{j_1} \cdots x_{j_{a-s+c}} y_{p_1} \cdots y_{p_{r-c}}\otimes x_{m_1} \cdots x_{m_{s-c}}y_{q_1} \cdots y_{q_{b-r+c}}
\end{align*}
for a fixed \(J,M,P,Q\) as in \cref{RSum3} is given by
\begin{align*}
(-1)^{\varepsilon(J,M) + \varepsilon(P,Q)+ \nu(M,P)+a+z}\sum_{c=0}^{a-z}(-1)^c
\binom{a-z}{c}.
\end{align*}
By the binomial theorem, this is zero unless \(a=z\), in which case \(J= \varnothing\), \(M= \{1,\ldots, a\}\), \(P=\{1, \ldots, b\}\), \(Q= \varnothing\), \(\varepsilon(J,M) = \varepsilon(P,Q)=0\), and \((-1)^{\nu(M,P)} = (-1)^{(|x_1| + \cdots + |x_a|)(|y_1| + \cdots + |y_b|)}\), so we have that 
\begin{align*}
G\Big(\hackcenter{}\hackcenter{
\begin{tikzpicture}[scale=0.7]
  \draw[thick, color=\clr] (0.2,0)--(0.2,0.1) .. controls ++(0,0.15) and ++(0,-0.15) .. (-0.2,0.45)--(-0.2,0.5);
  \draw[thick, color=\clr] (-0.2,0)--(-0.2,0.1) .. controls ++(0,0.15) and ++(0,-0.15) .. (0.2,0.45)--(0.2,0.5);
      \node[below] at (-0.2,0) {$\scriptstyle a$};
      \node[below] at (0.2,0) {$\scriptstyle b$};
\end{tikzpicture}}
\Big)
(x_1 \cdots x_a \otimes y_1 \cdots y_b)
=
(-1)^{(|x_1| + \cdots + |x_a|)(|y_1| + \cdots + |y_b|)}
y_1 \cdots y_b \otimes x_1 \cdots x_a,
\end{align*}
as desired.
\end{answer}
\end{proof}

\subsection{Basis Theorems for \texorpdfstring{$\aWeb$}{a-Web} and \texorpdfstring{$\pWeb$}{p-Web} }\label{SS:BasisTheorems}  We now prove the sets introduced in \cref{SS:BasisForPWebs} form $\k$-bases for the morphism spaces of $\aWeb$ and $\pWeb$.

\begin{theorem}\label{actind}
Assume \(|\bba| + |\bbb| \leq 2n\). Then 
\begin{align*}
\left\{  \left. G(\xi^{(A,B,C,D)})\;  \right| (A,B,C,D) \in \chi(\bba, \bbb)  \right\}
\end{align*}
is a family of linearly independent morphisms in \(\Hom_{\fp(n)}(S^{\bba}(V_n),S^{\bbb}(V_n))\).
\end{theorem}
\begin{proof}
Let \((A,B,C,D) \in \chi(\bba,\bbb)\). For all \(i=1, \ldots, t\) and \(j=1,\ldots,u\), define
\begin{align*}
{\tt r}_i = \sum_{\ell =1}^u C_{i\ell},
\qquad
\qquad
|C| = \sum_{k =1}^t {\tt r}_k,
\qquad
\text{ and }
\qquad
P_{ij} = \sum_{k=1}^{i-1} {\tt r}_k + \sum_{\ell =1 }^{j-1} C_{i \ell}.
\end{align*}
For $X \in \left\{A, B \right\}$ let
\begin{align*}
\left( (i_1^{X}, j_1^{X}), (i_2^{X}, j_2^{X}), \ldots, (i_\alpha^{X}, j_\alpha^{X})  \right)
\end{align*}
be an irredundant list of all pairs of indices \((i,j)\) such that \(i<j\) and \(X_{ij} = 1\).
%
%\begin{align*}
%\left( (i_1^A, j_1^A), (i_2^A, j_2^A), \ldots, (i_\alpha^A, j_\alpha^A)  \right)
%\end{align*}
%be an irredundant list of all pairs of indices \((i,j)\) such that \(i<j\) and \(A_{ij} = 1\). Let 
%\begin{align*}
%\left( (i_1^B, j_1^B), (i_2^B, j_2^B), \ldots, (i_\beta^B, j_\beta^B)  \right)
%\end{align*}
%be an irredundant list of all pairs of indices \((i,j)\) such that \(i<j\) and \(B_{ij} = 1\).
%
Let 
\begin{align*}
(i_1^D, i_2^D, \ldots, i_\delta^D)
\end{align*}
be an irredundant list of all indices \(i\) such that \(D_i = 1\). It follows from the fact that \(|\bba| + |\bbb| \leq 2n\) and the definition of the set \(\chi(\bba,\bbb)\) that \(|C|+\alpha + \beta + \delta \leq n\).

We define the following elements of \(S(V_n)^{\otimes t}\):
\begin{align*}
v^{(A,B,C,D),1}&:=
v_1 \cdots v_{{\tt r}_1}
\otimes 
v_{{\tt r}_1 +1} \cdots v_{{\tt r}_1 + {\tt r}_2}
\otimes 
\cdots 
\otimes
v_{{\tt r}_1 + \cdots + {\tt r}_{t-1}+1} \cdots v_{|C|}
\\
v^{(A,B,C,D),2}&:=
\prod_{k=1}^\alpha 1 \otimes \cdots \otimes 1 \otimes v_{|C|+ k} \otimes 1 \otimes \cdots \otimes 1 \otimes v_{-|C|-k} \otimes 1 \otimes \cdots \otimes 1,
\\
v^{(A,B,C,D),3}&:=
\prod_{k=1}^\delta 1 \otimes \cdots \otimes 1 \otimes v_{|C| + \alpha + k} v_{-|C| - \alpha - k} \otimes 1 \otimes \cdots \otimes 1,
\end{align*}
where the vectors \(v_{|C|+ k}\) and \( v_{-|C|-k}\) appear in the \(i_k^A\)-th and \(j_k^A\)-th slots, respectively, and the term \(v_{|C| + \alpha + k} v_{-|C| - \alpha - k}\) appears in the \(i_k^D\)-th slot. 

We also define the following elements of \(S(V_n)^{\otimes u}\):
\begin{align*}
w^{(A,B,C,D),1}
&:=
\prod_{i=1}^t
\prod_{j=1}^u
1 \otimes \cdots 1 \otimes v_{P_{ij}+1} \cdots v_{P_{ij}+C_{ij}} \otimes 1 \otimes \cdots \otimes 1,
\\
w^{(A,B,C,D),2}
&:=
\prod_{k=1}^\beta 1 \otimes \cdots \otimes 1 \otimes v_{|C|+ k} \otimes 1 \otimes \cdots \otimes 1 \otimes v_{-|C|-k} \otimes 1 \otimes \cdots \otimes 1
\end{align*}
where the term \(v_{P_{ij}+1} \cdots v_{P_{ij}+C_{ij}}\) appears in the \(j\)-th slot, and the vectors \(v_{|C|+ k}\) and \( v_{-|C|-k}\) appear in the \(i_k^B\)-th and \(j_k^B\)-th slots, respectively.

Considering \(S(V_n)^{\otimes t}\) and \(S(V_n)^{\otimes u}\) as associative algebras, we define
\begin{align*}
v^{(A,B,C,D)}&:=v^{(A,B,C,D),1 }\cdot v^{(A,B,C,D),2}\cdot v^{(A,B,C,D),3} \in S^{\bba}(V_n)\\
w^{(A,B,C,D)}&:=w^{(A,B,C,D),1 }\cdot w^{(A,B,C,D),2} \in S^{\bbb}(V_n).
\end{align*}

Since \(S^{\bbb}(V_n)\) has a \(\k\)-basis of tensor products of monomials in \(\{v_i \mid i \in I_n\}\), we may define a linear projection map
\begin{align*}
p_{(A,B,C,D)}: S^{\bbb}(V_n) \to \k\{ w^{(A,B,C,D)}\}.
\end{align*}

We define a partial order \(\succeq\) on \(\chi(\bba,\bbb)\) by setting
\begin{align*}
(A',B',C',D') \succeq (A,B,C,D)
\end{align*}
if and only if
\begin{align*}
A'_{ij} \leq A_{ij},
\qquad
B'_{ij} \leq B_{ij},
\qquad
C'_{ij} \geq C_{ij},
\qquad
D'_{i} \leq D_{i} 
\qquad
\textup{ for all }i,j.
\end{align*}

It is straightforward to check, with the aid of \cref{CrossTwist}, that
\begin{align}\label{actproj}
p_{(A,B,C,D)} \circ G(\xi^{(A',B',C',D')})(v^{(A,B,C,D)})
=
\begin{cases}
\pm w^{(A,B,C,D)} & \textup{if } (A',B',C',D') = (A,B,C,D)\\
0 & \textup{if } (A',B',C',D') \not \succeq (A,B,C,D).
\end{cases}
\end{align}

Now, assume that there exist nontrivial scalars \(c_{(A',B',C',D')} \in \k\) such that 
\begin{align*}
\sum_{(A',B',C',D') \in \chi(\bba,\bbb)} c_{(A',B',C',D')} G(\xi^{(A',B',C',D')}) = 0.
\end{align*}
Let \((A,B,C,D) \in \chi(\bba,\bbb)\) be maximal in the \(\succeq\) order such that \(c_{(A,B,C,D)} \neq 0\). Then we have by \cref{actproj} that
\begin{align*}
0=\sum_{(A',B',C',D') \in \chi(\bba,\bbb)} c_{(A',B',C',D')} p_{(A,B,C,D)} \circ G(\xi^{(A',B',C',D')})(v^{(A,B,C,D)}) = \pm c_{(A,B,C,D)}w^{(A,B,C,D)},
\end{align*}
a contradiction.
\end{proof}

\begin{corollary}\label{BasisThm}
The set 
\begin{align*}
\mathscr{B}:=\left\{ \left. \xi^{(A,B,C,D)} \right| (A,B,C,D) \in \chi(\bba, \bbb) \right\}
\end{align*}
is a \(\k\)-basis for \(\Hom_{\pWeb}(\bba,\bbb)\).
\end{corollary}
\begin{proof}
This follows by  \cref{webspan,actind}.
\end{proof}

The previous corollary along with the results of \cref{SS:IsomorphicHoms,T:pWebtopWebup} provide a basis theorem for $\pWeb_{\uparrow \downarrow}$, as well. The following result is also immediate.

\begin{corollary}\label{BasisThmA}
The set 
\begin{align*}
\mathscr{B}:=\left\{\left. \xi^{(0,0,C,0)} \right| (0,0,C,0) \in \chi(\bba, \bbb) \right\}
\end{align*}
is a \(\k\)-basis for \(\Hom_{\aWeb}(\bba,\bbb)\).
\end{corollary}
These morphism spaces have been studied in, e.g., \cite{RT,TVW}.  Also, this basis should be compared with the basis of ``reduced chicken foot diagrams'' given in \cite{BEPO}.

\begin{remark}\label{R:Generallinearindependence} The assumption that $\k$ is a field is for convenience and is not required for the basis theorems stated above.  Let $\k$ be an integral domain in which $2$ is invertible and let $V_{n}$ be the free $\k$-supermodule of rank $2n$ with homogenous basis as in \cref{SS:LSAoftypeP}.  A standard argument using Bergman's Diamond Lemma shows that  $S^{k}(V_{n})$ is a free $\k$-supermodule with the obvious basis.  Using this basis, one can verify that the maps given in \cref{SS:BasicHomomorphisms} and the functor $G$ are still defined, and that the above arguments go through without change.
\end{remark}

\subsection{ \texorpdfstring{$\pWeb$}{pWeb} and the Marked Brauer Category}\label{SS:Brauer}  We now explain how the marked Brauer category introduced in \cite{KT} can be viewed as a subcategory of $\pWeb$.  It should be noted that in \cite{KT} diagrams were read top-to-bottom, contrary to the convention here.  However, using the functor $\textup{refl}$ described in \cref{SS:IsomorphicHoms} one can easily translate between the two conventions.  In this section we assume $\k$ is a field of characteristic different from two.

\begin{definition} The marked Brauer category \(\mathcal{B}\) is the $\k$-linear strict monoidal supercategory generated by a single object \(\bullet\).  For \(k \in \Z_{\geq 0}\), we will use the notation \([k]\) to designate the object \(\bullet^{\otimes k}\). The category \(\mathcal{B}\) has generating morphisms:
\begin{align*}
\hackcenter{}
\hackcenter{
\begin{tikzpicture}[scale=0.8]
  \draw[thick, color=\clr] (0.4,0)--(0.4,0.1) .. controls ++(0,0.35) and ++(0,-0.35) .. (-0.4,0.9)--(-0.4,1);
  \draw[thick, color=\clr] (-0.4,0)--(-0.4,0.1) .. controls ++(0,0.35) and ++(0,-0.35) .. (0.4,0.9)--(0.4,1);
\end{tikzpicture}}: [2] \to [2],
\qquad
\qquad
\hackcenter{
\begin{tikzpicture}[scale=0.8]
 \draw (0,0) arc (0:180:0.4cm) [thick, color=\clr]; 
  \node at (-0.4,0.4) {$\scriptstyle\blacklozenge$};
\end{tikzpicture}}: [2] \to [0],
\qquad
\qquad
\hackcenter{
\begin{tikzpicture}[scale=0.8]
 \draw (0,0) arc (0:-180:0.4cm) [thick, color=\clr];
   \node at (-0.4,-0.4) {$\scriptstyle\blacklozenge$};
\end{tikzpicture}}: [0] \to [2].
\end{align*}
We call these morphisms \emph{twist}, \emph{cap}, and \emph{cup}, respectively. The \(\Z_2\)-grading is given by declaring twists to have parity \(\bar 0\), and caps and cups to have parity \(\bar 1\). The defining relations of \(\mathcal{B}\) are:\\
\begin{align}\label{MBstraight}
\hackcenter{}
\hackcenter{
\begin{tikzpicture}[scale=0.8]
  \draw[thick, color=\clr] (0,0)--(0,-0.8); 
 \draw (0,0) arc (0:180:0.4cm) [thick, color=\clr];
   \draw[thick, color=\clr] (-0.8,0)--(-0.8,-0.2); 
  \draw (-0.8,-0.2) arc (0:-180:0.4cm) [thick, color=\clr];
   \draw[thick, color=\clr] (-1.6,-0.2)--(-1.6,0.6); 
     \node at (-1.2,-0.6) {$\scriptstyle\blacklozenge$};
       \node at (-0.4,0.4) {$\scriptstyle\blacklozenge$};
\end{tikzpicture}}
\;
=
\;
\hackcenter{
\begin{tikzpicture}[scale=0.8]
   \draw[thick, color=\clr] (0,-0.8)--(0,0.6); 
\end{tikzpicture}}
=
-
\;
\hackcenter{
\begin{tikzpicture}[scale=0.8]
   \draw[thick, color=\clr] (-0.8,0)--(-0.8,-0.8); 
 \draw (0,0) arc (0:180:0.4cm) [thick, color=\clr];
   \draw[thick, color=\clr] (0,0)--(0,-0.2); 
  \draw (0.8,-0.2) arc (0:-180:0.4cm) [thick, color=\clr];
   \draw[thick, color=\clr] (0.8,-0.2)--(0.8,0.6); 
       \node at (-0.4,0.4) {$\scriptstyle\blacklozenge$};
         \node at (0.4,-0.6) {$\scriptstyle\blacklozenge$};
\end{tikzpicture}},
\end{align}
\begin{align}\label{MBcox}
\hackcenter{}
\hackcenter{
\begin{tikzpicture}[scale=0.8]
  \draw[thick, color=\clr] (0.4,0)--(0.4,0.1) .. controls ++(0,0.35) and ++(0,-0.35) .. (-0.4,0.7)--(-0.4,0.8) 
 .. controls ++(0,0.35) and ++(0,-0.35) .. (0.4,1.4)--(0.4,1.5);
  \draw[thick, color=\clr] (-0.4,0)--(-0.4,0.1) .. controls ++(0,0.35) and ++(0,-0.35) .. (0.4,0.7)--(0.4,0.8)
  .. controls ++(0,0.35) and ++(0,-0.35) .. (-0.4,1.4)--(-0.4,1.5);
\end{tikzpicture}}
\;
=
\;
\hackcenter{
\begin{tikzpicture}[scale=0.8]
  \draw[thick, color=\clr] (0.4,0)--(0.4,1.5);
  \draw[thick, color=\clr] (-0.4,0)--(-0.4,1.5);
\end{tikzpicture}},
\qquad
\qquad
\hackcenter{
\begin{tikzpicture}[scale=0.8]
  \draw[thick, color=\clr] (0.2,0)--(0.2,0.1) .. controls ++(0,0.35) and ++(0,-0.35) .. (-0.4,0.75)
  .. controls ++(0,0.35) and ++(0,-0.35) .. (0.2,1.4)--(0.2,1.5);
  \draw[thick, color=\clr] (-0.4,0)--(-0.4,0.1) .. controls ++(0,0.35) and ++(0,-0.35) .. (0.8,1.4)--(0.8,1.5);
  \draw[thick, color=\clr] (0.8,0)--(0.8,0.1) .. controls ++(0,0.35) and ++(0,-0.35) .. (-0.4,1.4)--(-0.4,1.5);
\end{tikzpicture}}
\;
=
\;
\hackcenter{
\begin{tikzpicture}[scale=0.8]
  \draw[thick, color=\clr] (0.2,0)--(0.2,0.1) .. controls ++(0,0.35) and ++(0,-0.35) .. (0.8,0.75)
  .. controls ++(0,0.35) and ++(0,-0.35) .. (0.2,1.4)--(0.2,1.5);
  \draw[thick, color=\clr] (-0.4,0)--(-0.4,0.1) .. controls ++(0,0.35) and ++(0,-0.35) .. (0.8,1.4)--(0.8,1.5);
  \draw[thick, color=\clr] (0.8,0)--(0.8,0.1) .. controls ++(0,0.35) and ++(0,-0.35) .. (-0.4,1.4)--(-0.4,1.5);
\end{tikzpicture}},
\end{align}
\begin{align}\label{MBpitch}
\hackcenter{}
\hackcenter{
\begin{tikzpicture}[scale=0.8]
\draw[thick, color=\clr] (1.2,-1)--(1.2,0);
  \draw[thick, color=\clr] (0.4,0)--(0.4,-0.1) .. controls ++(0,-0.35) and ++(0,0.35) .. (-0.4,-0.9)--(-0.4,-1);
  \draw[thick, color=\clr] (-0.4,0.4)--(-0.4,-0.1) .. controls ++(0,-0.35) and ++(0,0.35) .. (0.4,-0.9)--(0.4,-1);
   \draw (1.2,0) arc (0:180:0.4cm) [thick, color=\clr];
     \node at (0.8,0.4) {$\scriptstyle\blacklozenge$};
\end{tikzpicture}}
\;
=
\;
\hackcenter{
\begin{tikzpicture}[scale=0.8]
\draw[thick, color=\clr] (-1.2,-1)--(-1.2,0);
  \draw[thick, color=\clr] (0.4,0.4)--(0.4,-0.1) .. controls ++(0,-0.35) and ++(0,0.35) .. (-0.4,-0.9)--(-0.4,-1);
  \draw[thick, color=\clr] (-0.4,0)--(-0.4,-0.1) .. controls ++(0,-0.35) and ++(0,0.35) .. (0.4,-0.9)--(0.4,-1);
   \draw (-0.4,0) arc (0:180:0.4cm) [thick, color=\clr];
     \node at (-0.8,0.4) {$\scriptstyle\blacklozenge$};
\end{tikzpicture}},
\qquad
\qquad
\hackcenter{
\begin{tikzpicture}[scale=0.8]
\draw[thick, color=\clr] (1.2,1)--(1.2,0);
  \draw[thick, color=\clr] (0.4,0)--(0.4,0.1) .. controls ++(0,0.35) and ++(0,-0.35) .. (-0.4,0.9)--(-0.4,1);
  \draw[thick, color=\clr] (-0.4,-0.4)--(-0.4,0.1) .. controls ++(0,0.35) and ++(0,-0.35) .. (0.4,0.9)--(0.4,1);
   \draw (1.2,0) arc (0:-180:0.4cm) [thick, color=\clr];
     \node at (0.8,-0.4) {$\scriptstyle\blacklozenge$};
\end{tikzpicture}}
\;
=
\;
\hackcenter{
\begin{tikzpicture}[scale=0.8]
\draw[thick, color=\clr] (-1.2,1)--(-1.2,0);
  \draw[thick, color=\clr] (0.4,-0.4)--(0.4,0.1) .. controls ++(0,0.35) and ++(0,-0.35) .. (-0.4,0.9)--(-0.4,1);
  \draw[thick, color=\clr] (-0.4,0)--(-0.4,0.1) .. controls ++(0,0.35) and ++(0,-0.35) .. (0.4,0.9)--(0.4,1);
   \draw (-0.4,0) arc (0:-180:0.4cm) [thick, color=\clr];
     \node at (-0.8,-0.4) {$\scriptstyle\blacklozenge$};
\end{tikzpicture}},
\end{align}
\begin{align}\label{MBcupcapcross}
\hackcenter{
\begin{tikzpicture}[scale=0.8]
  \draw[thick, color=\clr] (0.4,0)--(0.4,0.1) .. controls ++(0,0.35) and ++(0,-0.35) .. (-0.4,0.9)--(-0.4,1);
  \draw[thick, color=\clr] (-0.4,0)--(-0.4,0.1) .. controls ++(0,0.35) and ++(0,-0.35) .. (0.4,0.9)--(0.4,1);
   \draw (0.4,0) arc (0:-180:0.4cm) [thick, color=\clr];
     \node at (0,-0.4) {$\scriptstyle\blacklozenge$};
\end{tikzpicture}}
\;
=
\;
-
\hackcenter{
\begin{tikzpicture}[scale=0.8]
   \draw (0.4,1) arc (0:-180:0.4cm) [thick, color=\clr];
     \node at (0,0.6) {$\scriptstyle\blacklozenge$};
\end{tikzpicture}},
\qquad
\qquad
\hackcenter{
\begin{tikzpicture}[scale=0.8]
  \draw[thick, color=\clr] (0.4,0)--(0.4,-0.1) .. controls ++(0,-0.35) and ++(0,0.35) .. (-0.4,-0.9)--(-0.4,-1);
  \draw[thick, color=\clr] (-0.4,0)--(-0.4,-0.1) .. controls ++(0,-0.35) and ++(0,0.35) .. (0.4,-0.9)--(0.4,-1);
   \draw (0.4,0) arc (0:180:0.4cm) [thick, color=\clr];
     \node at (0,0.4) {$\scriptstyle\blacklozenge$};
\end{tikzpicture}}
\;
=
\;
\hackcenter{
\begin{tikzpicture}[scale=0.8]
   \draw (0.4,-1) arc (0:180:0.4cm) [thick, color=\clr];
     \node at (0,-0.6) {$\scriptstyle\blacklozenge$};
\end{tikzpicture}},
\end{align}
\begin{align}\label{MBbubble}
\hackcenter{
\begin{tikzpicture}[scale=0.8]
   \draw (0.4,-1) arc (0:180:0.4cm) [thick, color=\clr];
    \draw (0.4,-1) arc (0:-180:0.4cm) [thick, color=\clr];
      \node at (0,-0.6) {$\scriptstyle\blacklozenge$};
        \node at (0,-1.4) {$\scriptstyle\blacklozenge$};
\end{tikzpicture}}\;
=\;0.
\end{align}
\end{definition}

Let \(\fp(n)\textup{-Mod}_{V}\) denote the monoidal supercategory of \(\fp(n)\)-modules generated by the natural module \(V_n\). That is, \(\fp(n)\textup{-mod}_{V}\) is the full subcategory of \(\fp(n)\textup{-mod}\) consisting of objects of the form \(\left\{ V_{n}^{\otimes k} \mid k \in \Z_{\geq 0}\right\}\).

\begin{theorem}\cite[Theorem 5.2.1]{KT}\label{T:KTFunctor}
There is a well-defined functor of monoidal supercategories
\begin{align*}
F: \mathcal{B} \to \fp(n)\textup{-Mod}_{V}
\end{align*}
given by \(F(\bullet) = V_{n}\) and on morphisms by
\begin{align*}
\hackcenter{}
\hackcenter{
\begin{tikzpicture}[scale=0.8]
  \draw[thick, color=\clr] (0.4,0)--(0.4,0.1) .. controls ++(0,0.35) and ++(0,-0.35) .. (-0.4,0.9)--(-0.4,1);
  \draw[thick, color=\clr] (-0.4,0)--(-0.4,0.1) .. controls ++(0,0.35) and ++(0,-0.35) .. (0.4,0.9)--(0.4,1);
\end{tikzpicture}}
\mapsto \tau ,
\qquad
\qquad
\hackcenter{
\begin{tikzpicture}[scale=0.8]
 \draw (0,0) arc (0:180:0.4cm) [thick, color=\clr];
   \node at (-0.4,0.4) {$\scriptstyle\blacklozenge$};
\end{tikzpicture}}
\mapsto \cap ,
\qquad
\qquad
\hackcenter{
\begin{tikzpicture}[scale=0.8]
 \draw (0,0) arc (0:-180:0.4cm) [thick, color=\clr];
   \node at (-0.4,-0.4) {$\scriptstyle\blacklozenge$};
\end{tikzpicture}}
\mapsto \cup .
\end{align*}
\end{theorem}

\begin{theorem}\label{KTthm1}
If $\k$ is a field of characteristic zero, then the functor $F$ is full.  That is, for all $a,b \in \Z_{\geq 0}$  the induced map of superspaces,
\begin{align*}
F:\Hom_{\mathcal{B}}([a],[b]) \xrightarrow{\sim} \Hom_{\fp(n)}(V_n^{\otimes a}, V_n^{\otimes b}),
\end{align*} is surjective. 
\end{theorem}

\begin{proof} When $\k=\C$ the statement follows from \cite[Section 4.9]{DLZ} (see the proof of \cite[Theorem 5.2.1]{CEIII} for details).  The basis theorem for $\mathcal{B}$ given in \cite[Theorem 2.3.1]{KT} and straightforward base change arguments show the functor is full for an arbitrary characteristic zero field.
\end{proof}

Let \(\pWeb_1\) be the full monoidal sub-supercategory of \(\pWeb\) whose objects are tuples consisting of only ones, including the empty tuple.

\begin{theorem}\label{Bwebfun}
There is an isomorphism of monoidal supercategories
\begin{align*}
F':\mathcal{B} \to \pWeb_1
\end{align*}
given on objects by \(\bullet \mapsto 1\) and on morphisms by
\begin{align*}
\hackcenter{}
\hackcenter{
\begin{tikzpicture}[scale=0.8]
  \draw[thick, color=\clr] (0.4,0)--(0.4,0.1) .. controls ++(0,0.35) and ++(0,-0.35) .. (-0.4,0.9)--(-0.4,1);
  \draw[thick, color=\clr] (-0.4,0)--(-0.4,0.1) .. controls ++(0,0.35) and ++(0,-0.35) .. (0.4,0.9)--(0.4,1);
\end{tikzpicture}}
\mapsto
\hackcenter{
\begin{tikzpicture}[scale=0.8]
  \draw[thick, color=\clr] (0.4,0)--(0.4,0.2) .. controls ++(0,0.35) and ++(0,-0.35) .. (-0.4,0.9)--(-0.4,1);
  \draw[thick, color=\clr] (-0.4,0)--(-0.4,0.2) .. controls ++(0,0.35) and ++(0,-0.35) .. (0.4,0.9)--(0.4,1);
      \node[below] at (-0.4,0) {$\scriptstyle 1$};
      \node[below] at (0.4,0.05) {$\scriptstyle 1$};
         \node[above] at (-0.4,1) {$\scriptstyle 1$};
      \node[above] at (0.4,1) {$\scriptstyle 1$};
      \node[above,white] at (-0.5,1) {\color{white} $\scriptstyle a$};
\end{tikzpicture}},
\qquad
\qquad
\hackcenter{
\begin{tikzpicture}[scale=0.8]
 \draw (0,0) arc (0:180:0.4cm) [thick, color=\clr];
   \node at (-0.4,0.4) {$\scriptstyle\blacklozenge$};
\end{tikzpicture}}
\mapsto
\hackcenter{
\begin{tikzpicture}[scale=0.8]
   \draw (0.4,-1) arc (0:180:0.4cm) [thick, color=\clr];
     \node at (0,-0.6) {$\scriptstyle\blacklozenge$};
      \node[below] at (-0.4,-1) {$\scriptstyle 1$};
       \node[below] at (0.4,-1) {$\scriptstyle 1$};
       \node[above] at (0, -0.4) {\color{white} 1};
\end{tikzpicture}}
\qquad
\qquad
\hackcenter{
\begin{tikzpicture}[scale=0.8]
 \draw (0,0) arc (0:-180:0.4cm) [thick, color=\clr];
   \node at (-0.4,-0.4) {$\scriptstyle\blacklozenge$};
\end{tikzpicture}}
\mapsto
\hackcenter{
\begin{tikzpicture}[scale=0.8]
   \draw (0.4,1) arc (0:-180:0.4cm) [thick, color=\clr];
     \node at (0,0.6) {$\scriptstyle\blacklozenge$};
           \node[above] at (-0.4,1) {$\scriptstyle 1$};
       \node[above] at (0.4,1) {$\scriptstyle 1$};
        \node[below] at (0, 0.4) {\color{white} 1};
\end{tikzpicture}}.
\end{align*}
\end{theorem}
\begin{proof}
We first check that relations (\ref{MBstraight})--(\ref{MBbubble}) are satisfied in \(\pWeb\). Relation (\ref{MBstraight}) holds by (\ref{StraightRel}). The relations (\ref{MBcox}, \ref{MBpitch}) hold by  \cref{BraidThm}. The relations (\ref{MBcupcapcross}, \ref{MBbubble}) hold by \cref{TwistMarkBubble}.

The functor \(F'\) restricts to an isomorphism on Hom spaces, as \(F'\) sends the basis morphisms described in \cite[Theorem 2.3.1]{KT} to the basis morphisms of \cref{BasisThm}.
\end{proof}

\subsection{The Functor  \texorpdfstring{$G_{\uparrow\downarrow}: \pWeb_{\uparrow\downarrow} \to \pmodSS$}{G:p-Webupdown -> p(n)-mod}}\label{SS:GFunctorOrientedWebs}

\begin{theorem}\label{OrGthm}
There is a well-defined functor of monoidal supercategories
\begin{align*}
G_{\uparrow \downarrow}: \pWeb_{\uparrow \downarrow} \to \pmodSS 
\end{align*}
given on objects by \(G_{\uparrow \downarrow}(\uparrow_a) = S^a(V_n)\) and \(G_{\uparrow \downarrow}(\downarrow_a)=S^a(V_n)^*\). The functor is given on morphisms by
$$
\hackcenter{
{}
}
\hackcenter{
\begin{tikzpicture}[scale=0.8]
\draw[thick, color=\clr,->] (0,-0.1)--(0,0.2);
  \draw[thick, color=\clr,->] (0,0)--(0,0.2) .. controls ++(0,0.35) and ++(0,-0.35) .. (-0.4,0.9)--(-0.4,1);
  \draw[thick, color=\clr,->] (0,0)--(0,0.2) .. controls ++(0,0.35) and ++(0,-0.35) .. (0.4,0.9)--(0.4,1);
      \node[above] at (-0.4,1) {$\scriptstyle a$};
      \node[above] at (0.4,1) {$\scriptstyle b$};
      \node[below] at (0,-0.1) {$\scriptstyle a+b$};
\end{tikzpicture}}
\mapsto 
\textup{spl}_{a+b}^{a,b},
\qquad
\hackcenter{
\begin{tikzpicture}[scale=0.8]
 \draw[thick, color=\clr,->] (-0.4,-0.1)--(-0.4,0.2);
  \draw[thick, color=\clr,->] (0.4,-0.1)--(0.4,0.2);
  \draw[thick, color=\clr,->] (-0.4,0)--(-0.4,0.1) .. controls ++(0,0.35) and ++(0,-0.35) .. (0,0.8)--(0,1);
\draw[thick, color=\clr,->] (0.4,0)--(0.4,0.1) .. controls ++(0,0.35) and ++(0,-0.35) .. (0,0.8)--(0,1);
      \node[below] at (-0.4,-0.1) {$\scriptstyle a$};
      \node[below] at (0.4,-0.1) {$\scriptstyle b$};
      \node[above] at (0,1) {$\scriptstyle a+b$};
\end{tikzpicture}}
\mapsto
\textup{mer}_{a,b}^{a+b},
\qquad
\hackcenter{
\begin{tikzpicture}[scale=0.8]
 \draw(0,0) arc (0:180:0.4cm) [thick, color=\clr,->];
    \node[below] at (0,0) {$\scriptstyle a$};
      \node[below] at (-0.8,0) {$\scriptstyle a$};
\end{tikzpicture}}
\mapsto
\textup{eval}_a,
\qquad
\hackcenter{
\begin{tikzpicture}[scale=0.8]
 \draw(0,0) arc (0:-180:0.4cm) [thick, color=\clr,->];
    \node[above] at (0,0) {$\scriptstyle a$};
      \node[above] at (-0.8,0) {$\scriptstyle a$};
\end{tikzpicture}}
\mapsto
\textup{coeval}_a,
$$
$$
\hackcenter{
\begin{tikzpicture}[scale=0.8]
 \draw[thick, color=\clr,->] (0,0)--(0,0.4);
  \draw[thick, color=\clr,<-] (0,0.6)--(0,1);
    \node[below] at (0,0) {$\scriptstyle 1$};
     \node[above] at (0,1) {$\scriptstyle 1$};
 \node[shape=coordinate](DOT) at (0,0.5) {};
  \draw[thick, color=\clr, fill=yellow]  (DOT) circle (1mm);
\end{tikzpicture}}
\mapsto
D,
\qquad
\hackcenter{
\begin{tikzpicture}[scale=0.8]
 \draw[thick, color=\clr,<-] (0,0)--(0,0.4);
  \draw[thick, color=\clr,->] (0,0.6)--(0,1);
    \node[below] at (0,0) {$\scriptstyle 1$};
     \node[above] at (0,1) {$\scriptstyle 1$};
 \node[shape=coordinate](DOT) at (0,0.5) {};
  \draw[thick, color=\clr, fill=blue]  (DOT) circle (1mm);
\end{tikzpicture}}
\mapsto
D^{-1},
\qquad
\hackcenter{
\begin{tikzpicture}[scale=0.8]
  \draw[thick, color=\clr,<-] (0.4,0)--(0.4,0.1) .. controls ++(0,0.35) and ++(0,-0.35) .. (-0.4,0.9)--(-0.4,1);
  \draw[thick, color=\clr,->] (-0.4,0)--(-0.4,0.1) .. controls ++(0,0.35) and ++(0,-0.35) .. (0.4,0.9)--(0.4,1);
      \node[above] at (-0.4,1) {$\scriptstyle b$};
      \node[above] at (0.4,1) {$\scriptstyle a$};
            \node[below] at (-0.4,0) {$\scriptstyle a$};
      \node[below] at (0.4,0) {$\scriptstyle b$};
\end{tikzpicture}}
\mapsto
\tau_{S^a(V_n),S^b(V_n)^*}.
$$
\end{theorem}
\begin{proof}
Note that upward caps, upward cups, and upward antennas are sent to \(\cup\), \(\cap\), and \(\textup{ant}\), respectively, so the defining up-arrow relations of \(\pWeb_{\uparrow \downarrow}\) are preserved by \(G_{\uparrow \downarrow}\) thanks to  \cref{Gthm}. We now check (\ref{OrStraightRel}--\ref{BubbleRel}).

Let \(B_a\) be a homogeneous \(\k\)-basis for \(S^a(V_n)\), with dual basis \(\{x^* \mid x \in B_a\}\) for \(S^a(V_n)^*\). Then for \(x,y \in B_a\) we have
\begin{align*}
\textup{eval}_a(x^* \otimes y) = \delta_{x,y},
\qquad
\textup{and}
\qquad
\textup{coeval}_a(1) = \sum_{x \in B_a} x \otimes x^*.
\end{align*}
So for all \(x \in B_a\), we have
\begin{align*}
(\textup{id} \otimes \textup{eval}_a) \circ (\textup{coeval}_a \otimes \textup{id})(x)
=
\sum_{y \in B_a}(\textup{id} \otimes \textup{eval}_a)(y \otimes y^* \otimes x)
=
\sum_{y \in B_a} \delta_{y,x}y
=
x,
\end{align*}
and
\begin{align*}
(\textup{eval}_a \otimes \textup{id})\circ (\textup{id} \otimes \textup{coeval}_a)(x^*)
=
\sum_{y \in B_a} (\textup{eval}_a \otimes \textup{id})(x^* \otimes y \otimes y^*)
=
\sum_{y \in B_a} \delta_{x,y}y^*
=
x^*,
\end{align*}
proving (\ref{OrStraightRel}).

By \cref{CrossTwist}, we have
\begin{align*}
\hackcenter{}
G_{\uparrow \downarrow}:
\hackcenter{
\begin{tikzpicture}[scale=0.8]
  \draw[thick, color=\clr,->] (0.4,0)--(0.4,0.1) .. controls ++(0,0.35) and ++(0,-0.35) .. (-0.4,0.8)--(-0.4,1);
  \draw[thick, color=\clr,->] (-0.4,0)--(-0.4,0.1) .. controls ++(0,0.35) and ++(0,-0.35) .. (0.4,0.8)--(0.4,1);
      \node[below] at (-0.4,0) {$\scriptstyle a$};
      \node[below] at (0.4,0.04) {$\scriptstyle b$};
      \node[above] at (-0.4,1) {$\scriptstyle a$};
      \node[above] at (0.4,1) {$\scriptstyle b$};
\end{tikzpicture}}
\mapsto
\tau_{S^a(V_n),S^b(V_n)}.
\end{align*}
From this it is immediate that for all \(x \in B_a\), \(y \in B_b\)
\iffalse
, we have
\begin{align*}
(\textup{eval}_a \otimes \textup{id} \otimes \textup{id}) &\circ (\textup{id} \otimes \tau_{S^b(V_n), S^a(V_n)} \otimes \textup{id})\circ (\textup{id} \otimes \textup{id} \otimes \textup{coeval}_a)(x^* \otimes y)\\
&=\sum_{z \in B_a}
(\textup{eval}_a \otimes \textup{id} \otimes \textup{id}) \circ (\textup{id} \otimes \tau_{S^b(V_n), S^a(V_n)} \otimes \textup{id})(x^* \otimes y \otimes z \otimes z^*)\\
&=\sum_{z \in B_a}
(\textup{eval}_a \otimes \textup{id} \otimes \textup{id}) (-1)^{|y|\cdot |z|}(x^* \otimes z \otimes y \otimes z^*)\\
&=\sum_{z \in B_a}
(-1)^{|y|\cdot |z|} \delta_{x,z} y \otimes z^* 
= (-1)^{|y|\cdot |x|} y \otimes x^*,
\end{align*}
so
\fi
the image of the leftward crossing under \(G_{\uparrow \downarrow}\) is \(\tau_{S^a(V_n)^*, S^b(V_n)}\), and thus we have that relation (\ref{LRCrossRel}) is preserved.

Finally, to check relation (\ref{BubbleRel}), we note that
\begin{align*}
\textup{eval}_a \circ \tau_{S^a(V_n)^*, S^a(V_n)} \circ \textup{coeval}_{a}(1)
&=
\sum_{x \in B_a} \textup{eval}_a \circ \tau_{S^a(V_n)^*, S^a(V_n)}(x \otimes x^* )\\
&=
\sum_{x \in B_a} (-1)^{|x|} \textup{eval}_a (x^* \otimes x)
=
\sum_{x \in B_a} (-1)^{|x|}
=
0,
\end{align*}
completing the proof.
\end{proof}

%%%%%%%%%%%%%%%%%%%%%%%%%%%%%%%%%%%%%%

\subsection{Explosion and Contraction} In this section $\k$ can be an integral domain.  For short, given $k \geq 1 $ we write \(y_k \in \Hom_{\pWeb}(k,1^k)\) and \(z_k \in \Hom_{\pWeb}(1^k,k)\) for the morphisms defined by \cref{E:multisplitmerge}.

\begin{lemma}\label{Lemzy}  For all $k \geq 1$ we have \(z_k \circ y_k = k! \cdot \textup{id}_k\).
\end{lemma}
\begin{proof}
Follows from repeated application of (\ref{DiagSwitchRel}).
\end{proof}

 Let \(k \in \Z_{\geq 0}\) and \(\bba = (a_1, \ldots, a_k) \in \Z_{\geq 0}^k\). We identify \(\bba\) with the object $(a_1, \dotsc,  a_{k})$ of $\pWeb$. We will also use the following associated notation:
\begin{align*}
 |\bba|:=a_1 + \cdots + a_{k},
 \qquad
 \bba!:= a_1! \cdots a_{k}!
 \qquad
 y_\ba := y_{a_1} \otimes \cdots \otimes y_{a_{k}},
 \qquad
 z_\ba := z_{a_1} \otimes \cdots \otimes z_{a_{k}}.
 \end{align*}
For any objects \(\bba, \bbb\) in $\pWeb$ we have linear maps:
 \begin{align*}
\exp_{\bba,\bbb}:\Hom_{\pWeb}(\bba, \bbb) \to  \Hom_{\pWeb}(1^{|\bba|}, 1^{|\bbb|}),
 \qquad
 f \mapsto y_{\bbb} \circ f \circ z_{\bba},
 \end{align*}
 and
  \begin{align*}
\con_{\bba, \bbb}:\Hom_{\pWeb}(1^{|\bba|}, 1^{|\bbb|}) \to  \Hom_{\pWeb}(\bba, \bbb),
 \qquad
 g \mapsto z_{\bbb} \circ g \circ y_{\bba}.
 \end{align*}
We refer to these maps as {\em explosion} and {\em contraction}, respectively.  See the proof of \cite[Theorem 1.10]{RT} for a picture showing them in use.
 
 \begin{lemma}\label{conexp}
 For every \(f \in \Hom_{\pWeb}(\bba, \bbb)\) we have \(\left(\con_{\bba,\bbb} \circ \exp_{\bba, \bbb}  \right) (f) = \bba!\bbb! \, f\).
 \end{lemma}
 \begin{proof}
 Follows from  \cref{Lemzy}.
 \end{proof}
 It should be noted that the explosion and contraction morphisms only involves $\aWeb$ morphisms and calculations.  They first appear in \cite{RT} and are now a standard tool in this area.

\subsection{Putting Things Together, \texorpdfstring{$\pWeb$}{pWeb} Edition}\label{SS:PuttingTogetherPWeb} 

Let \(r,s \in \Z_{\geq 0}\), \(\bba \in \Z_{\geq 0}^r\), \(\bbb \in \Z_{\geq 0}^s\). By \cref{Gthm}, we may define linear maps: 
\begin{gather*}
\widetilde{\exp}_{\bba,\bbb}:\Hom_{\fp(n)}\left( S^\bba(V_{n}), S^\bbb(V_{n})\right) \to  \Hom_{\fp(n)}\left( V_{n}^{|\bba|}, V_{n}^{|\bbb|}\right),\\
 f \mapsto G\left( y_{\bbb}\right) \circ f \circ G\left(z_{\bba}\right),
 \end{gather*}
 and
  
\begin{gather*}
\widetilde{\con}_{\bba, \bbb}:\Hom_{\fp(n)}\left(V_{n}^{|\bba|}, V_{n}^{|\bbb|}\right) \to  \Hom_{\fp(n)}\left(S^\bba(V_{n}), S^\bbb(V_{n})\right),\\
 g \mapsto G\left(z_{\bbb}\right) \circ g \circ G\left(y_{\bba}\right).
 \end{gather*}

\begin{lemma}\label{comdiag} For all objects $\bba$ and $\bbb$ in $\pWeb$ the following diagram commutes:
\begin{center}
\begin{tikzcd}
\Hom_{\pWeb}(\bba,\bbb)
\\
\\
\Hom_{\pWeb}(\bba,\bbb)
\ar[uu, "\bba!\bbb! \cdot \textup{id}"]
\ar[rr, "\exp_{\bba,\bbb}"']
\ar[dd, "G"']
&&
\Hom_{\pWeb_1}(1^{|\bba|},1^{|\bbb|})
\ar[uull, "\textup{\con}_{\bba,\bbb}"']
\ar[dd, "G"']
\\
&&&
\Hom_{\mathcal{B}}([|\bba|],[|\bbb|])
\ar[ul, "F'"']
\ar[dl, "F"]
\\
\Hom_{\fp(n)}(S^{\bba}(V_{n}), S^{\bbb}(V_{n})) 
\ar[rr, "\widetilde{\exp}_{\bba,\bbb}"]
\ar[dd, "\bba!\bbb! \cdot \textup{id}"']
&&
\Hom_{\fp(n)}(V_{n}^{|\bba|}, V_{n}^{|\bbb|}) 
\ar[ddll, "\widetilde{\textup{\con}}_{\bba,\bbb}"]
&\\
\\
\Hom_{\fp(n)}(S^{\bba}(V_{n}), S^{\bbb}(V_{n})) 
\end{tikzcd}
\end{center}
\end{lemma}
\begin{proof}
The top and bottom triangles and middle rectangle commute by  \cref{conexp} and  \cref{Gthm}. The triangle on the right can be seen to commute by checking the definitions of \(F, F', G\) on generating morphisms, together with \cref{tswaprel}.
\end{proof}

\begin{theorem}\label{HomIsomPlus} If \(\bba,\bbb\) are such that \(|\bba| + |\bbb| \leq 2n\), then the map
\begin{align*}
G:\Hom_{\pWeb}(\bba,\bbb)
\to
\Hom_{\fp(n)}\left(S^{\bba}(V_{n}), S^{\bbb}(V_{n})\right). 
\end{align*}
is injective.  If $\k$ is a field of characteristic zero, then the functor $G$ is full.
\end{theorem}
\begin{proof}
The injectivity statement follows from  \cref{actind} and  \cref{BasisThm}.

Now assume $\k$ has characteristic zero and consider the diagram in  \cref{comdiag}. Since \(F'\) is an isomorphism by  \cref{Bwebfun} and \(F\) is surjective by  \cref{KTthm1},  the map \(G\) on the right is surjective.  To see surjectivity of \(G\) along the left side, let \(\varphi \in \Hom_{\fp(n)}(S^{\bba}(V_{n}), S^{\bbb}(V_{n}))\). Then by surjectivity of the \(G\) on the right, there exists \(\theta \in \Hom_{\pWeb_1^+}(1^{|\bba|}, 1^{|\bbb|})\) such that
\begin{align*}
G\left(\theta \right) = \widetilde{\exp}_{\bba, \bbb}(\varphi) = G\left(y_{\bbb} \right) \circ \varphi \circ G\left(z_{\bba} \right).
\end{align*}
Then we compute
\begin{align*}
G\left(\frac{\textup{con}_{\bba,\bbb}(\theta)}{\bba!\bbb!}\right)
&=
\frac{1}{\bba!\bbb!} G(z_{\bbb} \circ \theta \circ y_{\bba})
=
\frac{1}{\bba!\bbb!} Gz_{\bbb} \circ G\theta \circ Gy_{\bba}
=
\frac{1}{\bba!\bbb!} Gz_{\bbb} \circ    Gy_{\bbb} \circ \varphi \circ Gz_{\bba}.   \circ Gy_{\bba}\\
&=\frac{1}{\bba!\bbb!} G(z_\bbb \circ y_\bbb) \circ \varphi \circ G(z_{\bba} \circ y_{\bba})
=
\frac{1}{\bba!\bbb!} G\left( \bbb! \cdot \textup{id}_{\bbb}\right) \circ \varphi \circ G\left( \bba!\cdot \textup{id}_{\bba}\right)
=
\varphi,
\end{align*}
as desired.  That is, \(G\) along the left side of the diagram is surjective, completing the proof.
\end{proof}

Via the functor \(G\), \cref{HomIsomPlus} shows that the basis for \(\Hom_{\pWeb}(\bba,\bbb)\) given in \cref{BasisThm} could be considered a `stable basis', or a `basis at infinity' for \(\Hom_{\fp(n)}\left(S^{\bba}(V_{n}), S^{\bbb}(V_{n})\right)\) in characteristic zero, since \(G\) defines an isomorphism whenever \(n \gg 0\).

\begin{remark}\label{FailFull}
The functor \(G\) need not be full over a field \(\k\) of positive characteristic. For example, if \(\textup{char}(\k) = 3\) and \(v\) is a nonzero even vector in \(V_n\), then there is a nonzero \(\fp(n)\)-module homomorphism \(\k \to S^3(V_n)\) given by \(1 \mapsto v^3\), but the image of \(G\) in \(\Hom_{\fp(n)}(\k, S^3(V_n))\) is zero.
\end{remark}

\subsection{Putting Things Together, \texorpdfstring{$\pWeb_{\uparrow\downarrow}$}{p-Web-updown} Edition}

Let \(r,s \in \Z_{\geq 0}\) and \(\bba \in \Z^r\), \(\bbb \in \Z^s\). Given a nonnegative integer $a$ it will be convenient to adopt the notation \(S^{a}(V_{n}):= S^a(V_{n})\),  \(S^{-a}(V_{n}):= S^a(V_{n})^*\), and, more generally,
\begin{align*}
S^{\bba}(V_{n}):= S^{a_1}(V_{n}) \otimes \cdots \otimes S^{a_r}(V_{n}).
\end{align*}
Let \(\bbc, \bbd, \varphi_1, \varphi_2\) be as in  \cref{MixToUp}. By  \cref{OrGthm}, we have an invertible linear map:
\begin{gather*}
\tilde{\Phi}:\Hom_{\fp(n)}\left( S^{\bba}(V_{n}), S^{\bbb}(V_{n})\right) \xrightarrow{\sim} 
\Hom_{\fp(n)}\left(S^\bbc(V_{n}), S^\bbd(V_{n})\right)\\
f \mapsto G_{\uparrow \downarrow}\left(\varphi_2  \right)\circ f \circ G_{\uparrow \downarrow}\left(\varphi_1 \right).
\end{gather*}

\begin{lemma}\label{Orcomdiag}  For any $\bba \in \Z^{r}$ and $\bbb \in \Z^{s}$, let $\bbc$ and $\bbd$ be as above.  Then the following diagram commutes:
\begin{center}
\begin{tikzcd}
\Hom_{\pWeb_{\uparrow \downarrow}}(|_\bba,|_\bbb)
\ar[rr, "\Phi"']
\ar[dd, "G_{\uparrow \downarrow}"']
&&
\Hom_{\pWeb_\uparrow}(\uparrow_\bbc, \uparrow_\bbd)
\ar[dd, "G_{\uparrow \downarrow}"']
\\
&&&
\Hom_{\pWeb}(\bbc, \bbd)
\ar[ul, "\iota_\uparrow"']
\ar[dl, "G"]
\\
\Hom_{\fp(n)}(S^{\bba}(V_{n}), S^{\bbb}(V_{n})) 
\ar[rr, "\tilde{\Phi}"]
&&
\Hom_{\fp(n)}(S^\bbc(V_{n}),S^\bbd(V_{n})) 
&
\end{tikzcd}
\end{center}
\end{lemma}
\begin{proof}
The left rectangle commutes by  \cref{OrGthm}. The triangle on the right can be seen to commute by checking the definitions of \(\iota_\uparrow, G, G_{\uparrow\downarrow}\) on generating morphisms.
\end{proof}

\begin{theorem}\label{OrHomIsomPlus}
If \(\bba \in \Z^{r}\) and $\bbb \in \Z^{s}$ are such that \(|\bba| + |\bbb| \leq 2n\), then the map
\begin{align*}
G_{\uparrow \downarrow}:\Hom_{\pWeb_{\uparrow \downarrow}}\left( |_\bba,|_\bbb \right)
\to
\Hom_{\fp(n)}\left( S^{\bba}(V_{n}), S^{\bbb}(V_{n})\right). 
\end{align*}
is injective.  If $\k$ is a field of characteristic zero, then the functor $G_{\uparrow \downarrow}$ is full.
\end{theorem}
\begin{proof}
Consider the diagram in  \cref{Orcomdiag}, with \(\bbc, \bbd\) as in \cref{MixToUp}. We have \(|\bbc| + |\bbd| \leq 2n\), so the map \(G\) is injective by \cref{HomIsomPlus}. Therefore the map \(\iota_\uparrow\) is injective, and is surjective by  \cref{iotaFull}, so the map \(G_{\uparrow \downarrow}\) on the right must be injective. As \(\Phi\) and \(\tilde{\Phi}\) are isomorphisms, we have that the map \(G_{\uparrow \downarrow}\) is injective as well.

 Moreover, if $\k$ is a field of characteristic zero, then by \cref{HomIsomPlus} the map \(G\) is surjective which, in turn, forces all maps in the diagram to be surjective and so $G_{\uparrow \downarrow}$ is full, as claimed.
\end{proof}

\begin{theorem}\label{T:pWebtopWebup}
The functor
\begin{align*}
\iota_\uparrow: \pWeb \to \pWeb_{\uparrow}
\end{align*}
is an isomorphism of categories.
\end{theorem}
\begin{proof}
The functor is bijective on objects, and full by  \cref{iotaFull}. For any fixed pair of objects $\bbc$ and $\bbd$ in $\pWeb$ we may choose $n$ such that \(2n\geq |\bbc| +|\bbd|\) and consider the diagram in \cref{Orcomdiag}. Then the map \(G\) is injective by \cref{HomIsomPlus}, forcing 
\begin{align*}
\Hom_{\pWeb}\left(\bbc,\bbd \right) \xrightarrow{\sim} \Hom_{\pWeb_\uparrow}\left( \uparrow_\bbc, \uparrow_\bbd \right),
\end{align*}
as required.
\end{proof}

\begin{conjecture}  It is interesting to consider the precise conditions which imply the injectivity and fullness statements of \cref{HomIsomPlus,OrHomIsomPlus}.  We conjecture that the fullness statement holds whenever \(\k\) is a field with characteristic greater than \(|\bba| + |\bbb|\).
\end{conjecture}

\subsection{A Particular Spanning Set}
In the follow-up paper \cite{DKM} it will be useful to work with a particular spanning set for $\Hom_{\pWeb_{\uparrow \downarrow}}\left(\emptyset, \uparrow^{\bba}\downarrow^{\bbb} \right)$. As the relevant diagrammatics are already defined herein, we request the reader's indulgence in including the necessary result here.  In light of the results of \cref{SS:Brauer}, readers familiar with explosion/contraction arguments will note that that the spanning set is a contraction of the Brauer diagram basis for $\Hom_{\mathcal{B}}\left(\emptyset, |\bba| + |\bbb| \right)$ given in \cite{KT}.

\begin{lemma}\label{0abdiagLem}
Assume \(\k\) is a field of characteristic zero, and $\bba, \bbb  \in \Z_{\geq 0}^{m}$. Then $\Hom_{\pWeb_{\uparrow \downarrow}}\left(\emptyset, \uparrow^{\bba}\downarrow^{\bbb} \right)$ is spanned by diagrams of the form
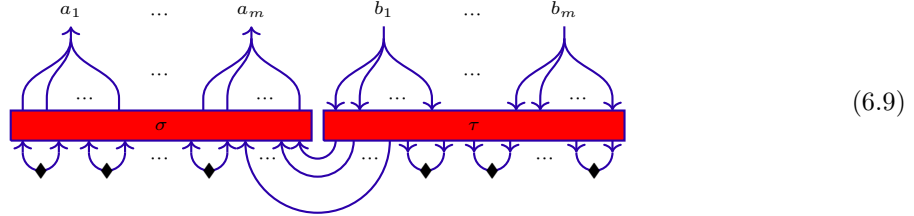
\begin{align}\label{0abdiag}
\hackcenter{}
\hackcenter{
\begin{tikzpicture}[scale=0.8]
   \draw[thick, color=\clr, fill=red] (-1,1.7)--(4,1.7)--(4,2.2)--(-1,2.2)--(-1,1.7);
    \node at (1.5,1.95) {$\scriptstyle \sigma$};
    %%%%%%%%%
    %%%%%%%%%
 \draw[thick, color=\clr,<-] (0,3.4+0.2)--(0,3.4+0);
 \draw[thick, color=\clr] (0,3.4+0) .. controls ++(0,-0.5) and ++(0,0.5) .. (-0.8,3.4+-1)--(-0.8,3.4+-1.2);
  \draw[thick, color=\clr] (0,3.4+0) .. controls ++(0,-0.5) and ++(0,0.5) .. (-0.4,3.4+-1)--(-0.4,3.4+-1.2);
       \draw[thick, color=\clr] (0,3.4+0) .. controls ++(0,-0.5) and ++(0,0.5) .. (0.8,3.4+-1)--(0.8,3.4+-1.2);
                  \node[left] at (0.6,3.4+-1) {$\scriptstyle \cdots $};
            \node[above] at (0,3.4+0.2) {$\scriptstyle a_1 $};
  %%%%%%%%%%
  \node at (1.5,2.8) {$\scriptstyle \cdots $};
   \node at (1.5,3.8) {$\scriptstyle \cdots $};
      \node at (1.5,1.4) {$\scriptstyle \cdots $};
          \node at (7.9,1.4) {$\scriptstyle \cdots $};
            \node at (3.3,1.4) {$\scriptstyle \cdots $};
             \node at (5,1.4) {$\scriptstyle \cdots $};
  %%%%%%%%%%
 \draw[thick, color=\clr,<-] (3+0,3.4+0.2)--(3+0,3.4+0);
 \draw[thick, color=\clr] (3+0,3.4+0) .. controls ++(0,-0.5) and ++(0,0.5) .. (3+-0.8,3.4+-1)--(3+-0.8,3.4+-1.2);
  \draw[thick, color=\clr] (3+0,3.4+0) .. controls ++(0,-0.5) and ++(0,0.5) .. (3+-0.4,3.4+-1)--(3+-0.4,3.4+-1.2);
       \draw[thick, color=\clr] (3+0,3.4+0) .. controls ++(0,-0.5) and ++(0,0.5) .. (3+0.8,3.4+-1)--(3+0.8,3.4+-1.2);
                  \node[left] at (3+0.6,3.4+-1) {$\scriptstyle \cdots $};
            \node[above] at (3+0,3.4+0.2) {$\scriptstyle a_m $};
  %%%%%%%%%%
  %%%%%%%%%%
  %%%%%%%%%%
     \draw[thick, color=\clr, fill=red] (5.2+-1,1.7)--(5.2+4,1.7)--(5.2+4,2.2)--(5.2+-1,2.2)--(5.2+-1,1.7);
    \node at (5.2+1.5,1.95) {$\scriptstyle \tau$};
    %%%%%%%%%
    %%%%%%%%%
 \draw[thick, color=\clr] (5.2+0,3.4+0.2)--(5.2+0,3.4+0);
 \draw[thick, color=\clr,->] (5.2+0,3.4+0) .. controls ++(0,-0.5) and ++(0,0.5) .. (5.2+-0.8,3.4+-1)--(5.2+-0.8,3.4+-1.2);
  \draw[thick, color=\clr,->] (5.2+0,3.4+0) .. controls ++(0,-0.5) and ++(0,0.5) .. (5.2+-0.4,3.4+-1)--(5.2+-0.4,3.4+-1.2);
       \draw[thick, color=\clr,->] (5.2+0,3.4+0) .. controls ++(0,-0.5) and ++(0,0.5) .. (5.2+0.8,3.4+-1)--(5.2+0.8,3.4+-1.2);
                  \node[left] at (5.2+0.6,3.4+-1) {$\scriptstyle \cdots $};
            \node[above] at (5.2+0,3.4+0.2) {$\scriptstyle b_1 $};
  %%%%%%%%%%
  \node at (5.2+1.5,2.8) {$\scriptstyle \cdots $};
   \node at (5.2+1.5,3.8) {$\scriptstyle \cdots $};
  %%%%%%%%%%
 \draw[thick, color=\clr] (5.2+3+0,3.4+0.2)--(5.2+3+0,3.4+0);
 \draw[thick, color=\clr,->] (5.2+3+0,3.4+0) .. controls ++(0,-0.5) and ++(0,0.5) .. (5.2+3+-0.8,3.4+-1)--(5.2+3+-0.8,3.4+-1.2);
  \draw[thick, color=\clr,->] (5.2+3+0,3.4+0) .. controls ++(0,-0.5) and ++(0,0.5) .. (5.2+3+-0.4,3.4+-1)--(5.2+3+-0.4,3.4+-1.2);
       \draw[thick, color=\clr,->] (5.2+3+0,3.4+0) .. controls ++(0,-0.5) and ++(0,0.5) .. (5.2+3+0.8,3.4+-1)--(5.2+3+0.8,3.4+-1.2);
                  \node[left] at (5.2+3+0.6,3.4+-1) {$\scriptstyle \cdots $};
            \node[above] at (5.2+3+0,3.4+0.2) {$\scriptstyle b_m $};
  %%%%%%%%%%
  %%%%%%%%%%
    %%%%%%%%%%
  %%%%%%%%%%
    \draw (4.4,1.7) arc (0:-180:0.3cm) [thick, color=\clr,->];
     \draw (4.7,1.7) arc (0:-180:0.6cm) [thick, color=\clr,->];
       \draw (5.3,1.7) arc (0:-180:1.2cm) [thick, color=\clr,->];
       %%%%%%%%%
        \draw (-0.2+0,0+1.5) arc (0:-180:0.3cm) [thick, color=\clr];
   \draw[thick, color=\clr,<-] (-0.2+-0.6,0.2+1.5)--(-0.2+-0.6,0+1.5);
     \draw[thick, color=\clr,<-] (-0.2+0,0.2+1.5)--(-0.2+0,0+1.5);
  \node at (-0.2+-0.3,-0.3+1.5) {$\scriptstyle\blacklozenge$};
 %%%%%%%
      \draw (0.9+0,0+1.5) arc (0:-180:0.3cm) [thick, color=\clr];
   \draw[thick, color=\clr,<-] (0.9+-0.6,0.2+1.5)--(0.9+-0.6,0+1.5);
     \draw[thick, color=\clr,<-] (0.9+0,0.2+1.5)--(0.9+0,0+1.5);
  \node at (0.9+-0.3,-0.3+1.5) {$\scriptstyle\blacklozenge$};
   %%%%%%%
      \draw (2.6+0,0+1.5) arc (0:-180:0.3cm) [thick, color=\clr];
   \draw[thick, color=\clr,<-] (2.6+-0.6,0.2+1.5)--(2.6+-0.6,0+1.5);
     \draw[thick, color=\clr,<-] (2.6+0,0.2+1.5)--(2.6+0,0+1.5);
  \node at (2.6+-0.3,-0.3+1.5) {$\scriptstyle\blacklozenge$};
  %%%%%%%%
         %%%%%%%%%
        \draw (6.4+-0.2+0,0+1.5) arc (0:-180:0.3cm) [thick, color=\clr];
   \draw[thick, color=\clr,->] (6.4+-0.2+-0.6,0.2+1.5)--(6.4+-0.2+-0.6,0+1.5);
     \draw[thick, color=\clr,->] (6.4+-0.2+0,0.2+1.5)--(6.4+-0.2+0,0+1.5);
  \node at (6.4+-0.2+-0.3,-0.3+1.5) {$\scriptstyle\blacklozenge$};
 %%%%%%%
      \draw (6.4+0.9+0,0+1.5) arc (0:-180:0.3cm) [thick, color=\clr];
   \draw[thick, color=\clr,->] (6.4+0.9+-0.6,0.2+1.5)--(6.4+0.9+-0.6,0+1.5);
     \draw[thick, color=\clr,->] (6.4+0.9+0,0.2+1.5)--(6.4+0.9+0,0+1.5);
  \node at (6.4+0.9+-0.3,-0.3+1.5) {$\scriptstyle\blacklozenge$};
   %%%%%%%
      \draw (6.4+2.6+0,0+1.5) arc (0:-180:0.3cm) [thick, color=\clr];
   \draw[thick, color=\clr,->] (6.4+2.6+-0.6,0.2+1.5)--(6.4+2.6+-0.6,0+1.5);
     \draw[thick, color=\clr,->] (6.4+2.6+0,0.2+1.5)--(6.4+2.6+0,0+1.5);
  \node at (6.4+2.6+-0.3,-0.3+1.5) {$\scriptstyle\blacklozenge$};
  \end{tikzpicture}}
\end{align}
where all strands except the those at the very top of the diagram are thin, the diagram has \(0 \leq j \leq |\bba|/2\) upward-oriented cups, \(|\bbb|/2 - |\bba|/2 + j\) downward-oriented cups, \(|\bba| - 2j\) leftward-oriented cups, \(\sigma\) consists only of upward-oriented crossings, and \(\tau\) consists only of downward-oriented crossings. 
\end{lemma}
\begin{proof}
Let \(\bbb' = (b_m, \ldots, b_1)\). 
We first consider $\Hom_{\pWeb_{\uparrow \downarrow}}\left(\uparrow^{\bbb' }, \uparrow^{\bba} \right)$. It follows from 
\cref{Bwebfun,comdiag} that this morphism space is spanned by diagrams of the form 

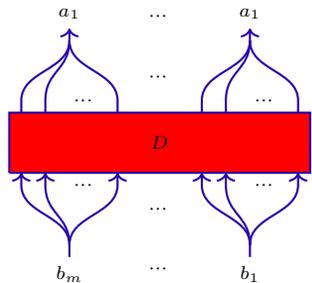
\begin{align*}
\hackcenter{}
\hackcenter{
\begin{tikzpicture}[scale=0.8]
 \draw[thick, color=\clr] (0,-0.2)--(0,0);
 \draw[thick, color=\clr,->] (0,0) .. controls ++(0,0.5) and ++(0,-0.5) .. (-0.8,1)--(-0.8,1.2);
  \draw[thick, color=\clr,->] (0,0) .. controls ++(0,0.5) and ++(0,-0.5) .. (-0.4,1)--(-0.4,1.2);
       \draw[thick, color=\clr,->] (0,0) .. controls ++(0,0.5) and ++(0,-0.5) .. (0.8,1)--(0.8,1.2);
                  \node[left] at (0.6,1) {$\scriptstyle \cdots $};
            \node[below] at (0,-0.2) {$\scriptstyle b_m $};
  %%%%%%%%%%
  \node at (1.5,-0.4) {$\scriptstyle \cdots $};
   \node at (1.5,0.6) {$\scriptstyle \cdots $};
  %%%%%%%%%%
 \draw[thick, color=\clr] (3+0,-0.2)--(3+0,0);
 \draw[thick, color=\clr,->] (3+0,0) .. controls ++(0,0.5) and ++(0,-0.5) .. (3+-0.8,1)--(3+-0.8,1.2);
  \draw[thick, color=\clr,->] (3+0,0) .. controls ++(0,0.5) and ++(0,-0.5) .. (3+-0.4,1)--(3+-0.4,1.2);
       \draw[thick, color=\clr,->] (3+0,0) .. controls ++(0,0.5) and ++(0,-0.5) .. (3+0.8,1)--(3+0.8,1.2);
                  \node[left] at (3+0.6,1) {$\scriptstyle \cdots $};
            \node[below] at (3+0,-0.2) {$\scriptstyle b_1 $};
  %%%%%%%%%%
   \draw[thick, color=\clr, fill=red] (-1,1.2)--(4,1.2)--(4,2.2)--(-1,2.2)--(-1,1.2);
    \node at (1.5,1.7) {$\scriptstyle D$};
    %%%%%%%%%
    %%%%%%%%%
 \draw[thick, color=\clr,<-] (0,3.4+0.2)--(0,3.4+0);
 \draw[thick, color=\clr] (0,3.4+0) .. controls ++(0,-0.5) and ++(0,0.5) .. (-0.8,3.4+-1)--(-0.8,3.4+-1.2);
  \draw[thick, color=\clr] (0,3.4+0) .. controls ++(0,-0.5) and ++(0,0.5) .. (-0.4,3.4+-1)--(-0.4,3.4+-1.2);
       \draw[thick, color=\clr] (0,3.4+0) .. controls ++(0,-0.5) and ++(0,0.5) .. (0.8,3.4+-1)--(0.8,3.4+-1.2);
                  \node[left] at (0.6,3.4+-1) {$\scriptstyle \cdots $};
            \node[above] at (0,3.4+0.2) {$\scriptstyle a_1 $};
  %%%%%%%%%%
  \node at (1.5,2.8) {$\scriptstyle \cdots $};
   \node at (1.5,3.8) {$\scriptstyle \cdots $};
  %%%%%%%%%%
 \draw[thick, color=\clr,<-] (3+0,3.4+0.2)--(3+0,3.4+0);
 \draw[thick, color=\clr] (3+0,3.4+0) .. controls ++(0,-0.5) and ++(0,0.5) .. (3+-0.8,3.4+-1)--(3+-0.8,3.4+-1.2);
  \draw[thick, color=\clr] (3+0,3.4+0) .. controls ++(0,-0.5) and ++(0,0.5) .. (3+-0.4,3.4+-1)--(3+-0.4,3.4+-1.2);
       \draw[thick, color=\clr] (3+0,3.4+0) .. controls ++(0,-0.5) and ++(0,0.5) .. (3+0.8,3.4+-1)--(3+0.8,3.4+-1.2);
                  \node[left] at (3+0.6,3.4+-1) {$\scriptstyle \cdots $};
            \node[above] at (3+0,3.4+0.2) {$\scriptstyle a_1 $};
  %%%%%%%%%%
  \end{tikzpicture}},
\end{align*}
where $D$ is a marked Brauer diagram as defined in \cite[Section 2.2]{KT}.  Using the relations discussed in \emph{loc.\  cit.} one can replace $D$ with $\sigma \circ E \circ \tau'$, where $\sigma$ and $\tau'$ consist of only upward-oriented crossings, and where $E$ is a diagram of the form

\begin{align*}
{}
\hackcenter{
\begin{tikzpicture}[scale=0.8]
 \draw (0,0) arc (0:-180:0.3cm) [thick, color=\clr];
   \draw[thick, color=\clr,<-] (-0.6,0.2)--(-0.6,0);
     \draw[thick, color=\clr,<-] (0,0.2)--(0,0);
         \node[above] at (0,0.2) {$\scriptstyle 1$};
      \node[above] at (-0.6,0.2) {$\scriptstyle 1$};
  \node at (-0.3,-0.3) {$\scriptstyle\blacklozenge$};
  %%%%%%%
    \node at (0.7,0.4) {$\scriptstyle \cdots $};
       \node at (0.7,0) {$\scriptstyle \cdots $};
  %%%%%%%
   \draw (2+0,0) arc (0:-180:0.3cm) [thick, color=\clr];
   \draw[thick, color=\clr,<-] (2+-0.6,0.2)--(2+-0.6,0);
     \draw[thick, color=\clr,<-] (2+0,0.2)--(2+0,0);
         \node[above] at (2+0,0.2) {$\scriptstyle 1$};
      \node[above] at (2+-0.6,0.2) {$\scriptstyle 1$};
  \node at (2+-0.3,-0.3) {$\scriptstyle\blacklozenge$};
  %%%%%%%%
  %%%%%%%%
   \draw (0,-1.5+0) arc (0:180:0.3cm) [thick, color=\clr];
   \draw[thick, color=\clr,>-] (-0.6,-1.5+-0.2)--(-0.6,-1.5+0);
     \draw[thick, color=\clr,>-] (0,-1.5+-0.2)--(0,-1.5+0);
         \node[below] at (0,-1.5+-0.2) {$\scriptstyle 1$};
      \node[below] at (-0.6,-1.5+-0.2) {$\scriptstyle 1$};
  \node at (-0.3,-1.5+0.3) {$\scriptstyle\blacklozenge$};
  %%%%%%%
    \node at (0.7,-1.5+-0.4) {$\scriptstyle \cdots $};
       \node at (0.7,-1.5+0) {$\scriptstyle \cdots $};
  %%%%%%%
   \draw (2+0,-1.5+0) arc (0:180:0.3cm) [thick, color=\clr];
   \draw[thick, color=\clr,>-] (2+-0.6,-1.5+-0.2)--(2+-0.6,-1.5+0);
     \draw[thick, color=\clr,>-] (2+0,-1.5+-0.2)--(2+0,-1.5+0);
         \node[below] at (2+0,-1.5+-0.2) {$\scriptstyle 1$};
      \node[below] at (2+-0.6,-1.5+-0.2) {$\scriptstyle 1$};
  \node at (2+-0.3,-1.5+0.3) {$\scriptstyle\blacklozenge$};
  %%%%%%%%
   \draw[thick, color=\clr,->] (2.6, -1.7)--(2.6, 0.2);
       \node[below] at (2.6,-1.7) {$\scriptstyle 1$};
        \node[above] at (2.6,0.2) {$\scriptstyle 1$};
       \node at (3.1,-1.9) {$\scriptstyle \cdots $};
        \node at (3.1,0.4) {$\scriptstyle \cdots $};
         \node at (3.1,-0.8) {$\scriptstyle \cdots $};
          \draw[thick, color=\clr,->] (1+2.6, -1.7)--(1+2.6, 0.2);
       \node[below] at (1+2.6,-1.7) {$\scriptstyle 1$};
         \node[above] at (1+2.6,0.2) {$\scriptstyle 1$};
\end{tikzpicture}}.
\end{align*}

By \cref{MixToUp}, there is an isomorphism of superspaces

\[
\Hom_{\pWeb_{\uparrow \downarrow}}\left(\uparrow^{\bbb' }, \uparrow^{\bba} \right) \cong \Hom_{\pWeb_{\uparrow \downarrow}}\left(\emptyset, \uparrow^{\bba}\downarrow^{\bbb} \right) 
\] given by:
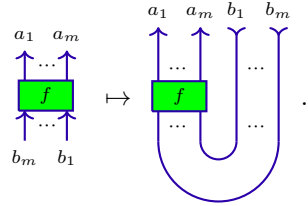
\begin{align*}
\hackcenter{}
\hackcenter{
\begin{tikzpicture}[scale=0.8]
\draw[thick, color=\clr, fill = green] (0,0)--(0,0.5)--(0.9,0.5)--(0.9,0)--(0,0);
 \draw[thick, color=\clr,->] (0.1,0.5)--(0.1,1);
  \draw[thick, color=\clr,->] (0.8,0.5)--(0.8,1);
      \draw[thick, color=\clr,->] (0.1,-0.5)--(0.1,0);
  \draw[thick, color=\clr,->] (0.8,-0.5)--(0.8,0);
      \node at (0.45,0.25) {$\scriptstyle f$};
         \node at (0.5,0.75) {$\scriptstyle \cdots$};
               \node at (0.5,-0.25) {$\scriptstyle \cdots$};
                \node[above] at (0.1,1) {$\scriptstyle a_1$};
                \node[above] at (0.8,1) {$\scriptstyle a_m$};
                 \node[below] at (0.1,-0.5) {$\scriptstyle b_m$};
                \node[below] at (0.8,-0.5) {$\scriptstyle b_1$};
\end{tikzpicture}}
\;
\mapsto
\hackcenter{
\begin{tikzpicture}[scale=0.8]
 \draw[thick, color=\clr,-<] (0.1,-0.5)--(0.1,1.4);
  \draw[thick, color=\clr,-<] (0.8,-0.5)--(0.8,1.4);
          \draw[thick, color=\clr,->] (-1.2,-0.5)--(-1.2,1.4);
           \draw[thick, color=\clr,->] (-0.5,-0.5)--(-0.5,1.4);
               \node at (0.45,-0.25) {$\scriptstyle \cdots$};
               \node at (-0.85,-0.25) {$\scriptstyle \cdots$};
                   \node at (0.45,0.7) {$\scriptstyle \cdots$};
               \node at (-0.85,0.7) {$\scriptstyle \cdots$};
                  \draw (0.1,-0.5) arc (0:-180:0.3cm) [thick, color=\clr];
                    \draw (0.8,-0.5) arc (0:-180:1cm) [thick, color=\clr];
                      \node[above] at (0.1,1.4) {$\scriptstyle b_1$};
                \node[above] at (0.8,1.4) {$\scriptstyle b_m$};
                 \node[above] at (-0.5,1.4) {$\scriptstyle a_m$};
                \node[above] at (-1.2,1.4) {$\scriptstyle a_1$};
                \draw[thick, color=\clr, fill=green] (-1.3,0)--(-1.3,0.5)--(-0.4,0.5)--(-0.4,0)--(-1.3,0);
                      \node at (-0.85,0.25) {$\scriptstyle f$};
\end{tikzpicture}}.
\end{align*}
Applying this map to the elements of our spanning set and using the relations in \cref{OrBraidThm} to pull \(\tau'\) and the upward-oriented caps to the right side of the diagram, it follows that $\Hom_{\pWeb_{\uparrow\downarrow}}\left(\emptyset, \uparrow^{\bba}\downarrow^{\bbb} \right)$ is spanned by diagrams of the form (\ref{0abdiag}).
\end{proof}

\section{Category Equivalences}\label{S:CategoryEquivalences}

\subsection{Faithfulness}\label{SS:faithfullness1}  Throughout this section $\k$ is an algebraically closed field of characteristic zero.  The assumption that $\k$ is algebraically closed allows us to cite the results from \cite{Coulembier1} used below.  We expect it is not necessary.  We also remark that Coulembier describes how to construct the morphism $f_{n+1}$.  It would be interesting to describe it explicitly in terms of diagrams.  

Let $\mathcal{I}_{n}$ denote the kernel of the functor 
\[
F: \mathcal{B} \to \fp (n)\text{-modules}.
\]  That is, $\mathcal{I}_{n}$ is the tensor ideal given by 
\[
\mathcal{I}_{n}\left([a],[b] \right) = \left\{ g \in \Hom_{ \mathcal{B}}\left([a],[b] \right) \mid F(g) =0 \right\}
\] for all objects $[a]$ and $[b]$ in $\mathcal{B}$.  The following is a reformulation of \cite[Theorem 8.3.1]{Coulembier1} so that it applies to the category $\mathcal{B}$.

\begin{theorem}\label{T:KernelofF} Let $n \geq 1$ and set $\ell = (n+1)(n+2)/2$.  Then, $\mathcal{I}_{n}$ is generated as a tensor ideal by a single morphism which lies in $\End_{\mathcal{B}}\left( [\ell]\right)$.
\end{theorem}

\begin{proof} By \cite[Theorem 8.3.1]{Coulembier1} there is a morphism $f_{n+1} \in  \End_{\mathcal{B}}\left( [\ell] \right)$ which is in the kernel of the functor $F$.  Since $F$ is a functor of $\k$-linear monoidal categories, the tensor ideal generated by $f_{n+1}$ is also contained in the kernel of $F$. 

 On the other hand, let $g \in  \Hom_{\mathcal{B}}\left( [a], [b] \right)$ be a nonzero element in the kernel of $F$.  By pre- and post-composing with cup and cap diagrams much as in \cref{SS:IsomorphicHoms}, one can define an isomorphism of superspaces 
\begin{equation}\label{E:YetAnotherIsom}
\Hom_{\mathcal{B}}\left( [a], [b] \right) \xrightarrow{\cong} \Hom_{\mathcal{B}}\left( \left[r \right], \left[r \right] \right),
\end{equation} where $r=(a+b)/2$.
Let $g' \in \Hom_{\mathcal{B}}\left( \left[r \right], \left[r \right] \right)$ be the image of $g$ under this isomorphism.  Since $F$ is a monoidal functor, $g'$ is in the kernel of the map $F:\End_{\mathcal{B}}\left( \left[r \right] \right) \to \Hom_{\fp (n)}\left(V^{\otimes r} \right)$.  By   \cite[Theorem 8.3.1]{Coulembier1},  $r \geq \ell$ and the kernel of this superalgebra homomorphism is generated as an ideal by $f_{n+1} \otimes \Id_{\bullet}^{\otimes \left(r-\ell \right)}$. Thus $g' = \sum_{i} a_{i}\left(f_{n+1} \otimes \Id_{\bullet}^{\otimes \left(r-\ell \right)} \right)b_{i}$ for some $a_{i},b_{i} \in \End_{\mathcal{B}}\left( \left[r \right] \right)$.  Applying the inverse of \cref{E:YetAnotherIsom} to this expression shows that $g$ lies in the tensor ideal generated by $f_{n+1}$, proving the claim.
\end{proof}

We abuse notation and write $f_{n+1}$ for the morphism $F'(f_{n+1})$ and $\iota_{\uparrow}(F'(f_{n+1}))$ in $\pWeb$ and $\pWeb_{\uparrow \downarrow}$, respectively.

\begin{theorem}\label{T:GKernels}  For any $n \geq 1$, the tensor ideal generated by the morphism $f_{n+1}$ in $\pWeb$ and $\pWeb_{\uparrow \downarrow}$ is the kernel of the functors $G$ and $G_{\uparrow \downarrow}$, respectively.
\end{theorem}

\begin{proof} This follow directly from the previous result and standard explosion/contraction arguments.  To explain, consider $\pWeb$ and the commutative diagram from \cref{comdiag}.  A simple diagram chase along with the fact that $F'$ is an isomorphism and $f_{n+1}$ is in the kernel of $F$ shows this morphism and, hence, the tensor ideal it generates, lies in the kernel of $G$.  

On the other hand, let $\bba$ and $\bbb$ be objects of $\pWeb$ and let $f \in \Hom_{\pWeb}(\bba ,\bbb)$ satisfy $G(f)=0$.  Then $G(\exp_{\bba , \bbb}(f))=0$.  Since $F'$ is an isomorphism, there is a $g \in \Hom_{\mathcal{B}}([|\bba |], [|\bbb |])$ such that $F'(g) = \exp_{\bba ,\bbb}(f)$.  By the commutativity of the right hand triangle, $F(g)=0$ and hence, by \cref{T:KernelofF}, $g$ lies in the tensor ideal of $\mathcal{B}$ generated by $f_{n+1}$. Applying the isomorphism $F'$, this implies $\exp_{\bba , \bbb }(f)$ lies in the tensor ideal of $\pWeb$ generated by $f_{n+1}$.  Since the map $\con_{\ob{a}, \ob{b}}$ is given by pre- and post-composing with morphisms, tensor ideals are preserved.   Therefore, $\con_{\bba , \bbb }(\exp_{\bba , \bbb }(f)) = \bba !\bbb !f$ lies in the tensor ideal of $\pWeb$ generated by $f_{n+1}$.  Hence, so does $f$.  This proves the claim regarding the functor $G$.

Diagram chase arguments using the commutative diagram in \cref{Orcomdiag} along with the fact that relevant the maps are given by pre- and post- composing with morphisms (and, hence, preserve tensor ideals) proves the statement for $G_{\uparrow \downarrow}$.
\end{proof}

We end this section by pointing out that under the assumptions of the present section the injectivity statements of \cref{HomIsomPlus,OrHomIsomPlus} can be made sharp.

\begin{proposition}\label{P:injectivity}  Let $n \geq 1$ and set $\ell = (n+1)(n+2)/2$.  Then, the maps 
\begin{align*}
G &:\Hom_{\pWeb}(\bba,\bbb)
\to
\Hom_{\fp(n)}(S^{\bba}(V_{n}), S^{\bbb}(V_{n})),\\
G_{\uparrow \downarrow} &:\Hom_{\pWeb_{\uparrow \downarrow}}(|_\bba,|_\bbb)
\to
\Hom_{\fp(n)}(S^{\bba}(V_{n}), S^{\bbb}(V_{n})), 
\end{align*}
are injective if and only if $ |\bba | + |\bbb |  < (n+1)(n+2)$.
\end{proposition}

\begin{proof}  Doing a diagram chase using \cref{comdiag,Orcomdiag}, one can show the injectivity of these maps is equivalent to the injectivity of the map 
\[
F: \End_{\mathcal{B}}\left([r], [r] \right) \to \End _{\fp (n)}\left(V^{\otimes r}, V^{\otimes r} \right),
\] where $r = \left(|\bba | + |\bbb | \right)/2$.  However, by \cite[Theorem 8.3.1]{Coulembier1} this map is injective if and only if $r < \ell$. This proves the claim.
\end{proof}

\subsection{Category Equivalences}\label{SS:categoryequivalences}

Let 
\[
\mathcal{B}(n),\;\;  \pnWeb,\;\; \text{and} \;\; \pnWeb_{\uparrow \downarrow}
\]  be the quotient of $\mathcal{B}$, $\pWeb$, and $\pWeb_{\uparrow \downarrow}$, respectively, by the tensor ideal generated by $f_{n+1}$,  where $f_{n+1}$ is the morphism given in \cref{SS:faithfullness1}.

%with the addition of the relation 
%\[
%f_{n+1}=0,
%\] where $f_{n+1}$ is the morphism given in \cref{T:KernelofF,T:GKernels}.

 As $f_{n+1}$ is in the kernel of the functors $F$, $G$, and $G_{\uparrow \downarrow}$, they induce well-defined functors which we call by the same name:
\begin{align}\label{E:Equivalences}
\notag F &: \mathcal{B}(n) \to \fp(n)\textup{-Mod}_{V}, \\
G &: \pnWeb \to \pmodS, \\
\notag G_{\uparrow \downarrow} &: \pnWeb_{\uparrow \downarrow} \to \pmodSS.   
\end{align}

\begin{theorem}\label{T:CategoryEquivalences} The functors given in \cref{E:Equivalences} are equivalences of monoidal supercategories.

\end{theorem}

\begin{proof}  Taken collectively, the previous theorems show the functors $F$, $G$, and $G_{\uparrow\downarrow}$ are essentially surjective, full, and faithful, proving the claim.
\end{proof}

\bibliographystyle{eprintamsplain}
\bibliography{Biblio}

\end{document}